\documentclass[11pt,reqno,twoside]{article}
\usepackage{mysty}

\usepackage{mathtools}

\usepackage{mathptmx}
\usepackage[T1]{fontenc}

\makeatletter
\let\reftagform@=\tagform@
\def\tagform@#1{\maketag@@@{(\ignorespaces\textcolor{black}{#1}\unskip\@@italiccorr)}}
\newcommand{\iref}[1]{\textup{\reftagform@{\tcr{\ref{#1}}}}}
\makeatother

\makeatletter
\newcommand{\pushright}[1]{\ifmeasuring@#1\else\omit\hfill$\displaystyle#1$\fi\ignorespaces}
\newcommand{\pushleft}[1]{\ifmeasuring@#1\else\omit$\displaystyle#1$\hfill\fi\ignorespaces}
\makeatother

\usepackage{titlesec}
\titlespacing*{\section}{0pt}{1.25\baselineskip}{0.25\baselineskip}
\titlespacing*{\subsection}{0pt}{0.75\baselineskip}{0.125\baselineskip}
\titlespacing*{\subsubsection}{0pt}{0.5\baselineskip}{0.125\baselineskip}
\titlespacing*{\paragraph}{0pt}{0.25\baselineskip}{0.25\baselineskip}

\makeatletter
\g@addto@macro\normalsize{%
  \setlength\abovedisplayskip{5pt}
  \setlength\belowdisplayskip{5pt}
  \setlength\abovedisplayshortskip{5pt}
  \setlength\belowdisplayshortskip{5pt}
}
\makeatother

\begin{document}
	

\title{\vspace{-3mm}Geometry of First-Order Methods and Adaptive Acceleration\vspace{-3mm}}
\author{
		Clarice Poon\thanks{Department of Mathematics, University of Bath, Bath UK. E-mail: cmhsp20@bath.ac.uk.}\and
		Jingwei Liang\thanks{DAMTP, University of Cambridge, Cambridge UK. E-mail: jl993@cam.ac.uk.}
		}
\date{}
\maketitle


\begin{center}
\begin{small}
\begin{minipage}[c]{0.925\textwidth}
{

\vspace{-4mm}

\begin{center}
{\bf Abstract}
\end{center}

\vspace{-4mm}

First-order operator splitting methods are ubiquitous among many fields through science and engineering, such as inverse problems, signal/image processing, statistics, data science and machine learning, to name a few. 
In this paper, we study a geometric property of first-order methods when applying to solve non-smooth optimization problems. With the tool of ``partial smoothness'', we design a framework to analyze the trajectory of the fixed-point sequence generated by first-order methods and show that locally, the fixed-point sequence settles onto a regular trajectory such as a straight line or a spiral. 
Based on this finding, we discuss the limitation of current widely used ``inertial acceleration'' technique, and propose a trajectory following adaptive acceleration algorithm. 
Global convergence is established for the proposed acceleration scheme based on the perturbation of fixed-point iteration. Locally, we first build connections between the acceleration scheme and the well-studied ``vector extrapolation technique'' in the field of numerical analysis, and then discuss local acceleration guarantees of the proposed acceleration scheme. 
Moreover, our result provides a geometric interpretation of these vector extrapolation techniques. 
Numerical experiments on various first-order methods are provided to demonstrate the advantage of the proposed adaptive acceleration scheme. 

}
\end{minipage}
\end{small}
\end{center}


\vspace{-4mm}

%


%


\section{Introduction}\label{sec:intro}


Non-smooth optimization is an active research area of modern optimization, which aims to find solutions of structured problems that are the sum of smooth and non-smooth functions, possibly under constraints and composition with (non)linear operators. It plays a fundamental role in various fields through science and engineering, such as inverse problems, signal/image processing, compressed sensing, statistics, data science and machine learning, etc. 
In the literature, numerical schemes, typically first-order (operator/proximal splitting) methods, have been designed to solve non-smooth optimization problems. 
Over the past decades, driven by the real-world problems arising from the aforementioned fields, non-smooth optimization and first-order methods have experienced tremendous growth and success, especially in large-scale problems.  
However, despite the huge success, first-order methods suffer a significant drawback: slow speed of convergence, which has made them the bottleneck of solving today's even larger-scale problems. With the increasing size of data sets and growing complexity of mathematical models of real-world problems, the need for novel fast and low computational cost algorithms is becoming increasingly strong.

In this paper, we denote $\fom$ the class of first-order methods and $\calF \in \fom$ a first-order method, for instance the proximal gradient descent; See Example \ref{eg:gd}. The iteration of $\calF$ usually can be (re)formulated as a fixed-point iteration in a real Hilbert space $\calH$
\beq\label{eq:fom}
\zkp = \calF (\zk)  ,
\eeq
where $\seq{\zk} \subset \calH$ is the fixed-point iterates that converges to $\zsol \in \fix(\calF)$ with $\fix(\calF)\eqdef\ba{z\in\calH: z = \calF(z)}$ which is assumed to be non-empty. 
We refer to \cite{bauschke2011convex} for more detailed accountant of fixed-point theory.

\subsection{Acceleration of first-order methods}

In the literature, numerous approaches are proposed to accelerate first-order methods, among them, the ``inertial technique'' and over-relaxation are probably the most widely used. 
Both these approaches belong to the realm of {\it extrapolation techniques}.  
Let $\calF$ be the first-order method in \eqref{eq:fom}, a general combination of extrapolation and $\calF$ takes the following form
\beq\label{eq:efom}
\begin{aligned}
	\zbark &= \calE(\zbarkm, \zk, \zkm, ...) , \\
	\zkp &= \calF(\zbark, \zk, \zkm, ...)  ,
\end{aligned}
\eeq
where $\calE$ is the extrapolation step that computes the point $\zbark$ based on $\zbarkm$ and the history of $\zk$ including $\ba{\zk,\zkm,..}$. 
In what follows we present a brief overview of the inertial technique and over-relaxation.

\subsubsection{Inertial acceleration} 

The very first inertial scheme is the ``heavy-ball method'' \cite{polyak1964some} proposed by Polyak which can significantly speed-up the performance of gradient descent, particularly when the problem is strongly convex and twice differentiable. 
The theoretical foundation of inertial acceleration is due to Nesterov, in \cite{nesterov83} he showed that a different combination of inertial and gradient descent can improve the $O(1/k)$ convergence rate of objective function value to $O(1/k^2)$. This result was further extended to the non-smooth case by Beck and Teboulle in \cite{fista2009} where they proposed the ``fast iterative shrinkage-thresholding algorithm'', a.k.a FISTA for speeding up Forward--Backward splitting method \cite{lions1979splitting} (\ie the proximal gradient descent). 
Note that, gradient descent and its proximal version are descent methods, that is the objective function value along the iteration is monotonically non-increasing along iteration\footnote{Descent methods include gradient descent and its proximal version (a.k.a. Forward--Backward splitting) and proximal point algorithm, note that the problem does not necessarily have to be convex \cite{attouch2013convergence}  for the method to be descent. Other first-order methods, such as Douglas--Rachford/ADMM and Primal--Dual splitting methods, are non-descent in general.}.

Over the years, the huge success of the accelerated (proximal) gradient descent schemes has motivated people to extend the inertial acceleration to other first-order methods.
Let $\calF$ be the first-order method in \eqref{eq:fom}, a generic inertial version of $\calF$ would read 
\beq\label{eq:ifom}
\begin{aligned}
	\zbark &= \zk + \ak (\zk - \zkm) , \\ 
	\zkp &= \calF(\zbark, \zk)  . 
\end{aligned}
\eeq
The inertial scheme first extrapolate a point $\zbark$ along the direction of $\zk-\zkm$, and then update the next $\zkp$ based on $\zbark$ with or without $\zk$.  
The formulation \eqref{eq:ifom} abstracts many existing inertial schemes in the literature, below is an example of gradient descent.

\begin{example}[Gradient descent]\label{eg:gd}
Consider an unconstrained smooth minimization problem,
$\min_{x\in\bbR^n} F(x)$ 
where $F: \bbR^n \to \bbR$ is proper convex differentiable with gradient $\nabla F$ being $L$-Lipschitz continuous. The iteration of gradient descent reads (for this case we use $\xk$ instead of $\zk$)
\[
\xkp = \xk - \gamma \nabla F(\xk)  ,
\]
where $\gamma \in ]0, 2/L[$ is the step-size. 
The fixed-point operator of gradient descent reads $\calF \eqdef \Id - \gamma \nabla F$. The ``heavy-ball method'' \cite{polyak1964some} takes the following form of iteration
\beq\label{eq:heavy-ball}
\begin{aligned}
	\xbark 
		&= \xk + \ak(\xk - \xkm) , \\
	\xkp 
		&= \xbark - \gamma \nabla F(\xk)   ,
\end{aligned}
\eeq
where $\ak \in [0,1]$ is the inertial parameter. 
If we further replace $\nabla F(\xk)$ with $\nabla F(\xbark)$, and compute $\ak$ via
$t_{k} = \frac{ 1 + \sqrt{1 + 4t_{k-1}^2 }}{2} ,~\ak = \frac{t_{k-1}-1}{t_{k}} $ with $t_0 = 1$, 
then \eqref{eq:heavy-ball} becomes the scheme of \cite{nesterov83} which achieves $O(1/k^2)$ convergence rate for $F(\xk) - F(\xsol)$ where $\xsol$ is a global minimizer of $F$. 
\end{example}

Other examples of inertial first-order methods include: the inertial versions of Proximal Point Algorithm \cite{alvarez2001inertial,alvarez2000minimizing}, Forward--Backward splitting \cite{moudafi2003convergence,lorenz2015inertial,liang2016thesis}, Douglas--Rachford splitting \cite{boct2015inertialDR} and inertial Primal--Dual splitting \cite{boct2016inertial}, or in general the inertial version of \KM fixed-point iteration \cite{mainge2008convergence,dong2018modified} which covers many of the inertial first-order methods as special cases. 
However, despite its overwhelming popularity, the combination of inertial technique and first-order methods suffers several drawbacks: 
\begin{itemize}[leftmargin=2em]
	\item {\bf Restricted parameter choices} Unlike the elegant inertial (proximal) gradient descent methods \cite{polyak1964some,nesterov83,fista2009}, the choices of (inertial) parameters for general inertial first-order methods, \eg \cite{boct2015inertialDR,lorenz2015inertial,mainge2008convergence,dong2018modified}, are quite restricted and complicated. 
	
	\item {\bf Complicated convergence proof} For inertial (proximal) gradient descent methods \cite{polyak1964some,nesterov83,fista2009}, a Lyapunov stability function can be found easily, even in the non-convex case, as (proximal) gradient descent is a descent method. Things become much more complicated for other first-order methods as they are non-decent, and Lyapunov functions can only be obtained under stronger assumptions or do not exist at all. As a result, the convergence proof becomes more complicated, which is also another reason of restricted parameter choices. 
	
	\item {\bf Lack of acceleration guarantees} For non-descent methods, there are limited acceleration guarantees, unless stronger assumptions, such as smoothness or strong convexity, are imposed. Examples of inertial schemes failing to provide acceleration can be easily found when no stronger assumptions are available; See Section \ref{sec:failure} for examples, and also \cite{PoonLiang2019b} and \cite[Chapter 4.5]{liang2016thesis}. 
\end{itemize}
Finally, it is worth mentioning that, in the literature, most inertial schemes consider only the momentum created by two past points, namely $\zk - \zkm$. 
For certain cases, use the momentum of more than two points could be beneficial, see Section \ref{sec:failure} for example. 
This is mentioned in \cite{polyak1964some}, and related work can be found in \cite{liang2016thesis,dong2019mikm}.

\subsubsection{Over-relaxation} 
In the field of fixed-point theory, another popular approach to accelerate convergence is the \emph{over-relaxation} which is the generalization of the successive over relaxation for linear systems. For the fixed-point iteration \eqref{eq:fom}, the relaxation of it reads
\beq\label{eq:rfom}
\zkp = \zk + \lambda_{k} \Pa{ \calF(\zk) - \zk} ,
\eeq
where $\lambda_{k} \in ]0, \bar{\lambda}]$ is the relaxation parameter and $\bar{\lambda}$ is the upper bound of $\lambda$ determined by the property of $\calF$. For example $\bar{\lambda} = \frac{1}{\alpha}$ when $\calF$ is so-called $\alpha$-averaged non-expansive; see Definition \ref{dfn:averaged-operator} and \cite{bauschke2011convex} for more detailed discussions.  
When $\bar{\lambda} > 1$ and $\lambda_{k} \in ]1, \bar{\lambda}]$, \eqref{eq:rfom} is the over-relaxed version of $\calF$.

Below we briefly show that over-relaxed $\calF$ is equivalent to an inertial version of \eqref{eq:fom} which is a special case of \eqref{eq:ifom}. Denote $\ak = \lambda_k-1$, then we can rewrite \eqref{eq:rfom} as\footnote{Strictly speaking, \eqref{eq:rfom} is equivalent to \beqn
\begin{aligned}
\zk &= \zbark + \ak \pa{ \zbark - \zkm} , \\
\zbarkp &= \calF(\zk) .
\end{aligned}
\eeqn We switch $\zk$ and $\zbark$ in order to comply with \eqref{eq:efom}.} 
\beq\label{eq:rfom_2}
\begin{aligned}
\zbark &= \zk + \ak \pa{ \zk - \zbarkm} , \\
\zkp &= \calF(\zbark) .
\end{aligned}
\eeq
Instead of extrapolating a point along the direction $\zk-\zkm$, relaxation uses $\zk-\zbarkm$. 
Over-relaxation can also significantly improved the convergence speed of \eqref{eq:fom}, such as the over-relaxed projection based algorithms for feasibility problems \cite{gubin1967method,bauschke2016optimal}. 
However, same as the inertial scheme \eqref{eq:ifom}, over-relaxation is not guarantee to provide acceleration. For instance, in \cite{bauschke2016optimal,liang2017localDR}, the authors showed that the optimal $\lambda_k$ for Douglas--Rachford splitting when applied to (locally) polyhedral problem is $1$, that is no relaxation provides the best performance. 

Generally speaking, over-relaxation suffers the same problem as inertial schemes, that its acceleration guarantees are method and problem dependent:
\begin{itemize}
\item For descent methods, \eg Forward--Backward splitting, owing to the result of \cite{liang2017activity}, it can be shown that locally over-relaxation can provide acceleration. 
\item While for other algorithms, such as Douglas--Rachford splitting, the performance of over-relaxation depends on the problem to solve and parameters of the algorithm \cite{bauschke2016optimal,liang2017localDR}.
\end{itemize}

\subsection{Main contributions}

In this paper, motivated by the behavior of inertial first-order methods and the over-relaxation scheme, we present a systematic study on the geometric properties of first-order methods and their acceleration. 
We first show that the performance of inertial and relaxation is determined by the trajectory of the generated fixed-point sequence $\seq{\zk}$. Namely, different trajectories result in different outcomes (see Section \ref{sec:trajectory-fom}). 
When considering non-smooth optimization, we present a unified framework for analyzing the local trajectory of the fixed-point sequence of first-order methods. Based on this finding, we propose a generic {\it trajectory following linear prediction scheme} for accelerating first-order methods. 
More precisely, our contributions in this paper are summarized below.

\paragraph{Geometry of \fom~via trajectory of fixed-point sequences} 
In the literature of first-order methods, numerous first-order operator splitting methods are proposed based on the structures of the optimization problems at hand. 
However, the study on the structure of first-order methods is rather limited, and this is mainly due to the non-linearity of the fixed-point iteration. 
In this paper, by focusing on non-smooth optimization, with the help of ``partial smoothness'' (Definition \ref{dfn:psf}), in Section \ref{sec:trajectory-fom} we propose a generic framework for analyzing the trajectory of $\seq{\zk}$ generated by the fixed-point iteration \eqref{eq:fom} (Section \ref{sec:framework}). 
More precisely, we utilize the fact that $\calF$ can be linearized locally around the solution along some $C^2$-smooth manifold(s), up to residuals. This means there exists a square matrix $M_{\calF}$ such that
\[
\zkp - \zk = M_{\calF} (\zk - \zkm) + o(\norm{\zk-\zkm}) .
\]
Based on the spectral properties of $M_{\calF}$, we show that different first-order methods admit different types of trajectories for the fixed-point sequence $\seq{\zk}$: 
\begin{itemize}[leftmargin=2em]
\item For (proximal) gradient descent, we show that the spectrum of $M_{\calF}$ is real, as a result the eventual trajectory of $\seq{\zk}$ is a straight line; See Section \ref{subsec:trajectory-fb}. 

\item For other popular first-order methods, such as Douglas--Rachford splitting and alternating direction method of multipliers (ADMM), based on the properties of the functions and parameters, we show that the leading eigenvalue of $M_{\calF}$ can be either real or complex, and the eventual trajectory of $\seq{\zk}$ could be either a straight line (real leading eigenvalue) or a spiral (complex leading eigenvalue); See Section \ref{subsec:trajectory-dr}. For Primal--Dual splitting methods, when the problem is locally polyhedral, the leading eigenvalue of $M$ is complex and the eventual trajectory of $\seq{\zk}$ is a spiral; see Section \ref{subsec:trajectory-pd}.  
\end{itemize}

\paragraph{Limitation of inertial and over-relaxation}
The trajectory of first-order methods allows us to analyze the limitations of inertial technique and over-relaxation. In Section \ref{sec:failure}, based on examples of Douglas--Rachford splitting method, we show that for inertial
\begin{itemize}
\item When the trajectory of $\seq{\zk}$ is a \emph{straight line}, then inertial can provide substantial acceleration. 

\item When the trajectory of $\seq{\zk}$ is a \emph{logarithmic spiral}, we show that inertial will always fail to provide acceleration, and one should not consider relaxation either. 

\item When the trajectory of $\seq{\zk}$ is an \emph{elliptical spiral}, inertial and over-relaxation can provide acceleration under proper implementation. 
\end{itemize}

\paragraph{An adaptive acceleration via linear prediction}

The limitation of inertial and over-relaxation techniques, particularly their failures, implies that the correct acceleration scheme should be able to adapt to the trajectory of the underlying sequence, which is another core contribution of this work.  
%
By exploiting the eventual regularity, \ie either straight line or spiral, of the trajectory of $\seq{\zk}$, in Section \ref{sec:lp} we propose an adaptive linear prediction scheme for accelerating first-order methods which is able to follow the trajectory of the fixed-point sequence. 
Global convergence based on perturbation of fixed-point iteration is provided. Local acceleration guarantees are also provided for the proposed adaptive scheme, based on the connections with existing vector extrapolation techniques.

Our proposed linear prediction scheme belongs to the realm of vector extrapolation techniques, while our derivation provides an alternative geometric interpretation for polynomial extrapolation methods such as minimal polynomial extrapolation (MPE) \cite{cabay1976polynomial} and reduced rank extrapolation (RRE) \cite{eddy1979extrapolating,mevsina1977convergence}. 
Our linear prediction bridges the gap between inertial schemes and polynomial extrapolation methods. 
Moreover, our geometric interpretation of linear prediction provides insights on how to enhance the robustness and performance of extrapolation methods.

\subsection{Related work}

Over the past decades, owing to the tremendous success of inertial acceleration \cite{nesterov83,fista2009}, the inertial technique has been widely adapted to accelerate other first-order algorithms. For example the inertial versions of Douglas--Rachford and alternating direction method of multipliers (ADMM) \cite{bot2014inertial,pejcic2016accelerated,kadkhodaie2015accelerated,francca2018admm}, Primal--Dual splitting \cite{chan2014inertial,liang2016thesis}. 
In terms of 
fixed-point iteration, the inertial versions of it are also studied in the literature  \cite{mainge2008convergence,dong2018general}. 
Multi-step inertial schemes, \ie using the momentum created by more than two past points, are also considered in the literature, see for instance \cite{dong2019mikm,liang2016thesis}. 
However, for most of these works, to ensure acceleration guarantees of inertial, stronger assumptions are needed, such as Lipschitz continuity or strong convexity, see \cite{francca2018admm} for ADMM. When it comes to general non-smooth problems, some of them would fail to provide acceleration. 
Moreover, as discussed in \cite[Chapter 4]{liang2016thesis}, for certain problems and algorithms, such as basis pursuit problem and Douglas--Rachford splitting method, only multi-step inertial scheme can provide acceleration.

For more generic acceleration techniques, there are extensive works in numerical analysis on the topic of convergence acceleration for sequences. Given an arbitrary sequence $\seq{\zk}\subset\RR^n$ with limit $\zsol$, the goal of convergence acceleration is to find a transformation  $\Ee_k : \{z_{k-j}\}_{j=1}^q\to  \zbark \in \RR^n$ such that $\zbark$ converges faster to $\zsol$. In general, the process by which $\{\zk\}$ is generated is unknown, $q$ is chosen to be a small integer, and $\zbark$ is referred to as the extrapolation of $\zk$. Some of the best known examples  include Richardson's extrapolation \cite{richardson1927viii}, the $\Delta^2$-process of Aitken \cite{aitken1927xxv} and Shank's algorithm \cite{shanks1955non}. 
Much of the work on the extrapolation of vector sequences was initiated by Wynn \cite{wynn1962acceleration} who generalized the work of Shank to vector sequences. 
We refer to the article \cite{brezinski2001convergence} and books \cite{brezinski2013extrapolation,sidi2003practical} for a detailed historical perspective on the development of these techniques.   
In Section \ref{sec:acc-gua}, the formulation two such schemes are provided: minimal polynomial extrapolation (MPE) \cite{cabay1976polynomial} and Reduced Rank Extrapolation (RRE) \cite{eddy1979extrapolating,mevsina1977convergence} (which is also a variant of Anderson acceleration developed independently in \cite{Anderson}), which are particularly relevant to this present work. 

More recently, there has been a series of work on  a regularized version of RRE \cite{scieur2016regularized,scieur2017nonlinear,bollapragada2018nonlinear}. As mentioned in \cite{sidi2008vector}, the stability of vector extrapolation techniques depend on the stability of computing the extrapolation coefficients. To address this instability, \cite{scieur2016regularized}  proposed to apply Tikhonov regularization when computing the extrapolation coefficients. {The regularization parameter in these works rely on a grid search based on objective function which is only doable for descent methods. Building on this idea, \cite{fu2019anderson} extends this idea of regularisation to the Douglas-Rachford splitting method by considering the relative primal-dual gap.  We however stress that the objectives and contributions of our work is different: we directly handle the non-smoothness of optimization problems by studying the eventual trajectories of the generated sequences.  This geometry that we uncover provides an understanding of when inertial techniques or vector extrapolation techniques can be applied. Finally, we mention that this work is a substantial extension of our conference paper \cite{PoonLiang2019b}, where we carried out the trajectory analysis for  Alternating Direction Method of Multipliers (ADMM) iterations.
}



\paragraph*{Paper organization}

Necessary notations and definitions are collected in Section \ref{sec:math_bg}. In Section \ref{sec:trajectory-fom}, we propose a generic framework for analyzing the trajectory of first-order methods, and the analysis of Forward--Backward, Douglas--Rachford and Primal--Dual splitting methods are discussed in detail. 
The limitations of the inertial technique and over-relaxation when applied to non-descent methods are discussed in Section \ref{sec:failure}. The trajectory motivated linear prediction acceleration scheme is described in Section \ref{sec:lp}, where global convergence is also provided. In Section \ref{sec:acc-gua}, by connecting linear prediction with existing polynomial extrapolation methods, local acceleration guarantees are provided. 
Numerical experiments are presented in Section \ref{sec:exp}. 
Trajectory of linear systems and proofs of main theorems are collected in the appendix.


\section{Mathematical background}
\label{sec:math_bg}

Throughout the paper, $\bbR^n$ is a $n$-dimensional Euclidean space equipped with scalar inner product $\iprod{\cdot}{\cdot}$ and associated norm $\norm{\cdot}$. $\Id$ denotes the identity operator on $\bbR^n$. $\lsc(\bbR^n)$ denotes the class of proper convex and lower semi-continuous functions on $\bbR^n$.

\subsection{Convex and set-valued analysis}

For a nonempty convex set $C \subset \bbR^n$, denote by $\Aff(C)$ its affine hull and by $\LinHull(C)$ the smallest subspace parallel to $\Aff(C)$. Denote $\iota_{C}$ the indicator function of $C$, $\normal{C}$ the associated normal cone operator and $\PT{C}$ the orthogonal projection operator on $C$. 
 
%

The sub-differential of a proper convex and lower semi-continuous function $R \in \lsc(\bbR^n)$ is the set-valued operator defined by $\partial R :  \bbR^n \setvalued \bbR^n ,~ x\mapsto \Ba{ g\in\bbR^n | R(x') \geq R(x) + \iprod{g}{x'-x} , \forall x' \in \bbR^n }$. 
%
Let $\gamma > 0$, the proximity operator or proximal mapping, of $R$ is defined by 
$\prox_{\gamma R}(\cdot) \eqdef \argmin_{x\in\RR^n}~ \gamma R(x)  + \sfrac{1}{2}\norm{x-\cdot}^2 $.
%
The Fenchel conjugate, or simply conjugate, of $R$ is defined by $R^*(v) \eqdef \sup_{x\in\bbR^n} ~ \pa{ \iprod{x}{v} - R(x) } $. 
A function $R: \bbR^n \to ]-\infty, +\infty]$ is polyhedral if its epigraph, $\mathrm{epi}(R) \eqdef \Ba{ (x, t): x\in\bbR^n, v \in \bbR, v \geq R(x)} \subseteq \bbR^{n+1}$, is a polyhedral set.


%

\vgap


\begin{definition}[Monotone operator]\label{def:mon-opt}
A set-valued mapping $A : \bbR^n \setvalued \bbR^n$ is said to be monotone if, given any $x_1,x_2 \in \bbR^n$ there holds
\beqn
\iprod{x_1-x_2}{v_1-v_2} \geq 0  ,~~ \forall v_1 \in A(x_1) ~~{\rm and}~~ v_2 \in A(x_2) . 
\eeqn%
It is maximal monotone if its $\gra (A) \eqdef \ba{(x, v) \in \bbR^{n}\times \bbR^{n} | v \in A(x) }$ can not be contained in the graph of any other monotone operators.
\end{definition}

For a maximal monotone operator $A$, $(\Id + A)^{-1}$ denotes its resolvent. It is known that for function $R \in \lsc(\bbR^n)$, its sub-differential $\partial R$ is maximal monotone \cite{rockafellar1997convex}, and that $\prox_{R}=(\Id + \partial R)^{-1}$.

%
%

\begin{definition}[Non-expansive operator] \label{dfn:averaged-operator}
An operator $\calF: \bbR^n \rarrow \bbR^n$ is non-expansive if
\beqn
\norm{\calF (x) - \calF (y)} \leq \norm{x-y} ,~~~~ \forall x, y  \in  \bbR^n  .
\eeqn
That is, $\calF$ is $1$-Lipschitz continuous. 
For any $\alpha \in ]0,1[$, $\calF$ is called $\alpha$-averaged if there exists a non-expansive operator $\calF'$ such that $\calF = \alpha \calF' + (1-\alpha)\Id$.
\end{definition}

The fixed points of non-expansive operators in general are not available explicitly. To find them, one has to apply certain iterative procedures, one of the most-known is the \KM iteration \cite{krasnosel1955two,mann1953mean}. 

\begin{definition}[\KM iteration] \label{dfn:KMi}
	Let $\calF : \bbR^n  \rightarrow \bbR^n $ be a non-expansive operator such that $\fix(\calF ) \neq \emptyset$. Let $\lambda_{k} \in [0, 1]$ and choose $x_{0} \in \bbR^n$ arbitrarily, the \KM iteration of $ \calF $ reads
	\beqn
	\zkp 
	= \zk + \lambda_{k} \pa{ \calF  (\zk) - \zk } .
	\eeqn%
	Moreover, if $\lambda_{k} \in [0, 1]$ is such that $\sum_{k\in\bbN} \lambda_{k}(1-\lambda_{k}) = \pinf$, then $\seq{\zk}$ converges to a point in $\fix(\calF )$ \cite{bauschke2011convex}. 
\end{definition}

When $\calF$ is $\alpha$-averaged, the upper bound of $\lambda_k$ becomes $\frac{1}{\alpha}$, and the condition needed for convergence of \KM iteration changes to $\sum_{k\in\bbN} \lambda_{k}(\frac{1}{\alpha}-\lambda_{k}) = \pinf$.

%

\subsection{Angle between subspaces}

Let $T_1, T_2$ be two subspaces, and without the loss of generality, assume $1\leq p\eqdef\dim(T_1) \leq q\eqdef\dim(T_2) \leq n-1$.

\begin{definition}[Principal angles] 
The principal angles $\theta_k \in [0,\frac{\pi}{2}]$, $k=1,\ldots,p$ between subspaces $T_1$ and $T_2$ are defined by, with $u_0 = v_0 \eqdef 0$, and
\begin{align*}
\cos(\theta_k) \eqdef \iprod{u_k}{v_k} = \max \iprod{u}{v} ~~~\mathrm{s.t.}~~~	
& u \in T_1, v \in T_2, \norm{u}=1, \norm{v}=1, ~~
\iprod{u}{u_i}=\iprod{v}{v_i}=0, ~ i=0,\dotsm,k-1 .
\end{align*}
The principal angles $\theta_k$ are unique and satisfy $0 \leq \theta_1 \leq \theta_2 \leq \dotsm \leq \theta_p \leq \pi/2$.
\end{definition}

\begin{definition}[Friedrichs angle]\label{def:friedrichs-angle}
The Friedrichs angle $\theta_{F} \in ]0,\frac{\pi}{2}]$ between $T_1$ and $T_2$ is
\begin{equation*}
\cos\pa{ \theta_F } \eqdef \max \iprod{u}{v} ~~~\mathrm{s.t.}~~~
u \in T_1 \cap (T_1 \cap T_2)^\perp, \norm{u}=1 ,~
v \in T_2 \cap (T_1 \cap T_2)^\perp, \norm{v}=1 .
\end{equation*}
\end{definition}



\begin{lemma}[{\cite{Bauschke14}}]
\label{lem:fapa}
The Friedrichs angle is exactly $\theta_{d+1}$ where $d \eqdef \dim(T_1 \cap T_2)$. Moreover, $\theta_{F} > 0$.
\end{lemma}

%

\subsection{Partial smoothness}



Let $\calM$ be a $C^2$-smooth Riemannian manifold, denote $\tanSp{\calM}{x}$ the tangent space to $\calM$ at any point $x$ in $\calM$. 
The definition below of partial smoothness is adapted from \cite{LewisPartlySmooth} to the case of $\lsc(\bbR^n)$ functions.

\begin{definition}[{Partly smooth function \cite{LewisPartlySmooth}}]\label{dfn:psf}
A function $R \in \lsc(\bbR^n)$ is partly smooth at $\xbar$ relative to a set $\calM_{\xbar}$ if $\calM_{\xbar}$ is a $C^2$ manifold around $\xbar$, and:
%
\begin{itemize}[label={}, leftmargin=2.5cm, itemsep=1pt]
\item[\bf Smoothness] \label{PS:C2}
$R$ restricted to $\calM_{\xbar}$ is $C^2$-smooth around $\xbar$.
%
%
\item[\bf Sharpness] \label{PS:Sharp}
The tangent space $\tanSp{\calM_{\xbar}}{\xbar} = \LinHull\pa{\partial R(\xbar)}^\bot$.
\item[\bf Continuity] \label{PS:DiffCont}
The set-valued mapping $\partial R$ is continuous at $x$ relative to $\calM_{\xbar}$.
\end{itemize}
\end{definition}


{\noindent}Loosely speaking, a partly smooth function behaves \emph{smoothly} along the smooth manifold $\calM_{\xbar}$, and \emph{sharply} transversal to $\calM_{\xbar}$. 
The class of partly smooth functions at $\xbar$ relative to $\calM_{\xbar}$ is denoted as $\PSF{\xbar}{\calM_{\xbar}}$.
We reference \cite[Chapter 5]{liang2016thesis} and the references therein for popular examples of partly smooth functions which include: indicator function of partly smooth set, $\ell_1$-norm, $\ell_{1,2}$-norm, $\ell_{\infty}$-norm, total variation and nuclear norm, etc. 
In the past few year, partial smoothness has proven to be a powerful tool for analyzing the local convergence behaviors of first-order methods \cite{liang2014local,liang2016thesis,liang2017activity,molinari2018convergence} when applied to non-smooth optimization.

\subsection{Sequence trajectory}\label{sec:pre-type-trajectory}

Let $\seq{\zk}$ be a sequence in $\bbR^{n}$ whose limiting point exists. 
Given $k$, define $\vk \eqdef \zk - \zkm$ the displacement vector. To characterize the trajectory of $\seq{\zk}$, we use the angle $\theta_k$ between 
$\vk, \vkm$ which is define by
\beq\label{eq:angle-thetak}
\theta_k 
\eqdef \angle(\vk, \vkm) 
= \arccos\Ppa{ \sfrac{\iprod{\vk}{\vkm}}{\norm{\vk}\norm{\vkm}} } .
\eeq
In this paper, we are interested in three different types of trajectories, which are summarized in Table \ref{tab:types} below: straight line and two types of spiral (logarithmic and elliptical).  
For these three types of trajectories, we have
\begin{enumerate}[label={(\Roman{*})},leftmargin=3em]
	\item For Type I trajectory, $\theta_k$ converges to $0$ which means eventually $\seq{\zk}$ lies in a \emph{straight line}. 
	
	\item For Type II trajectory, instead of converging to $0$, $\theta_k$ converge to some $\theta_{F} \in ]0, \pi/2[$ implying that the trajectory of $\seq{\zk}$ is a \emph{logarithmic spiral}. See the top view of the Type II trajectory above. 
	
	\item For Type III trajectory, different from the former two cases, $\theta_k$ eventually oscillate in an interval, which results in an \emph{elliptical spiral}. See the top view of the Type III trajectory above.
	
\end{enumerate}
Detailed discussion on these trajectories are presented in Section \ref{sec:trajectory-ls} of the appendix.

\begin{remark}[What determines the trajectory of $\seq{\zk}$]\label{rmk:what}
Suppose the sequence $\seq{\zk}$ above is generated by a linear system of the form $\zkp = M \zk$ where $M$ is a matrix whose spectral radius is strictly smaller than $1$\footnote{If the spectral radius of $M$ is equal to $1$, then as long as the power of $M$ converges, \ie there exists a matrix $\widetilde{M}$ such that $\widetilde{M} = \lim_{k\to\pinf} M^k$, then we can consider the leading eigenvalue of $M-\widetilde{M}$ instead of $M$.}. The type of trajectory of $\seq{\zk}$ is determined by the \emph{leading eigenvalue} of $M$ --- real leading eigenvalue leads to straight-line trajectory, and complex eigenvalue leads to spiral trajectory. 
For the type of spiral trajectory, it relies on the further properties of the leading eigenvalue; Section \ref{sec:trajectory-ls} of the appendix. 
\end{remark}

\renewcommand{\arraystretch}{1.125}

\begin{table}[!ht]
\begin{center}
\caption{Three types of trajectory of sequence.}
\label{tab:types}
\begin{tabular}{ |c|c|c| }
	\hline
	{Type I: straight line} & {Type II: logarithmic spiral} & {Type III: elliptical spiral} \\
	\hline
	{$\theta_k \to 0$} & {$\theta_k \to \theta_{F} \in ]0, \pi/2[$} & {$\theta_k \to [\utheta, \otheta] \subset ]0, \pi/2[$} \\ 
	\hline
	& & \\[-2ex]
	\includegraphics[scale=0.45]{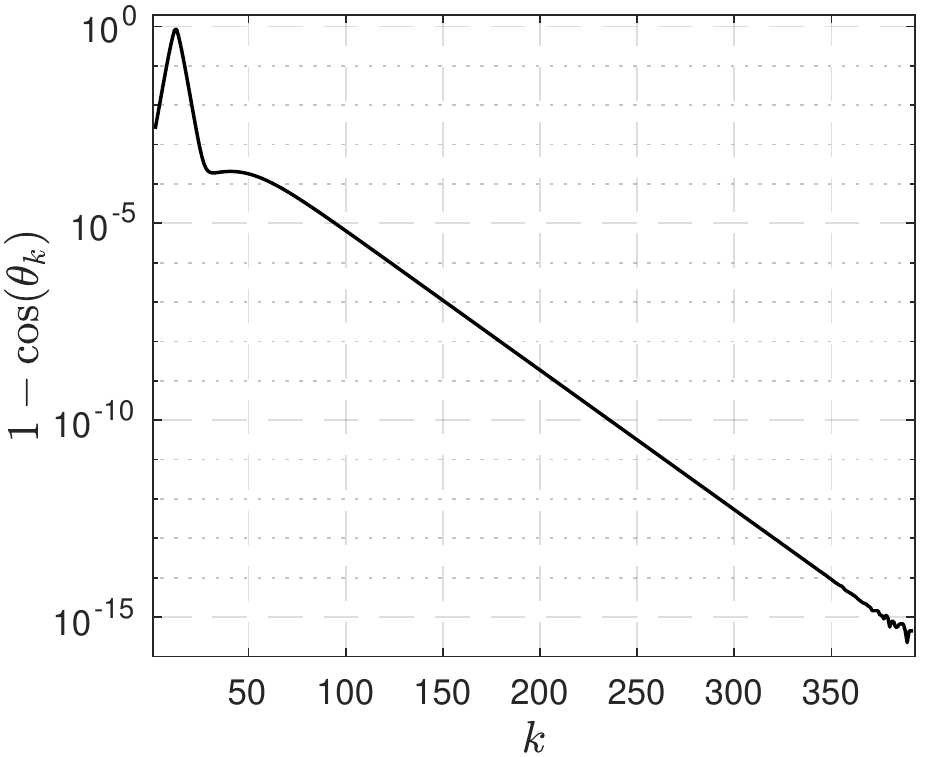}
	&
	\includegraphics[scale=0.45]{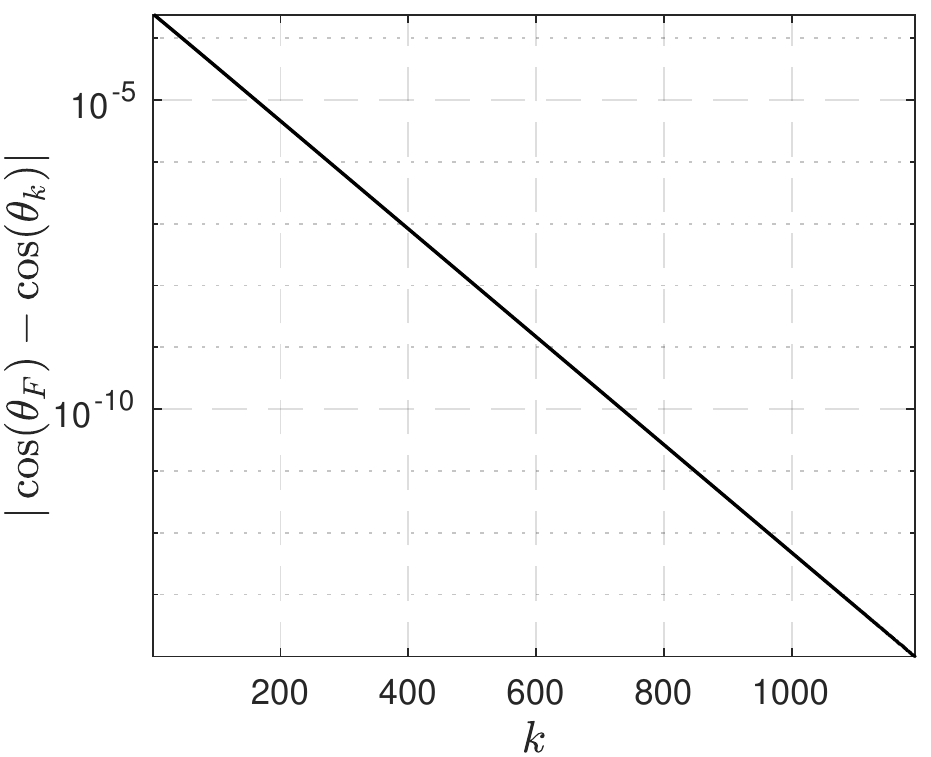}
	&
	\includegraphics[scale=0.45]{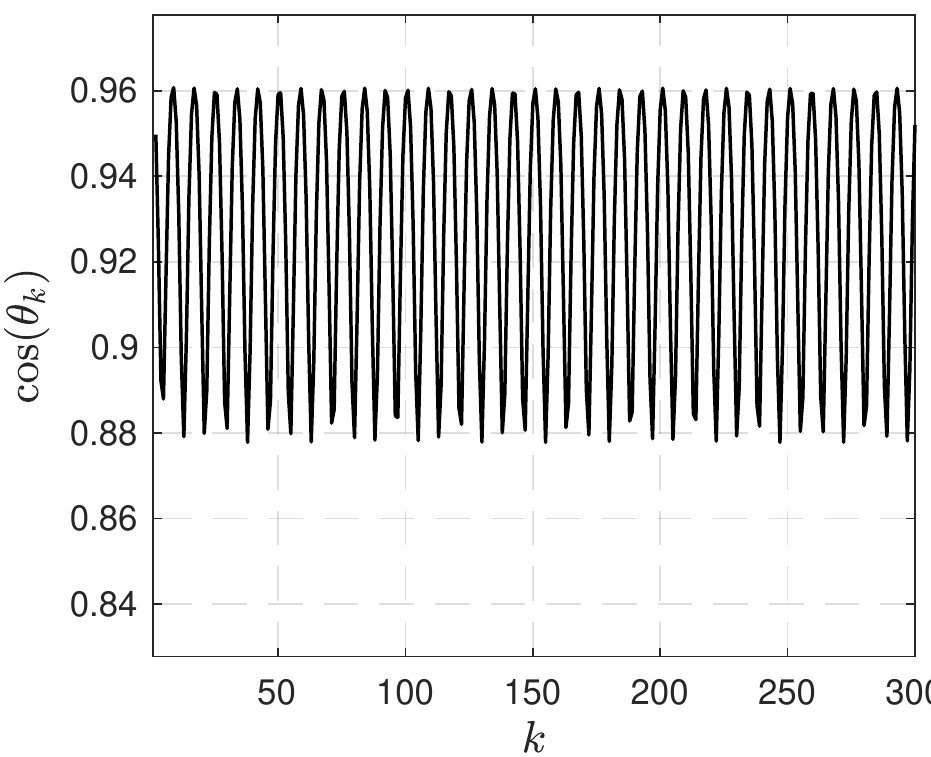} \\
	\hline
	& & \\[-2.5ex]
	\includegraphics[scale=0.45]{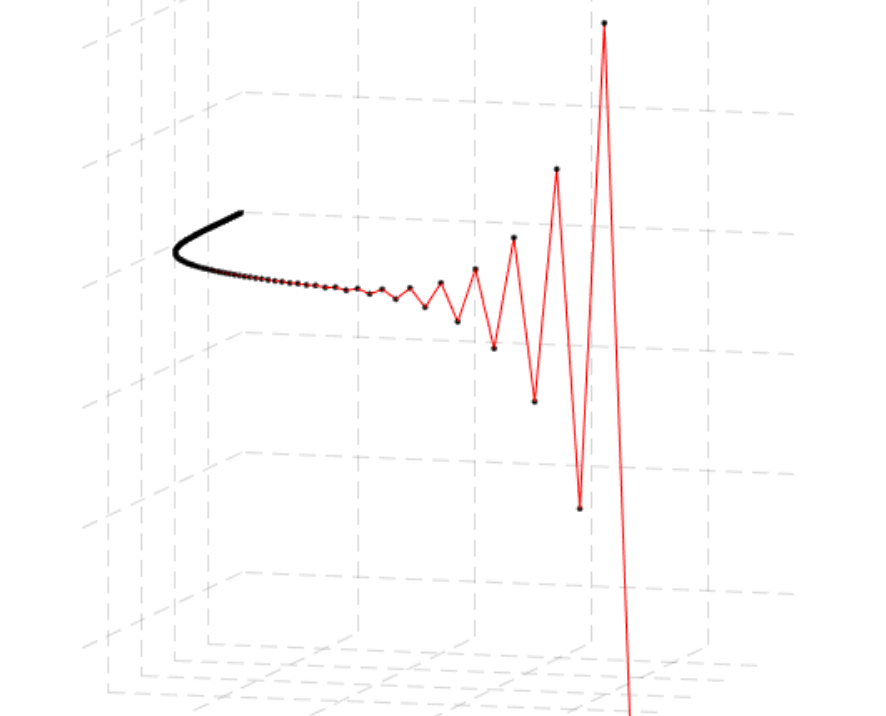}
	&
	\includegraphics[scale=0.4]{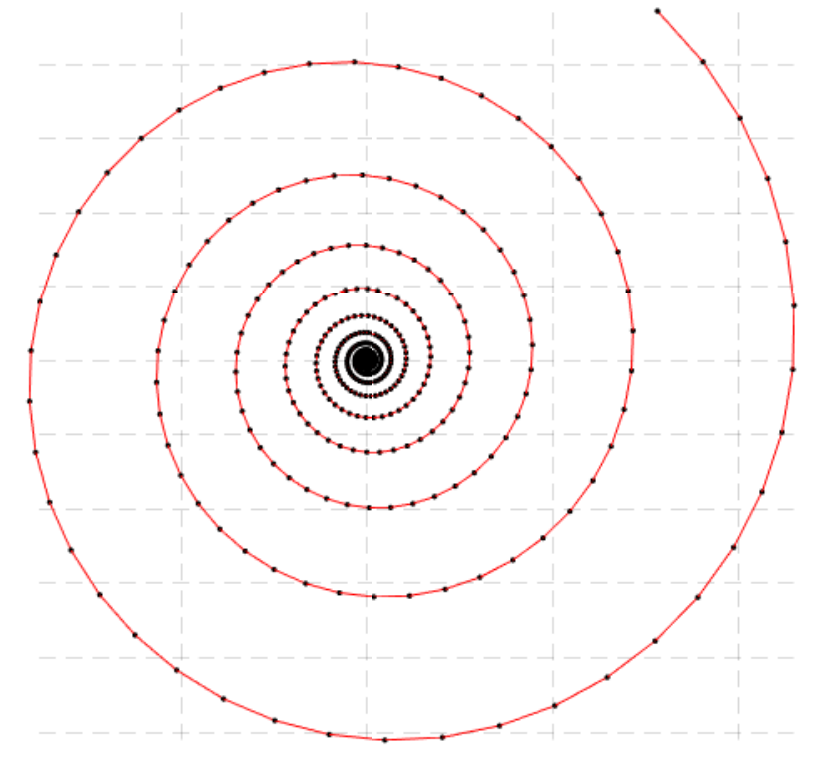}
	&
	\includegraphics[scale=0.4]{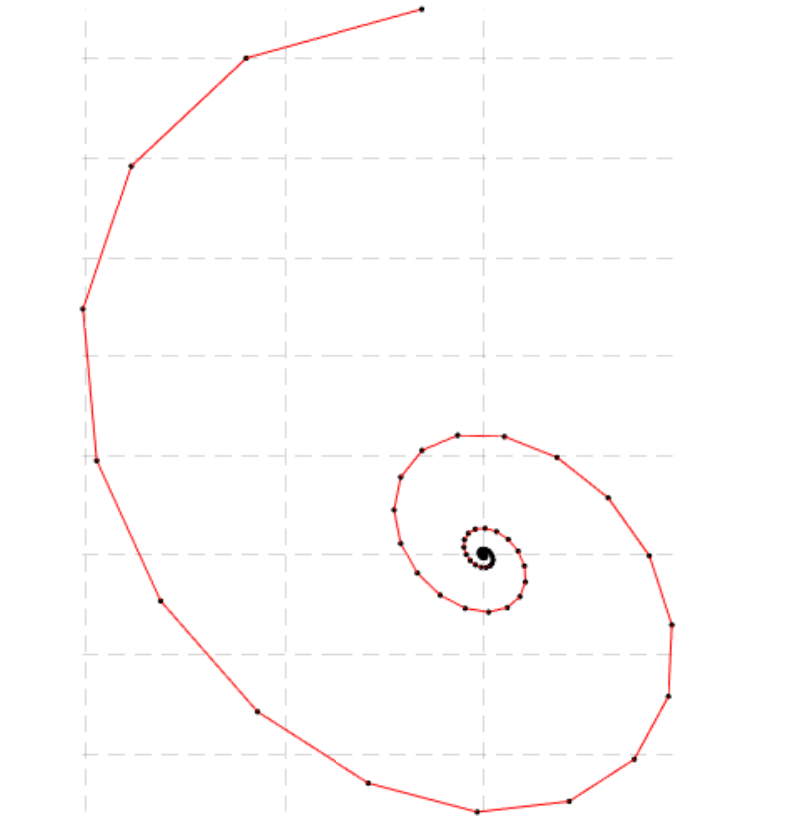} \\
	& & \\[-4.5ex]
	&
	\includegraphics[scale=0.4]{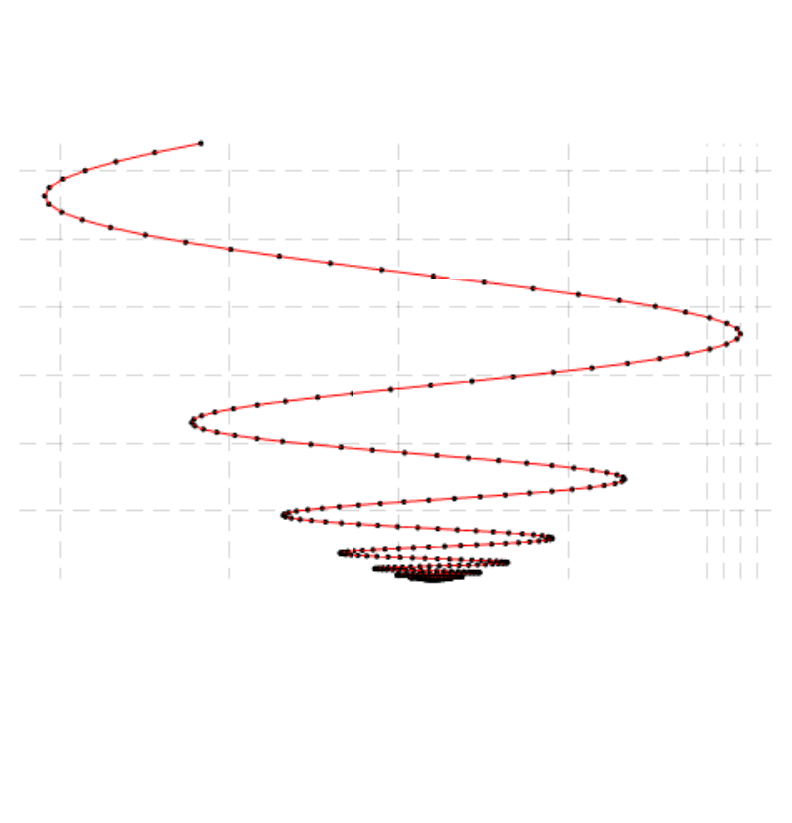} 
	&
	\includegraphics[scale=0.4]{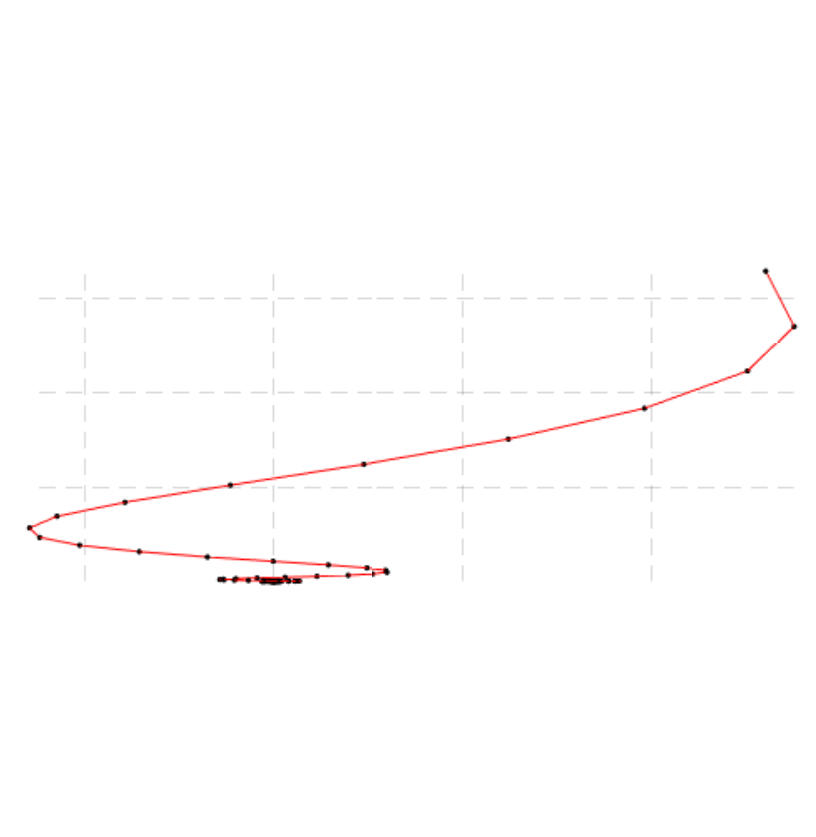}  \\
	\hline
\end{tabular}
\end{center}
\end{table}

\renewcommand{\arraystretch}{1}


\vspace{-3ex}

%




\section{Local trajectory of first-order methods}\label{sec:trajectory-fom}

As mentioned above, sequence trajectory can be easily analyed in the case of linear systems. 
When dealing with \emph{non-smooth optimization}, although the fixed point operators are nonlinear, locally around the solution the fixed-point operators can be linearized with respect to some smooth manifolds under the help of ``partial smoothness''. 
In this section, we present an abstract framework for analyzing the local trajectory of first-order methods and apply it to analyze several popular first-order methods. 
%
%
All the proofs for propositions in this section are provided in Section \ref{proof:trajectory-fom} of the appendix.

\subsection{A framework based on partial smoothness}
\label{sec:framework}


Recall the fixed-point iteration of first-order methods \eqref{eq:fom}: $\zkp = \calF (\zk)$. 
%
Define the difference vector $\vk \eqdef \zk - \zkm$ and the angle $\theta_k \eqdef \angle(\vk, \vkm)$ between $\vk, \vkm$ as in \eqref{eq:angle-thetak}. 
%
We propose the following framework for analyzing the trajectory of sequence $\seq{\zk}$.

\begin{center}
	\begin{minipage}{0.975\linewidth}
		\begin{algorithm}[H]
			\NoCaptionOfAlgo
			\caption{A framework for analyzing local trajectory of first-order methods} \label{alg:framework}
			\textbf{1. Convergent sequence} The iteration is convergent and ${\zk} \to \zsol \in \fix(\calF)$.
			
			\textbf{2. Manifold identification} Under a proper non-degenerate condition, see \eg \eqref{eq:ndc-fb} and \eqref{eq:ndc-dr}, the sequence(s) generated by $\calF$ has finite manifold identification property.
			
			\textbf{3. Local linearization}  There exists a linear matrix $M_{\calF}$ such that along the identified smooth manifold(s) the global non-linear iteration locally can be linearized
			\beq\label{eq:linearization-zkp-zk} 
			\zkp-\zk = M_{\calF} (\zk-\zkm) + o(\norm{\zk-\zkm})  . 
			\eeq
			
			\textbf{4. Spectrum of $M_{\calF}$} Owing to the structure of the optimization problem and first-order method, $M_{\calF}$ will have certain spectral properties, \eg real or complex spectrum. 
			
			\textbf{5. Trajectories of $\seq{\zk}$} The leading eigenvalue of $M_{\calF}$ determines the trajectory of $\seq{\zk}$. 
		\end{algorithm}
	\end{minipage}
\end{center}
\setcounter{algocf}{0}

\begin{remark}  $~$
	\begin{itemize}
		
		\item ``Steps 1-4'' of the above framework are also the essential steps of the local linear convergence analysis framework for first-order methods \cite{liang2016thesis}. For example, if the spectral radius of $M_{\calF}$ is strictly smaller than $1$, then one can derive the local linear convergence result. 
		
		\item The finite manifold identification is not necessarily for $\seq{\zk}$, as general first-order methods generate several different points along each iteration. Take Douglas--Rachford splitting method (see Eq. \eqref{eq:dr}) for example, $\seq{\zk}$ is the fixed-point sequence of the method, however the identification is for the shadow sequences $\seq{\uk}$ and $\seq{\xk}$; See Section \ref{subsec:trajectory-dr} for details. 
		
		\item The $o$-terms in \eqref{eq:linearization-zkp-zk} are due to the non-linearity of $\calF$ and the curvature of the identified manifold(s). 
		In a series of work \cite{liang2014local,liang2016thesis,liang2017activity}, the linearization is considered with respect to $\zsol$, that is 
		$ \zkp-\zsol = M_{\calF} (\zk-\zsol) + o(\norm{\zk-\zsol}) $. 
		The main reason of linearization in terms of $\zkp$ and $\zk$ is to better motivate the acceleration scheme in Section \ref{sec:lp}. 
		
		
		
	\end{itemize}
\end{remark}



\begin{remark}[What determines the trajectory of $\seq{\zk}$ continued]

{Although the linearization \eqref{eq:linearization-zkp-zk} makes it possible to analyze the trajectory of $\seq{\zk}$, it is still difficult to give precise characterizations  due to the presence of the small $o$-term. In the case where  $M_{\calF}$ contains only real eigenvalues, we show in Theorem \ref{thm:trajectory-fb} that this small $o$-term can be ignored, while in the case where $M_{\calF}$ contains complex eigenvalues, we give characterizations only when   the optimization problem to be \emph{locally polyhedral} around the solution (see Theorem \ref{thm:trajectory-dr-1} and \ref{thm:trajectory-pd-1}).}

\end{remark}

In the following, we apply the above framework to analyze the trajectory of three classical first-order algorithms: Forward--Backward splitting \cite{lions1979splitting}, Douglas--Rachford/ADMM \cite{douglas1956numerical,gabay1976dual} and Primal--Dual splitting \cite{chambolle2011first}. 
{For the purpose of readability, in this section we mainly provide the qualitative description of the trajectory of these methods (\eg which type), and omit the quantitative characterization (\eg speed of convergence of $\cos(\theta_k)$). }
All the proofs for propositions in this section are provided in Section \ref{proof:trajectory-fom}.

\subsection{Forward--Backward splitting}\label{subsec:trajectory-fb}

Forward--Backward splitting \cite{lions1979splitting} is designed to solve the following optimization problem
\beq\label{eq:problem-fb}\tag{$\mathcal{P}_{\mathrm{FB}}$}
\min_{x \in \bbR^n }~ \{\Phi(x) \eqdef R(x) + F(x)\} ,
\eeq
where the following assumptions are imposed
\begin{enumerate}[leftmargin=4em,label= ({${\bf F}$\textbf{.\arabic{*}})},ref= ${\bf F}$\textbf{.\arabic{*}}]
\item \label{FB:R} $R \in \lsc\pa{\bbR^n}$ is proper convex and lower semi-continuous.
\item \label{FB:F} $F \in C^{1,1}(\bbR^n )$ is convex differentiable with gradient $\nabla F$ being $L$-Lipschitz continuous.
\item \label{FB:minimizers-nonempty} $\Argmin(\Phi) \neq \emptyset$, \ie the set of minimizers is non-empty.
\end{enumerate}
The iteration of Forward--Backward splitting method is described in Algorithm \ref{alg:fb}.

\begin{center}
\begin{minipage}{0.975\linewidth}
\begin{algorithm}[H]
\caption{Forward--Backward splitting} \label{alg:fb}
\KwIn{$\gamma \in ]0, 2/L[$.}
{\noindent{\bf{Initial}}}: $x_{0} \in \bbR^n$.  \\ 
{\noindent{\bf{Repeat:}}} 
\beq\label{eq:fb}
\xkp = \prox_{\gamma R}\Pa{\xk - \gamma \nabla F (\xk)} . 
\eeq
%
{\noindent{\bf{Until:}}} $\norm{\xkp-\xk} \leq \tol$.
\end{algorithm}
\end{minipage}
\end{center}

The fixed-point formulation of Forward--Backward splitting is quite straightforward, which reads
\beqn
\xkp = \fFB(\xk) 
\qwhereq
\fFB \eqdef \prox_{\gamma R}\pa{\Id - \gamma \nabla F} .
\eeqn

\begin{remark}
In the literature, various inertial variants of Forward--Backward splitting are proposed, such as inertial Forward--Backward and FISTA \cite{fista2009,chambolle2015convergence,liang2017activity,liang2018improving}. However, these methods will not be covered in this paper as the fixed-point operators of these schemes are not \emph{non-expansive}, and trajectory of the sequence and acceleration for these schemes are much more complicated. 
\end{remark}

Let $\xsol \in \Argmin(R+F)$ be a global minimizer, we impose the following non-degeneracy condition
\beq\label{eq:ndc-fb}\tag{$\mathrm{ND}_{_{\mathrm{FB}}}$}
- \nabla F(\xsol) \in \ri\Pa{\partial R(\xsol)} 	.
\eeq
We refer to \cite{liang2017activity} for more detailed discussions about these conditions for the local linear convergence of the general Forward--Backward-type splitting methods.

\begin{remark}
	Throughout this section, we impose the non-degeneracy conditions, also for the Douglas--Rachford and Primal--Dual splitting methods, for our analysis. Based on a recent work \cite{fadili2018sensitivity}, when the function $R$ is so-called ``mirror-stratifiable'', condition \eqref{eq:ndc-fb} can be removed. However, we will not dive into this direction, since it will not affect the conclusion of this section. 
\end{remark}

We have the following result for the trajectory of $\seq{\xk}$. Redefine $\vk = \xk - \xkm$.


\begin{theorem}\label{thm:trajectory-fb}
	For problem \eqref{eq:problem-fb} and the Forward--Backward splitting method \eqref{eq:fb}, suppose that assumptions \iref{FB:R}-\iref{FB:minimizers-nonempty} are true, then $\sequence{\xk}$ converges to a global minimizer $\xsol \in \Argmin(\Phi)$. 
If, moreover, $R \in \PSF{\xsol}{\Msol}$, $F$ is locally $C^2$ around $\xsol$ and condition \eqref{eq:ndc-fb} holds, there exists a matrix $\mFB$ such that for all $k$ large enough
\beqn
\xkp - \xk = \mFB(\xk-\xkm) + o(\norm{\xk-\xkm})  .
\eeqn
Moreover, we have
\begin{enumerate}[label={\rm (\roman{*})}]

	\item All the eigenvalues of $\mFB$ are real and lie in $]-1, 1]$. 
	
	\item { Let $\sigma_2$ be the second largest eigenvalue of $\mFB$. If $\norm{x_{k+1} - x_k} \asymp \rho^k$ for some $\rho \in ]\sigma_2, 1[$, then the angle $\theta_k$ is convergent with $\theta_k \to 0$ and $\seq{\xk}$ is a Type I sequence. 
}

\end{enumerate}
\end{theorem}

\begin{remark} $~$
	\begin{itemize}
	
		\item The detailed expression of $\mFB$ can be found in Section \ref{proof:trajectory-fom}, and the result holds true for varying but convergent step-size $\gamma_{k} \in [0, 2/L]$.

		\item If there holds $R$ is locally polyhedral around $\xsol$, $F$ is quadratic, then for the linearization we have directly $\xkp-\xk = \mFB(\xk-\xkm)$ without the $o$-terms and a straight-line trajectory is guaranteed.  
		
		\item 
		Note that the linearization means that locally, $\norm{\xk-\xkm} = O(\rho^k)$ for $\rho \in ]\sigma_1, 1[$ with $\sigma_1$ being the largest eigenvalue of $\mFB$. So, (ii) of Theorem \ref{thm:trajectory-fb} implies that under mild assumption, the eventual trajectory of $\seq{\xk}$ for Forward--Backward is a straight line.
		
		
	\end{itemize}
\end{remark}

\begin{example}\label{eg:lasso}
We consider regularized least square  
\[
\min_{x\in\bbR^n}~ R(x) + \sfrac{1}{2}\norm{Ax-b}^2
\]
to demonstrate the property of $\seq{\theta_k}$. 
Two different cases of $R$ are considered: $\ell_1$-norm which is polyhedral and nuclear norm which is not polyhedral. We have $A \in \bbR^{m\times n}$ and for each cases the settings are
\begin{description}[leftmargin=3.0cm]
\item[{$\ell_{1}$-norm}] $(m,n)=(48,128)$, the solution $\xsol$ has $14$ non-zero elements.
\item[{Nuclear norm}] $(m,n)=(868,1024)$, the solution $\xsol$ has rank of $2$. 
\end{description}
For both examples, $A$ is generated from the standard random Gaussian ensemble. 
The numerical results are shown in Figure \ref{fig:trajectory-fb}. 
For $\ell_1$-norm, besides $\theta_k$, we also provide the change of support size of $\xk$, \ie $\abs{\supp(\xk)}$: 
\begin{itemize}
	\item For the support of $\xk$, three phases can be observed: at beginning $\xk$ is almost in the whole space, then the size of supports starts to decrease and eventually becomes stable which is the activity identification.  

	\item The behavior of $\theta_k$ also has three phases: 1) when $\xk$ is in the whole space, $\theta_k$ is equal or very close to $0$; 2) When the support is decreasing, $\theta_k$ oscillates; 3) After identification, $\theta_k$ converges to $0$ linearly. 
	
\end{itemize}
For nuclear norm, the change of rank of $\xk$ is provided
\begin{itemize}
	\item Different form the $\ell_1$-norm, the rank of $\xk$ gradually decreases, results in a staircase observation. 

	\item For $\theta_k$, inside each staircase, it decrease first and then increases. But after identification of the rank, it converges to $0$ linearly. 
	
	
\end{itemize}
\end{example}

\begin{figure}[!ht]
\centering
\subfloat[$\ell_1$-norm]{ \includegraphics[width=0.45\linewidth]{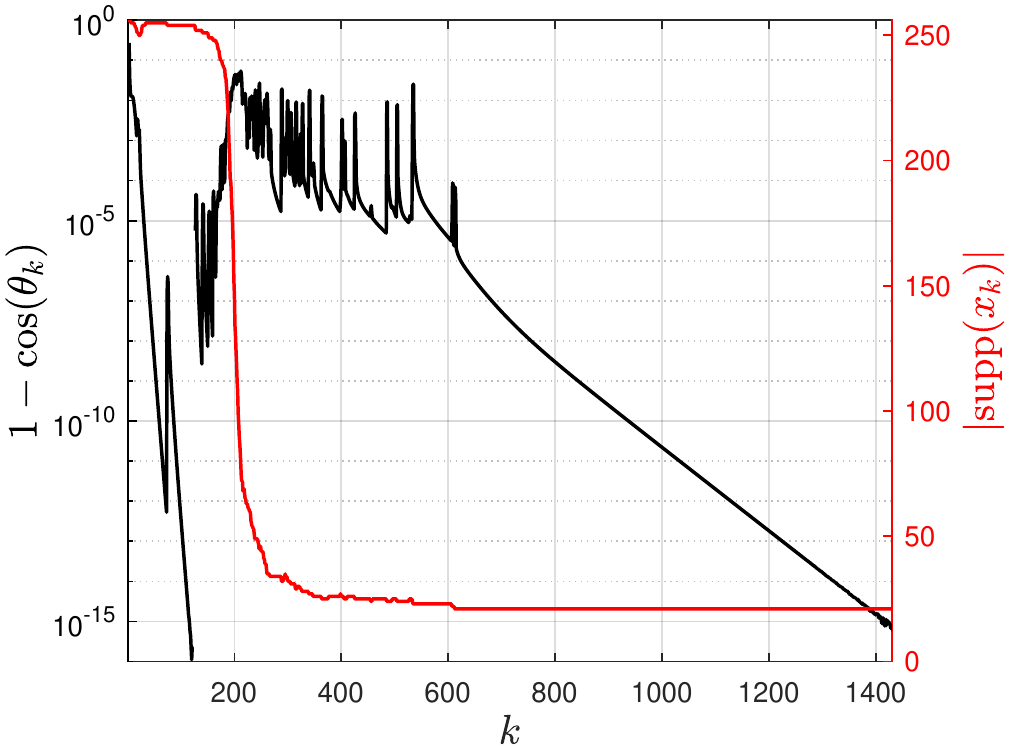}  }     {\hspace{6pt}}
\subfloat[Nuclear norm]{ \includegraphics[width=0.45\linewidth]{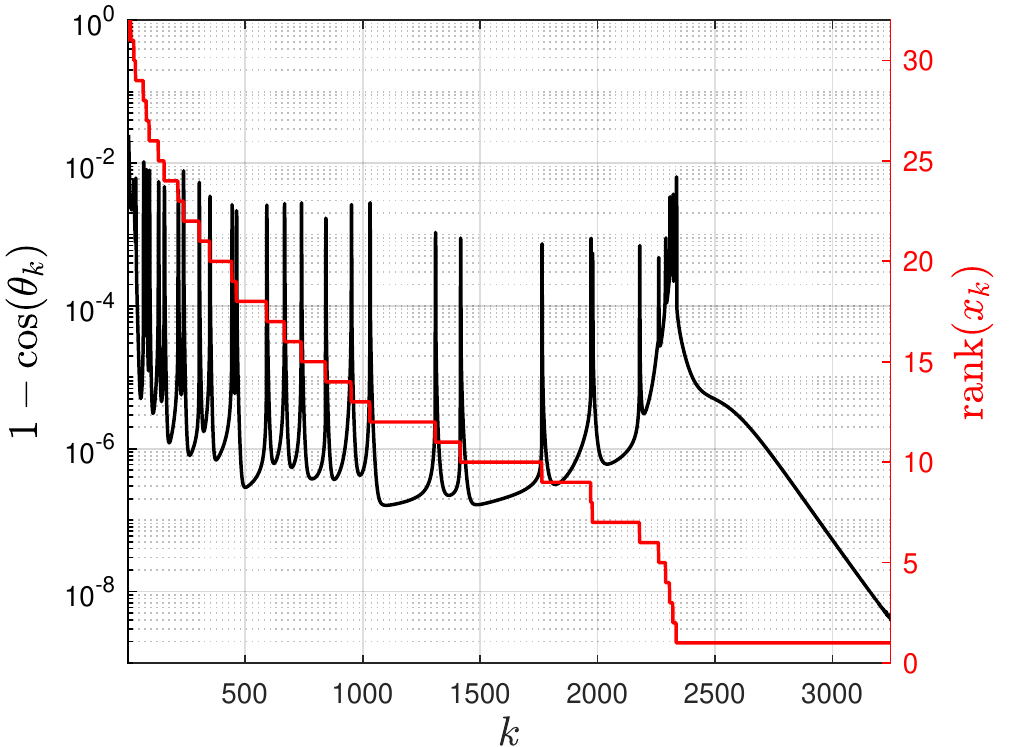}  }     \\
\caption{Finite activity identification and property of $\theta_k$ for Forward--Backward splitting method. }
\label{fig:trajectory-fb}
\end{figure}


\subsection{Douglas--Rachford splitting and ADMM}\label{subsec:trajectory-dr}

The second example is Douglas--Rachford splitting \cite{douglas1956numerical} for solving the sum of two non-smooth functions
\beq\label{eq:problem-dr}\tag{$\calP_{\mathrm{DR}}$}
\min_{x \in \bbR^n}~ R(x) + J(x) ,
\eeq
where we assume
\begin{enumerate}[leftmargin=4em, label= ({${\bf D}$\textbf{.\arabic{*}})}, ref= ${\bf D}$\textbf{.\arabic{*}}]
\item \label{DR:RJ}
$R, J \in \lsc\pa{\bbR^n}$, the proper convex and lower semi-continuous functions.

\item \label{DR:domain}
$\ri\pa{\dom(R)} \cap \ri\pa{\dom(J)} \neq \emptyset$, \ie the domain qualification condition.

\item \label{DR:minimizers-nonempty}
$\Argmin(R+J) \neq \emptyset$, \ie the set of minimizers is non-empty.
\end{enumerate}
The standard Douglas--Rachford splitting method \cite{douglas1956numerical} is described in Algorithm \ref{alg:dr}.

\begin{center}
\begin{minipage}{0.975\linewidth}
\begin{algorithm}[H]
\caption{Douglas--Rachford splitting} \label{alg:dr}
\KwIn{$\gamma > 0$.}
{\noindent{\bf{Initial}}}: $z_0 \in \bbR^n,~ x_0 = \prox_{\gamma J} (z_0)$\;
{\noindent{\bf{Repeat:}}} 
\beq\label{eq:dr}
\begin{aligned}
\ukp &= \prox_{\gamma R}\pa{2\xk - \zk} , \\
\zkp &= \zk + \ukp - \xk , \\
\xkp &= \prox_{\gamma J} (\zkp) ,
\end{aligned}
\eeq
{\noindent{\bf{Until:}}} $\norm{\zkp-\zk} \leq \tol$.
\end{algorithm}
\end{minipage}
\end{center}

The fixed-point formulation of Douglas--Rachford with respect to $\zk$ is
\beqn
\zkp = \fDR(\zk) 
\qwhereq
\fDR \eqdef \sfrac{1}{2} \Pa{ (2\prox_{\gamma R} - \Id) (2\prox_{\gamma J} - \Id) + \Id } .
\eeqn

\begin{remark}
It is well known that the alternating direction method of multipliers (ADMM) is closely connected with Douglas--Rachford splitting method, for its local trajectory property of generated sequences, we refer to \cite{PoonLiang2019b} for a detailed discussion. 

\end{remark}

Below we first present the linearization of Douglas--Rachford iteration \eqref{eq:dr} and then discuss the trajectory of $\seq{\zk}$ under two different cases: both $R, J$ in \eqref{eq:problem-dr} are non-smooth as in \iref{DR:RJ}, and one of the functions is smooth. We shall see that two different trajectories are exhibited by the method.

\subsubsection{Linearization of Douglas--Rachford splitting}

Let $\zsol \in \fix(\fDR)$ and $\xsol = \prox_{\gamma J} (\zsol) \in \Argmin(R+J)$ such that $\zk\to\zsol$ and $\xk,\uk\to\xsol$, from \eqref{eq:dr} the corresponding first-order optimality condition reads 
$\xsol - \zsol \in \gamma \partial R(\xsol) ~\textrm{and}~ \zsol - \xsol \in \gamma \partial J(\xsol)$. 
We assume the following non-degeneracy condition
\beq\label{eq:ndc-dr}\tag{$\textrm{ND}_{_{\mathrm{DR}}}$}
\xsol - \zsol \in \gamma \ri\Pa{ \partial R(\xsol) } \qandq
\zsol - \xsol \in \gamma \ri\Pa{ \partial J(\xsol) } 	.
\eeq 
Let $\MmR$ and $\MmJ$ be two $C^2$-smooth manifolds around $\xsol$. 

\begin{theorem}\label{thm:linearization-dr}
For problem \eqref{eq:problem-dr} and the Douglas--Rachford splitting algorithm \eqref{eq:dr}, suppose that the conditions \iref{DR:RJ}-\iref{DR:minimizers-nonempty} are true, then $\seq{\zk}$ converges to a point $\zsol \in \fix\pa{\fDR}$ and $\seq{\xk}, \seq{\uk}$ converge to $\xsol \eqdef \prox_{\gamma R}(\zsol) \in \Argmin(R+J)$. 
If moreover, $R \in \PSF{\xsol}{\MmR}$ and $J \in \PSF{\xsol}{\MmJ}$ are partly smooth and condition \eqref{eq:ndc-dr} holds, then there exists a matrix $\mDR$ such that for all $k$ large enough
\beqn
 \zkp-\zk = \mDR (\zk-\zkm) + o(\norm{\zk-\zkm}) .
\eeqn
%
\end{theorem}

See Section \ref{proof:trajectory-dr} for the proof and expression of $\mDR$. We refer to \cite{liang2017localDR} for detailed discussions on the local linear convergence of Douglas--Rachford splitting method.


\subsubsection{Trajectory of Douglas--Rachford splitting}

We first consider the case that both $R$ and $J$ are non-smooth. 
Let $R \in \PSF{\xsol}{\MmR}, J \in \PSF{\xsol}{\MmJ}$, denote $T_{\xsol}^R, T_{\xsol}^J$ the tangent spaces of $\MmR, \MmJ$ at $\xsol$, respectively. 
And let $\PR, \PJ$ be the projection operators onto $T_{\xsol}^R, T_{\xsol}^J$, respectively. 
Denote $\theta_F$ the Friedrichs angle between $T_{\xsol}^R$ and $T_{\xsol}^J$.

\begin{theorem}\label{thm:trajectory-dr-1}
For problem \eqref{eq:problem-dr} and the Douglas--Rachford splitting algorithm iteration \eqref{eq:dr}, assume that Theorem \ref{thm:linearization-dr} holds. 
If, moreover, $R, J$ are locally polyhedral around $\xsol$, then $\zkp-\zk = \mDR (\zk-\zkm)$ with
\beqn
\mDR = \PR \PJ + (\Id - \PR)(\Id - \PJ )  .
\eeqn
If moreover $\dim(T_{\xsol}^R \cap T_{\xsol}^J) < \min\Ba{ \dim(T_{\xsol}^R ) , \dim(T_{\xsol}^J)} $, the angle $\theta_k$ is convergent to $\theta_F \in ]0, \pi/2]$ and $\seq{\zk}$ is a Type II sequence. 

\end{theorem}


%
\begin{remark} $~$
	\begin{itemize}
	\item 
	The spectral properties of $\mDR$ is much more difficult to analyze compared to that of Forward--Backward splitting, and for the case both $R,J$ are non-smooth, we make the additional assumption of  local polyhedrality around the solution. 
	For this setting, $\mDR$ is a normal matrix, hence quasi-diagonalizable \cite[Theorem 2.5.8]{horn1990matrix}, where the leading block of the decomposition reads
	\[
B = \cos(\theta_F) \begin{bmatrix} \cos(\theta_F) & \sin(\theta_F) \\ -\sin(\theta_F) & \cos(\theta_F) \end{bmatrix} .
\]
	The condition $\dim(T_{\xsol}^R \cap T_{\xsol}^J) < \min\Ba{ \dim(T_{\xsol}^R ) , \dim(T_{\xsol}^J)} $ ensures $\theta_F \in ]0, \pi/2]$ which makes $B$ a rotation. Consequently the local trajectory of the sequence $\seq{\zk}$ is a \emph{logarithmic spiral}. Moreover, the choice of $\gamma$ does not affect the local trajectory of $\seq{\zk}$ as $B$ only depends on the Friedrichs angle $\theta_F$.

	\item 
	The analysis of the $\cos(\theta_k)$ depends on the explicit expression of the leading eigenvalues of $\mDR$, which is only available when $R,J$ are locally polyhedral around the solution. For the case that $R, J$ are general partly smooth function, the behavior of $\theta_k$ depends on $\gamma$ due to the non-trivial Riemannian Hessian of $R,J$; See Figure \ref{fig:trajectory-dr} for an illustration. 
		
	\end{itemize}
\end{remark}

\begin{example}\label{eg:exp-dr}
We use the affine constrained problem
\beq\label{eq:R-affine}
\min_{x\in\bbR^n}~ R(x) \enskip \mathrm{such~that} \enskip  Ax = A\cx 
\eeq
to demonstrate the property of $\seq{\theta_k}$. 
Similar to Example \ref{eg:lasso}, $\ell_1$-norm and nuclear norm are considered for $R$, $A \in \bbR^{m\times n}$ is generated from the standard random Gaussian ensemble and
\begin{description}[leftmargin=3.0cm]
\item[{$\ell_{1}$-norm}] $(m,n)=(48,128)$, $\cx$ has $8$ non-zero elements.
\item[{Nuclear norm}] $(m,n)=(620,1024)$, $\cx$ has rank of $2$. 
\end{description}
The results are shown in Figure \ref{fig:trajectory-dr}, the observations of $\ell_1$-norm are similar to those in Example \ref{eg:lasso}, except that $\theta_k$ eventually converges to some non-zero values. 
For nuclear norm, two choices of $\gamma$, $\gamma=1, 6$, are considered. Observe that after rank identification, $\theta_k$ oscillates in an interval for $\gamma=1$ and behaves smoothly for $\gamma = 6$. 
%
%
%
\end{example}

\begin{figure}[!ht]
\centering
\subfloat[$\ell_1$-norm]{ \includegraphics[width=0.315\linewidth]{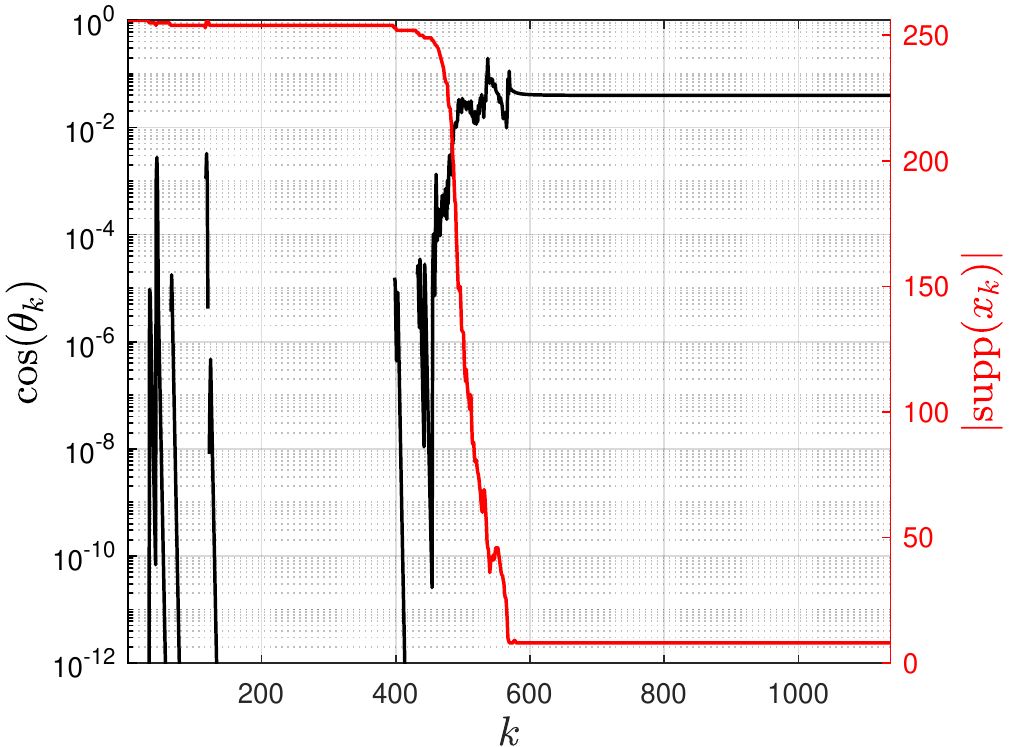}  }     {\hspace{1pt}}
\subfloat[Nuclear norm: $\gamma = 1$]{ \includegraphics[width=0.315\linewidth]{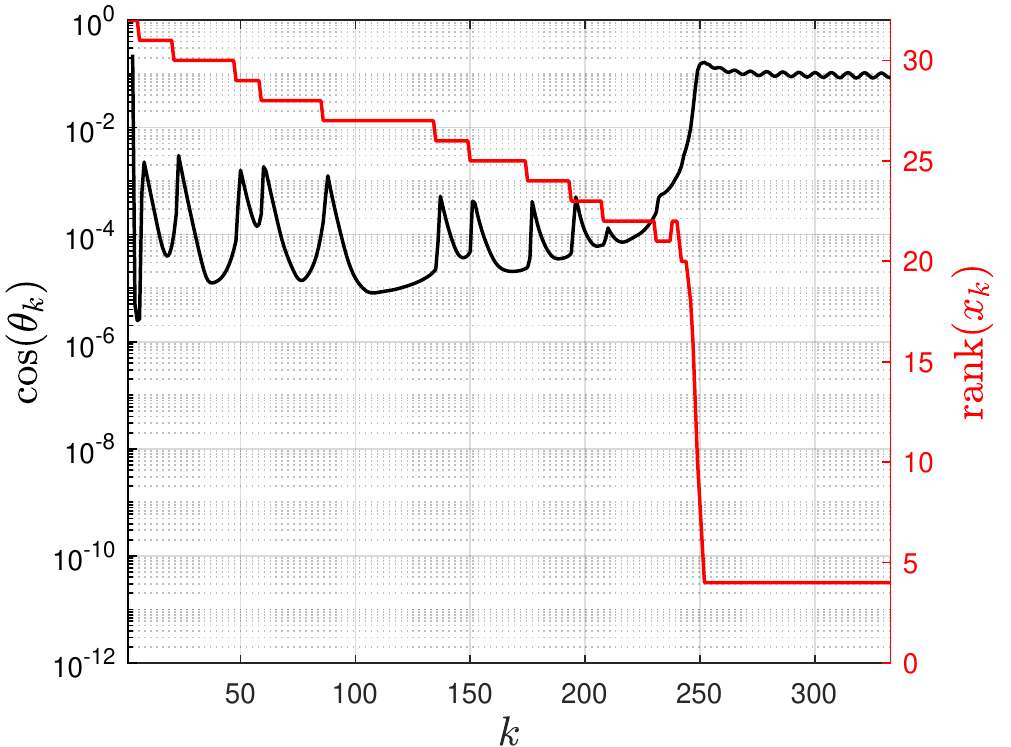}  }     {\hspace{1pt}}
\subfloat[Nuclear norm: $\gamma = 6$]{ \includegraphics[width=0.315\linewidth]{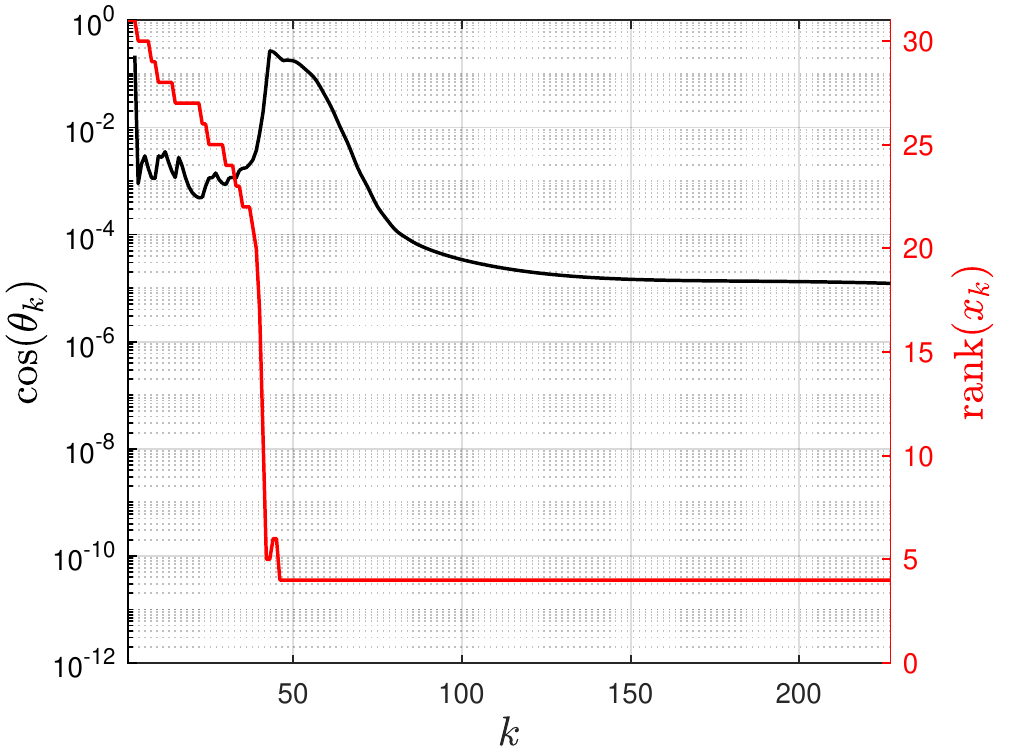}  }     \\
\caption{Finite activity identification and property of $\theta_k$ for Douglas--Rachford splitting method for solving affine constrained problem. For the plot (b) and (c), the $A$ and $\cx$ in \eqref{eq:R-affine} are the same.} 
\label{fig:trajectory-dr}
\end{figure}


\begin{remark}
As the nuclear norm is not polyhedral, it has non-trivial Riemannian Hessian matrix. Therefore, the choices of $\gamma$ affects the eventually behavior of $\seq{\theta_k}$. While for $\ell_1$-norm, the value that $\seq{\theta_k}$ converges to is independent of $\gamma$. 
\end{remark}

Assume now that $R$ is locally $C^{2}$-smooth around the solution, we shall see that different from the above polyhedral case, the choice of $\gamma$ will impact the trajectory of $\seq{\zk}$.

\begin{theorem}\label{thm:trajectory-dr-2}
For problem \eqref{eq:problem-dr} and the Douglas--Rachford splitting algorithm \eqref{eq:dr}, assume conditions \iref{DR:RJ}-\iref{DR:minimizers-nonempty} are true. 
If $R$ is locally $C^2$ around $\xsol$ and $J \in \PSF{\xsol}{\MmJ}$ is partly smooth and condition \eqref{eq:ndc-dr} holds for $J$, then Theorem \ref{thm:linearization-dr} holds. 
If moreover $\gamma$ is chosen such that $\gamma < \frac{1}{\norm{\nabla^2 R(\xsol)}}$, then all the eigenvalues of $\mDR$ are real.
%
Consequently, provided that $\norm{x_{k+1} - x_k} \asymp \rho^k$ for some $\rho \in ]\sigma_2, 1[$ where   $\sigma_2$ is the second largest eigenvalue of $\mDR$, the angle $\theta_k$ converges to $0$ and $\seq{\zk}$ is a Type I sequence. 

%
\end{theorem}


\begin{remark}$~$
\begin{itemize}
\item The result also holds true for the case when both $R, J$ are smooth. 

\item
When $\gamma \geq  \frac{1}{\norm{\nabla^2 R(\xsol)}}$, $\mDR$ can have complex eigenvalues, however not necessarily for the leading one, as a result the trajectory of $\seq{\zk}$ can be either straight line or spiral. 


\end{itemize}
\end{remark}

We refer to Figure \ref{fig:success_idr} for example of applying Douglas--Rachford splitting method to solve LASSO problem, on how the choice of $\gamma$ affects the trajectory of $\seq{\zk}$.

\subsection{Primal--Dual splitting} \label{subsec:trajectory-pd}

For problem \eqref{eq:problem-dr}, consider function $J$ is composed with a linear mapping $L$
\beq\label{eq:problem-pd}\tag{\textrm{$\calP_{\mathrm{PD}}$}}
\min_{x \in \bbR^n}~ R(x) + J (Lx) ,
\eeq
where we assume 
\begin{enumerate}[label = {({${\bf P}$\textbf{.\arabic{*}}})}, ref = {{${\bf P}$\textbf{.\arabic{*}}}}, leftmargin=4em]
\item\label{PD:RJ}
$R \in \lsc\pa{\bbR^n}$ and $J \in \lsc\pa{\bbR^m}$.
\item\label{PD:L}
$L:\bbR^n\rightarrow\bbR^m$ is a linear mapping.
\item\label{PD:minimizers-nonempty}
The inclusion $0 \in \ran \pa{ \partial R + L^T \circ \partial J \circ L }$ holds.
\end{enumerate}
The problem above can be handled efficiently by ADMM, in the literature, another popular approach is the Primal--Dual splitting method. 
The {saddle-point problem} associated to \eqref{eq:problem-pd} reads
\beq\label{eq:saddle-problem}\tag{\textrm{$\calP_{\mathrm{SP}}$}}
\min_{x\in\bbR^n} \max_{w\in\bbR^m}~ R(x) + \iprod{Lx}{w} - J^*(w)  ,
\eeq
where $J^*$ is the Legendre-Fenchel conjugate of $J$. 
If we fully dualize \eqref{eq:problem-pd}, then we obtain its {Fenchel-Rockafellar dual form} 
\beq\label{eq:dual-problem}\tag{\textrm{$\calD_{\mathrm{PD}}$}}
\min_{w\in \bbR^m} R^*(-L^T w) + J^* (w)  .
\eeq
Denote by $\calX$ and $\calW$ the sets of solutions of problem \eqref{eq:problem-pd} and \eqref{eq:dual-problem}, respectively.

Below we describe a Primal--Dual splitting method \cite{chambolle2011first} for solving the saddle point problem.

\begin{center}
\begin{minipage}{0.975\linewidth}
\begin{algorithm}[H]
\caption{A Primal--Dual splitting method} \label{alg:pd}
\KwIn{$\gammaR, \gammaJ > 0$ such that $\gammaR\gammaJ\norm{L}^2 < 1$ and $\tau \in [0,1]$.}
{\noindent{\bf{Initial}}}: $x_0\in\bbR^n$, $w_0\in\bbR^m$\;
{\noindent{\bf{Repeat:}}} 
\beq\label{eq:pd}
\begin{aligned}
\xkp &= \prox_{\gammaR R} \pa{ \xk - \gammaR L^T \wk } , \\
\xbarkp &= \xkp + \tau(\xkp - \xk) , \\
\wkp &= \prox_{\gammaJ J^*} \pa{ \wk  + \gammaJ L \xbarkp } , 
\end{aligned}
\eeq
{\noindent{\bf{Until:}}} $\norm{\xkp-\xk}+\norm{\wkp-\wk} \leq \tol$.
\end{algorithm}
\end{minipage}
\end{center}

Define the following augmented variable $\zk$ and operators
\beq\label{eq:A-B-calV}
\zk \eqdef \begin{pmatrix} \xk \\ \wk \end{pmatrix} ,\enskip
\bmA \eqdef \begin{bmatrix} \partial R & L^* \\ -L & \partial J^* \end{bmatrix}  \qandq
\bcV \eqdef \begin{bmatrix} \Id_{n}/\gammaR & -L^* \\ -L & \Id_{m}/\gammaJ \end{bmatrix} ,
\eeq
where $\Id_{n}, \Id_{m}$ are the identity operators on $\bbR^n$ and $\bbR^m$, respectively.  
We have $\bmA$ is maximal monotone \cite{briceno2011monotone} and $\bcV$ is self-adjoint and $\nu$-positive definite for $\nu = \pa{ 1 - \mathsmaller{\sqrt{\gammaR\gammaJ\norm{L}^2}} } \min\ba{ \frac{1}{\gammaR} , \frac{1}{\gammaJ} }$ \cite{vu2011splitting,combettes2014variable}. 
The fixed-point characterization of \eqref{eq:pd} when $\tau = 1$ reads
\beq\label{eq:fbview-pd}
\zkp
= \pa{\bcV + \bmA}^{-1} {\bcV} (\zk)
= \pa{\bId + \bcV^{-1} \bmA}^{-1} (\zk) ,
\eeq
which is a special case of proximal point algorithm \cite{combettes2014variable,combettes2014forward}. 
We also refer to \cite{vu2011splitting,combettes2014forward} for more general form of Primal--Dual splitting methods.

%


\subsubsection{Linearization of Primal--Dual splitting}

Let $(\xsol, \wsol) \in \calX \times \calW$ be a saddle-point, the first-order optimality condition entails $-L^T \wsol  \in {\partial R(\xsol)} $ and $ L\xsol  \in {\partial J^*(\wsol)}$. We impose the following non-degeneracy condition
\beq\label{eq:ndc-pd}\tag{$\textrm{ND}_{_{\mathrm{PD}}}$}
-L^T \wsol  \in \ri\Pa{\partial R(\xsol)}  \qandq
L\xsol  \in \ri\Pa{\partial J^*(\wsol)} 	.
\eeq
%
%
%
Let $\MxR$ and $\MwJc$ be $C^2$-smooth manifolds around $\xsol$ and $\wsol$, respectively.

\begin{theorem}\label{thm:linearization-pd}
	For problem \eqref{eq:problem-pd} and the Primal--Dual splitting algorithm \eqref{eq:pd}, suppose assumptions \iref{PD:RJ}-\iref{PD:minimizers-nonempty} are true. If $\tau = 1$ and $\gammaR, \gammaJ$ are chosen such that $\gammaR\gammaJ \norm{L}^2 < 1$, then $(\xk,\wk) \to (\xsol, \wsol) \in \calX \times \calW$.  
	If moreover, $R \in \PSF{\xsol}{\MxR}$ and $J^* \in \PSF{\wsol}{\MwJc}$ are partly smooth and condition \eqref{eq:ndc-pd} holds, then for all $k$ large enough there exists a matrix $\mPD$ such that
	\beqn
	\zkp-\zk = \mPD (\zk-\zkm) + o(\norm{\zk-\zkm}) .
	\eeqn
\end{theorem}

See Section \ref{proof:trajectory-pd} for the proof and expression of $\mPD$. We refer to \cite{liang2018localPD} for detailed discussions on the local linear convergence of a class of Primal--Dual splitting methods.

\subsubsection{Trajectory of Primal--Dual splitting}

The trajectory of Primal--Dual splitting also depends on the explicit analysis of the spectrum of $\mPD$ which is only available when $R, J^*$ are locally polyhedral around the saddle point. 
Let $R \in \PSF{\xsol}{\MmR}, J \in \PSF{\wsol}{\MwJc}$, and denote $T_{\xsol}^R, T_{\wsol}^{J^*}$ the tangent spaces of $\MmR, \MwJc$ at $\xsol$ and $\wsol$, respectively. Denote $\barL \eqdef \PJc L \PR$. 

\begin{theorem}\label{thm:trajectory-pd-1}
	For problem \eqref{eq:problem-pd} and the Primal--Dual iteration \eqref{eq:pd}, assume Theorem \ref{thm:linearization-pd} holds. 
If, moreover, $R, J^*$ locally are polyhedral around $(\xsol, \wsol)$, then $\zkp-\zk = \mPD (\zk-\zkm) $ with
\[
\mPD = \begin{bmatrix} \Id_{n} & - \gammaR \barL^T \\ \gammaJ \barL  & \Id_{m} - (1+\tau)\gammaJ\gammaR \barL \barL^T \end{bmatrix} .
\] 
Moreover, $\mPD$ is block diagonalizable with the leading block being $2\times 2$ which corresponds to elliptical rotation. 
Then there exist $\utheta, \otheta$ such that eventually $\theta_k \in [\utheta, \otheta]$, and $\seq{\zk}$ is a Type III sequence. 
\end{theorem}


\begin{remark}$~$
\begin{itemize}
\item 
Let $\sigma$ be the leading eigenvalue of $\barL \barL^T$, then the leading block of the decomposition of $\mPD$ reads
\[
B = \begin{bmatrix} 1 & - \gammaR \sigma \\ \gammaJ \sigma & 1 - (1+\tau) \gammaR\gammaJ \sigma^2 \end{bmatrix} .
\]
Owing to Proposition \ref{prop:type-iii-a}, there exist some $\psi, \phi \in [0, \pi/2]$ and $l,s >0$ such that
\[
B
= \sfrac{1}{\sqrt{ 1- \tau \gammaJ\gammaR \sigma^2 }}
\begin{bmatrix} \cos(\psi) & - \sin(\psi) \\ \sin(\psi) & \cos(\psi) \end{bmatrix} 
\begin{bmatrix} \cos(\phi) & \tfrac{s}{l} \sin(\phi) \\ - \tfrac{l}{s}\sin(\phi) & \cos(\phi) \end{bmatrix} , 
\]
with 
\[
\begin{aligned}
\psi &= \mathrm{arccot}\bPa{ - \tfrac{ \gammaJ-\gammaR }{ (1+\tau) \gammaR\gammaJ \sigma } }  , \\
\phi &= \arccos\bPa{ \tfrac{(\gammaR+\gammaJ) \sigma \sin(\psi) + (2-(1+\tau) \gammaR\gammaJ \sigma^2)\cos(\psi) }{2{\sigma}} } \\
\sfrac{s}{l} &= \tfrac{1-(1+\tau) \gammaR\gammaJ \sigma^2}{\sin(\psi)\sin(\phi) {\sigma}} - \cot(\psi)\cot(\phi).
\end{aligned}
\]
This means that $B$ is a composition of circular rotation and elliptical rotation discussed in Proposition \ref{prop:comp-circular-elliptical}.  We can furthermore show (by  invoking  Proposition \ref{prop:elliptical-rotation} and \ref{prop:comp-circular-elliptical}) that $\utheta = \psi - \chiM, \otheta = \psi - \chim $ with 
$$
\textstyle \cos(\chiM)
= \frac{ \pa{ \frac{s}{2l} + \frac{l}{2s} } \cos(\phi) - \abs{ \frac{s}{2l} - \frac{l}{2s} } }{ \ssqrt{ \sin^2(\phi)  + \pa{\pa{ \frac{s}{2l} + \frac{l}{2s} } \cos(\phi) - \abs{ \frac{s}{2l} - \frac{l}{2s} } }^2 } }
\qandq
\cos(\chim)
= \frac{ \pa{ \frac{s}{2l} + \frac{l}{2s} } \cos(\phi) + \abs{ \frac{s}{2l} - \frac{l}{2s} } }{ \ssqrt{ \sin^2(\phi)  + \pa{\pa{ \frac{s}{2l} + \frac{l}{2s} } \cos(\phi) + \abs{ \frac{s}{2l} - \frac{l}{2s} } }^2 } }   .
$$


\item 
Similar to the case of Douglas--Rachford splitting, the trajectory of Primal--Dual, when both $R, J^{*}$ are locally polyhedral, is obtained via the explicit analysis of the spectrum of $\mPD$ which is available when $R, J^{*}$ are general partly smooth functions. 
%
\end{itemize}
\end{remark}

\begin{example}
We continue using the problem \eqref{eq:R-affine} in Example \ref{eg:exp-dr} to demonstrate the trajectories of Primal--Dual splitting method. 
The observations are shown in Figure \ref{fig:trajectory-pd}, 
\begin{itemize}
\item For $\ell_1$-norm, $\theta_k$ eventually oscillates in an interval which complies with our result in Theorem \ref{thm:trajectory-pd-1}. 
\item For nuclear norm, though it is not covered by our result as nuclear norm is not polyhedral, locally the value of $\theta_k$ also oscillates. 
\end{itemize}
\end{example}

\vspace{-4mm}

\begin{figure}[!ht]
\centering
\subfloat[$\ell_1$-norm]{ \includegraphics[width=0.425\linewidth]{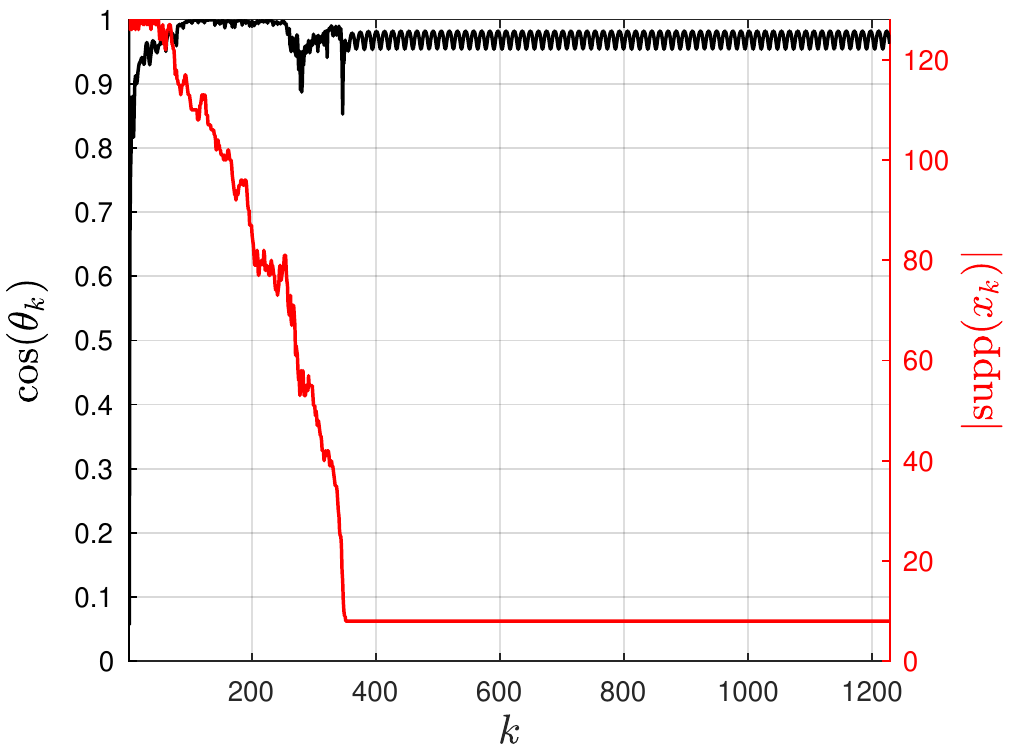}  }     {\hspace{6pt}}
\subfloat[Nuclear norm]{ \includegraphics[width=0.425\linewidth]{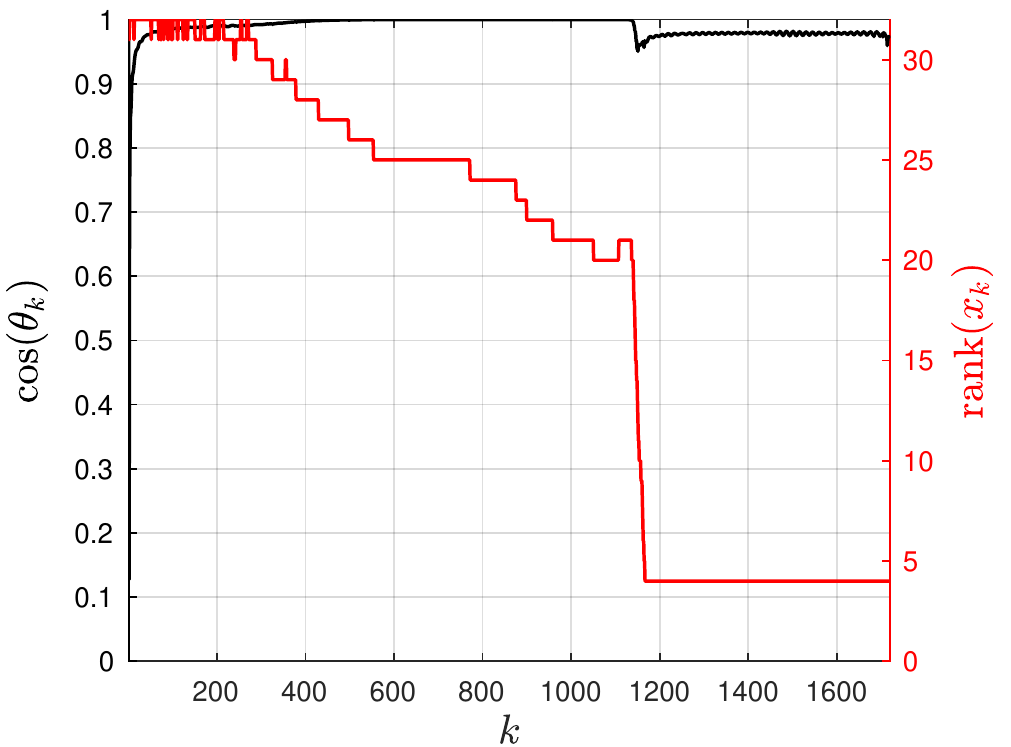}  }     \\
\caption{Finite activity identification and property of $\theta_k$ for Primal--Dual splitting method for solving affine constrained problem. }
\label{fig:trajectory-pd}
\end{figure}



%
%
%
%



\section{The failure of inertial technique}\label{sec:failure}

The trajectory results from previous section provide a geometric explanation why inertial acceleration works for (proximal) gradient descent methods but not the others. For (proximal) gradient descent, as the trajectory of the generated sequence eventually approximates a straight line, the direction of $\xk-\xkm$ points towards the solution, hence moving certain distance along the inertial direction provides acceleration. However, when the trajectory of the generated sequence is a spiral, as for the cases of Douglas--Rachford and Primal--Dual splitting, the direction of $\zk-\zkm$ does not point toward the solution, hence fail to provide acceleration.

In this section, we consider Douglas--Rachford and two different problems to demonstrate the outcomes of inertial acceleration. 
We show that the performance of inertial Douglas--Rachford is both problem and parameter dependent. 
Specializing the inertial scheme \eqref{eq:ifom} to the case of Douglas--Rachford splitting, we obtain an inertial Douglas--Rachford splitting scheme described in Algorithm \ref{alg:idr}.

\begin{center}
	\begin{minipage}{0.975\linewidth}
		\begin{algorithm}[H]
			\caption{An inertial Douglas--Rachford splitting} \label{alg:idr}
			\KwIn{$\gamma > 0$.}
			{\noindent{\bf{Initial:}}} $z_0 \in \bbR^n,~ \bar{z}_{0} = z_{0},~ x_0 = \prox_{\gamma J} (\bar{z}_0)$\;
			{\noindent{\bf{Repeat:}}} 
				\beq\label{eq:idr}
				\begin{aligned}
					\ukp &= \prox_{\gamma R}\pa{2\xk - \zbark} , \\
					\zkp &= \zbark + \ukp - \xk , \\
					\zbarkp &= \zkp + \ak (\zkp - \zk) , \\
					\xkp &= \prox_{\gamma J} (\zbarkp) ,
				\end{aligned}
				\eeq
			{\noindent{\bf{Until:}}} $\norm{\zkp-\zk} \leq \tol$.
		\end{algorithm}
	\end{minipage}
\end{center}



\vspace{-3mm}

\subsection{Feasibility problem}

We first consider a feasibility problem of two subspaces.  
For simplicity, consider the problem in $\bbR^2$: let $T_1, T_2 \subset \bbR^2$ be two intersecting lines. The problem of finding the common point of $T_1, T_2$ can be written as 
\beq\label{eq:feasibility}
\min_{x \in \bbR^2}~ \iota_{T_1}(x) + \iota_{T_2} (x) . 
\eeq
As the proximal mapping of indicator functions is projection, the above problem can be easily handle by  Douglas--Rachford splitting method. 

For the inertial Douglas--Rachford \eqref{eq:idr}, we consider $\ak \equiv 0.3$ and compare it with the standard Douglas--Rachford splitting scheme \eqref{eq:dr}. The comparison is provided in Figure \ref{fig:failure_idr}, with the left figure showing the convergence speed of $\norm{\zk-\zkm}$ and right figure the trajectory of sequence $\seq{\zk}$. 
We observe that
\begin{itemize}
	\item The inertial Douglas--Rachford with $\ak\equiv0.3$ (\emph{gray} line) is slower than the standard scheme (\emph{black} line). Moreover, it can be shown that for this feasibility example, the inertial Douglas--Rachford is slower as long as $\ak > 0$; See Section A of \cite{PoonLiang2019b}. 
	\item The slow performance of inertial Douglas--Rachford can also be visualized by the trajectory of the sequence $\seq{\zk}$. For inertial Douglas--Rachford, it increases the length of the trajectory of $\seq{\zk}$, see the difference between \emph{gray} and \emph{black} spirals. 
\end{itemize}


\begin{figure}[!ht]
\centering
\subfloat[Convergence of $\norm{\zk-\zkm}$]{ \includegraphics[width=0.45\linewidth]{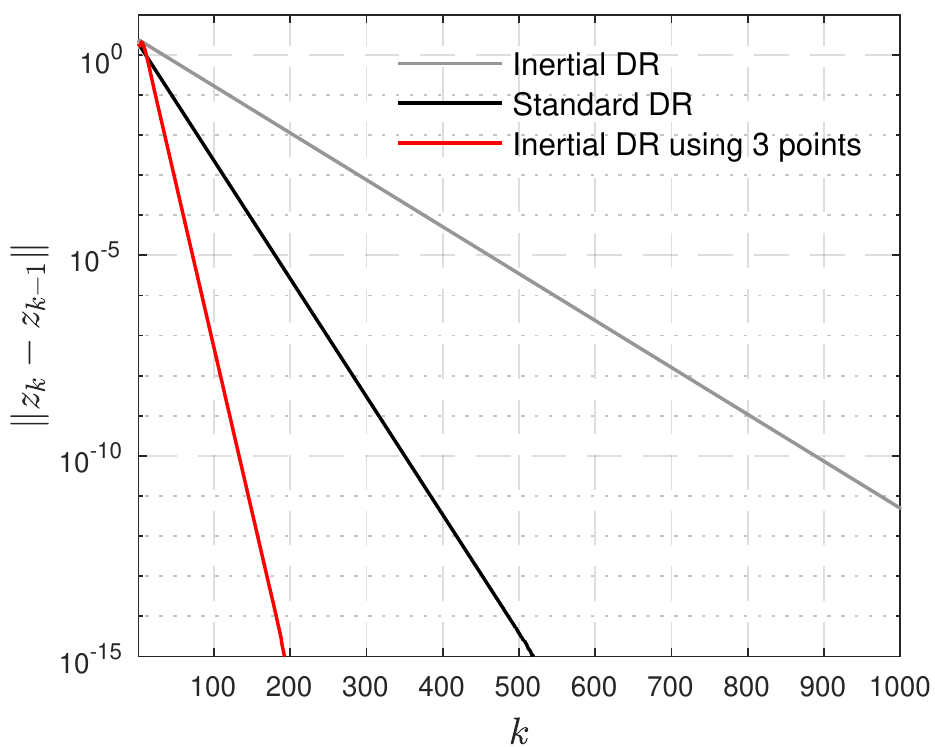} }     {\hspace{6pt}}
\subfloat[Trajectory of $\seq{\zk}$]{ \includegraphics[width=0.45\linewidth]{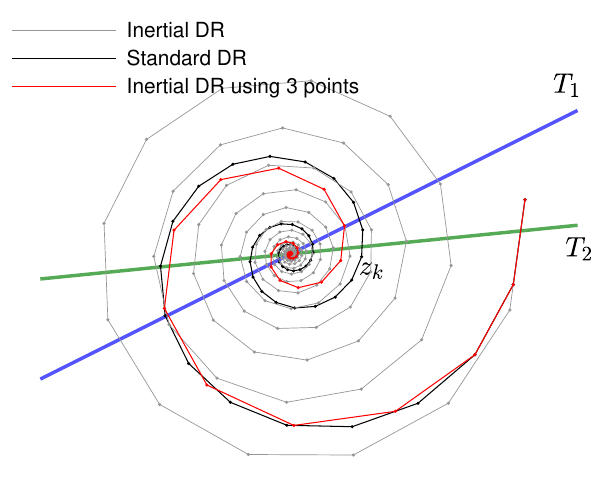} }    \\
\caption{Comparison between standard Douglas--Rachford and inertial Douglas--Rachford. (a) convergence of $\norm{\zk-\zkm}$; (b) trajectory of sequence $\seq{\zk}$.}
\label{fig:failure_idr}
\end{figure}


The above comparisons, is quite different from the improvement of \eg heavy-ball method over gradient descent, which is an evidence that the trajectory of the sequence affects the outcome of inertial acceleration.
We also remark that in \cite[Chapter 4]{liang2016thesis}, the author suggested to use more than two points for computing $\zbark$ in \eqref{eq:idr}, for example the following three-point approach
\beq\label{eq:three_point}
\zbark = \zk + \ak (\zk - \zkm) + \bk (\zkm - \zkmm)   .
\eeq
Particularly, with $\ak > 0$ and $\bk < 0$, the above choice can provide acceleration. 
In Figure \ref{fig:failure_idr}, we also provide such a test with $(\ak,\bk) = (0.6, -0.3)$, the corresponding convergence observation of $\norm{\zk-\zkm}$ is shown in \emph{red} line which is faster than the standard Douglas--Rachford splitting scheme. Trajectory-wise, the length is also shorter than that of the standard Douglas--Rachford. 

\vspace{-1mm}

\begin{remark}$~$
\begin{itemize}
\item 
The observation is not limited to the simple feasibility case, but rather a class of problems. For example, as long as the problem is (locally) polyhedral around the solution, the above observation can be obtained. 

\item 
The difference between the inertial scheme \eqref{eq:idr} and \eqref{eq:three_point} implies that when the trajectory is a spiral, only $\zk-\zkm$ is not enough to estimate or fit the direction that $\zk$ travels, while the three-point scheme can solve the problem. 
However, the problem with the three-point scheme is that, the value of $\bk$ needs to be negative, and in general there is no good way to determine the choices of $\ak,\bk$, let alone theoretical acceleration guarantees. This is one motivation behind the adaptive acceleration in Section \ref{sec:lp}. 
\end{itemize}
\end{remark}

\subsection{LASSO problem}

The second problem we consider is the LASSO problem
\beq\label{eq:lasso}
\min_{x\in\bbR^n}~ \mu\norm{x}_{1} + \sfrac{1}{2} \norm{A x - f}^2  , 
\eeq
where $A \in \bbR^{m\times n}$ is random Gaussian matrix with $m<n$.

Since $\frac{1}{2} \norm{A x - b}^2$ is $C^2$ smooth, we know from Theorem \ref{thm:trajectory-dr-2} that the trajectory of $\seq{\zk}$ is determined by the choice of $\gamma$. As a result, two different choices of $\gamma$, $\gamma \in \ba{ \frac{0.9}{\norm{A}^2}, \frac{10}{\norm{A}^2} }$, are considered. For each $\gamma$, four different choices of $\ak$ are chosen: $\ak \equiv 0, \ak \equiv 0.3, \ak \equiv 0.7$ and $\ak = \frac{k-1}{k+3}$. Note that the last choice of $\ak$ corresponds to the Nesterov's scheme of \cite{su2014differential} and the FISTA scheme of \cite{chambolle2015convergence}.

\begin{figure}[!ht]
\centering
\subfloat[$\gamma=\frac{10}{\norm{A}^{2}}$: $\cos(\theta_k)$ and $|\supp(\xk)|$]{ \includegraphics[width=0.45\linewidth]{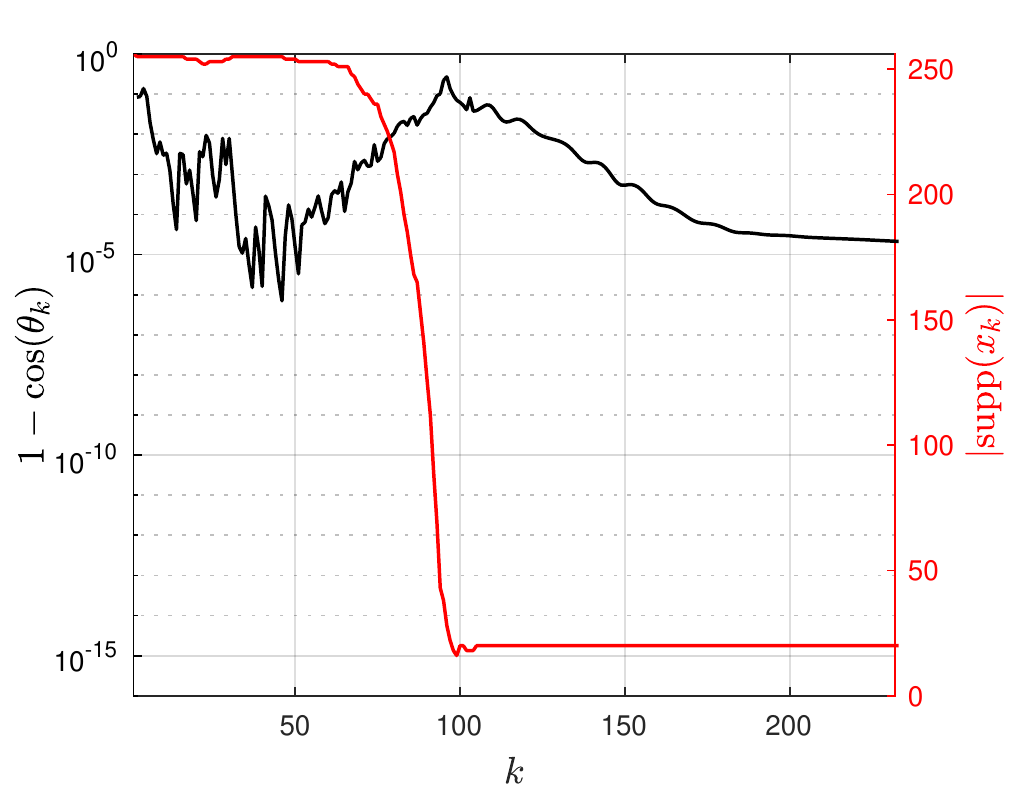} }     {\hspace{6pt}}
\subfloat[$\gamma=\frac{0.9}{\norm{A}^{2}}$: $\cos(\theta_k)$ and $|\supp(\xk)|$]{ \includegraphics[width=0.45\linewidth]{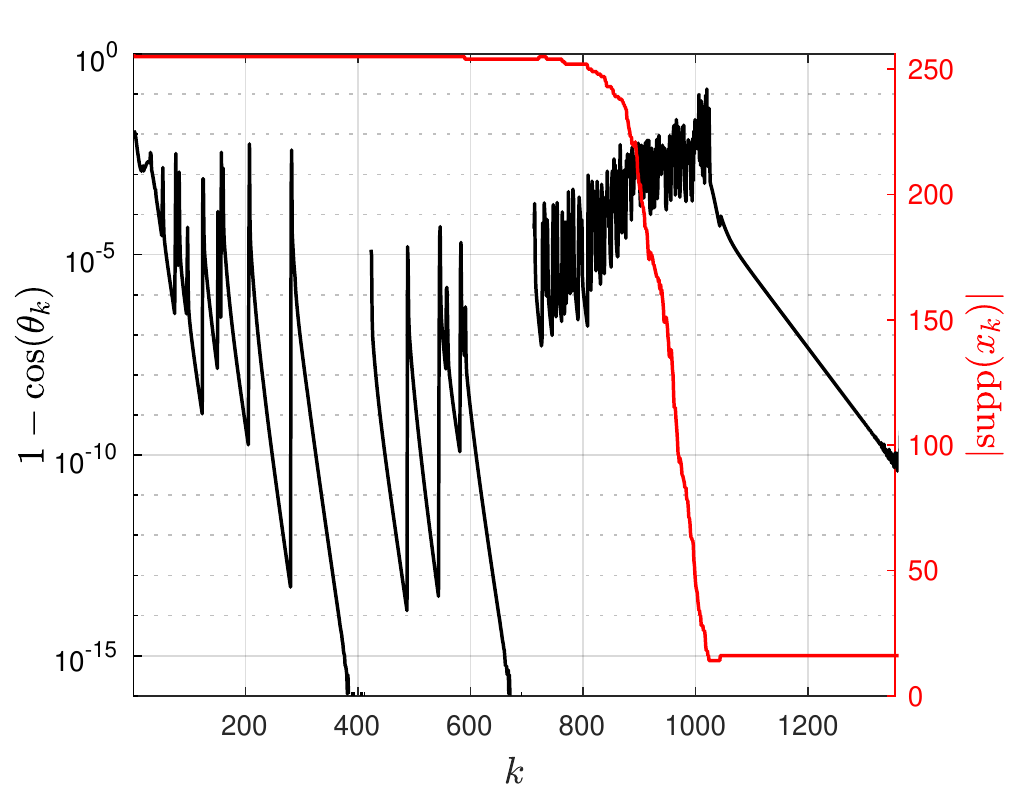} }    \\[-2mm]
\subfloat[$\gamma=\frac{10}{\norm{A}^{2}}$: Convergence of $\norm{\zk-\zkm}$]{ \includegraphics[width=0.45\linewidth]{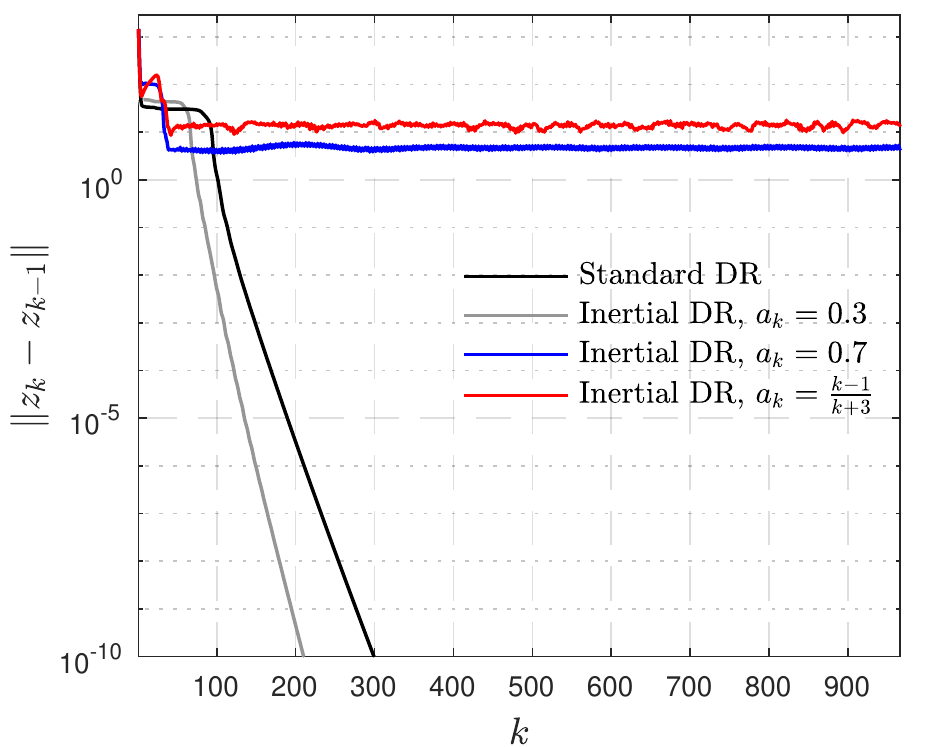} }     {\hspace{6pt}}
\subfloat[$\gamma=\frac{0.9}{\norm{A}^{2}}$: Convergence of $\norm{\zk-\zkm}$]{ \includegraphics[width=0.45\linewidth]{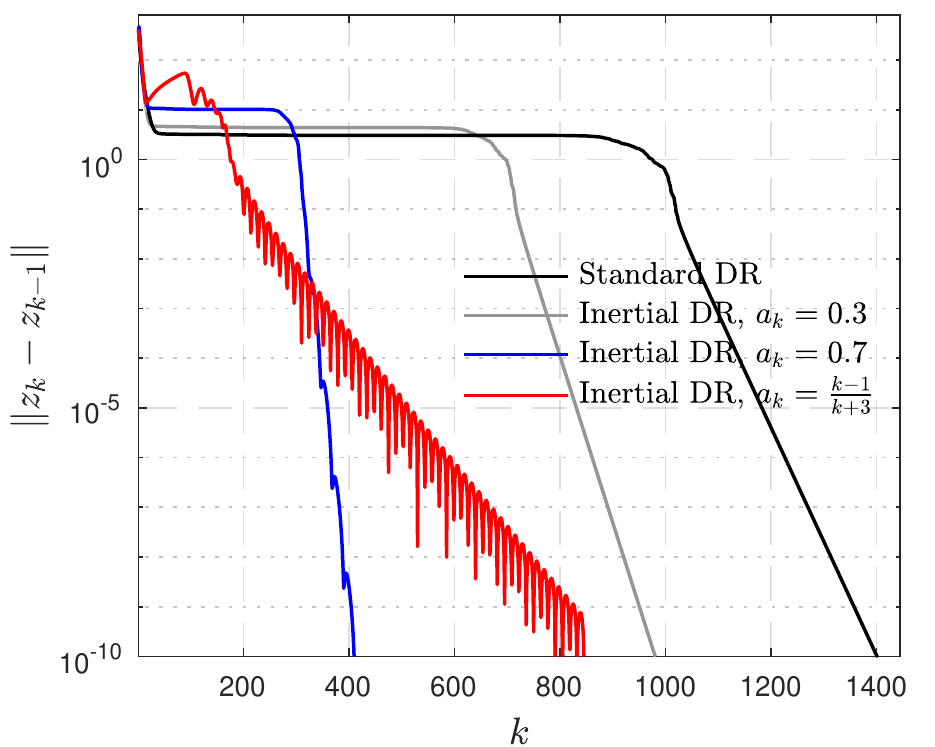} }    \\
\caption{Comparison between standard Douglas--Rachford and inertial Douglas--Rachford. (a) $\cos(\theta_k)$ and $\supp(\xk)$ for $\gamma = \frac{10}{\norm{K}^2}$; (b) $\cos(\theta_k)$ and $\supp(\xk)$ for $\gamma = \frac{0.9}{\norm{K}^2}$; (c) convergence of $\norm{\zk-\zkm}$ for $\gamma = \frac{10}{\norm{K}^2}$; (b) trajectory of sequence $\seq{\zk}$ for $\gamma = \frac{0.9}{\norm{K}^2}$.}\vspace{-6mm}
\label{fig:success_idr}
\end{figure}

For the numerical example, we consider $K \in \bbR^{64\times 256}$ and $\mu = 2$, $f$ is the measurement of an $\xob$ which is $8$-sparse under small additive white Gaussian noise. The results are shown in Figure \ref{fig:success_idr}, 
\begin{itemize}[leftmargin=2em]
	\item Case $\gamma=\frac{10}{\norm{A}^{2}}$: in Figure \ref{fig:success_idr}~(a), the \emph{red} line shows the support identification of the iterates $\xk$, and after support identification which is about $k=150$, the angle $\theta_k$ is \emph{not} converging to $0$. 
	\item Case $\gamma=\frac{0.9}{\norm{A}^{2}}$: 
	In Figure \ref{fig:success_idr}~(b), from about $k=1600$, $\theta_k$ is converging to $0$. 
\end{itemize}
Then in terms of the performances of the inertial schemes, 
\begin{itemize}[leftmargin=2em]
	\item Case $\gamma=\frac{10}{\norm{A}^{2}}$: in Figure \ref{fig:success_idr}~(c), out of three inertial schemes, only the one with $\ak\equiv 0.3$ is convergent. This is due to that fact that the trajectory of $\seq{\zk}$ is a spiral for $\gamma=\frac{10}{\norm{A}^{2}}$. 
	\item Case $\gamma=\frac{0.9}{\norm{A}^{2}}$: All choices of $\ak$ work since $\seq{\zk}$ eventually forms a straight line. 
	Among these four choices of $\ak$, $\ak=0.7$ is the fastest, with the FISTA choice a bit slower. 
\end{itemize}
It can be observed that, under the two choices of $\gamma$, the standard Douglas--Rachford is faster for $\gamma=\frac{10}{\norm{A}^{2}}$ than $\gamma = \frac{0.9}{\norm{A}^2}$. However, we remark that our main focus here is to demonstrate how the trajectory of $\seq{\zk}$ affects the outcome of inertial acceleration. 

\subsection{A geometric interpretation on the failure of inertial}\label{sec:geo_inertial}

In this part, we provide a geometric explanation on how the sequence trajectory  affects the outcome of inertial acceleration, similar analysis can also be obtained for over-relaxation. 
%
To this end, consider the angle $\vartheta_k$ defined by $\vartheta_k \eqdef \angle(\zk-\zkm, \zsol-\zk) = \arccos\Pa{ \frac{\iprod{\zk-\zkm}{\zsol-\zk}}{\norm{\zk-\zkm}\norm{\zsol-\zk}} }$. 
The motivation of considering this angle is shown below in Figure \ref{fig:inertial_triangle}: let $\zbark = \zk + \ak(\zk-\zkm)$ with $\ak > 0$
\begin{itemize}
\item If $\vartheta_k$ is \emph{acute}, then it can be shown that $\norm{\zsol-\zbark} < \norm{\zsol-\zk}$ as long as $$ \ak < \sfrac{ 2\cos(\vartheta_k)\norm{\zsol-\zk} }{\norm{\zk-\zkm}}.$$
When $ \ak = \sfrac{ \cos(\vartheta_k)\norm{\zsol-\zk} }{\norm{\zk-\zkm}} $, $\zbark$ is the closest point to $\zsol$ in the radial line $a \mapsto \zk + a(\zk-\zkm)$ with $a\geq 0$.  
\item When $\cos(\vartheta_k)<0$, 
from the above equation we get that $\norm{\zsol-\zbark} < \norm{\zsol-\zk}$ holds only for non-positive $\zk$, which means extrapolate only slows down the convergence.  

\end{itemize}


\begin{figure}[!ht]
	\centering
	\includegraphics[width=0.425\linewidth]{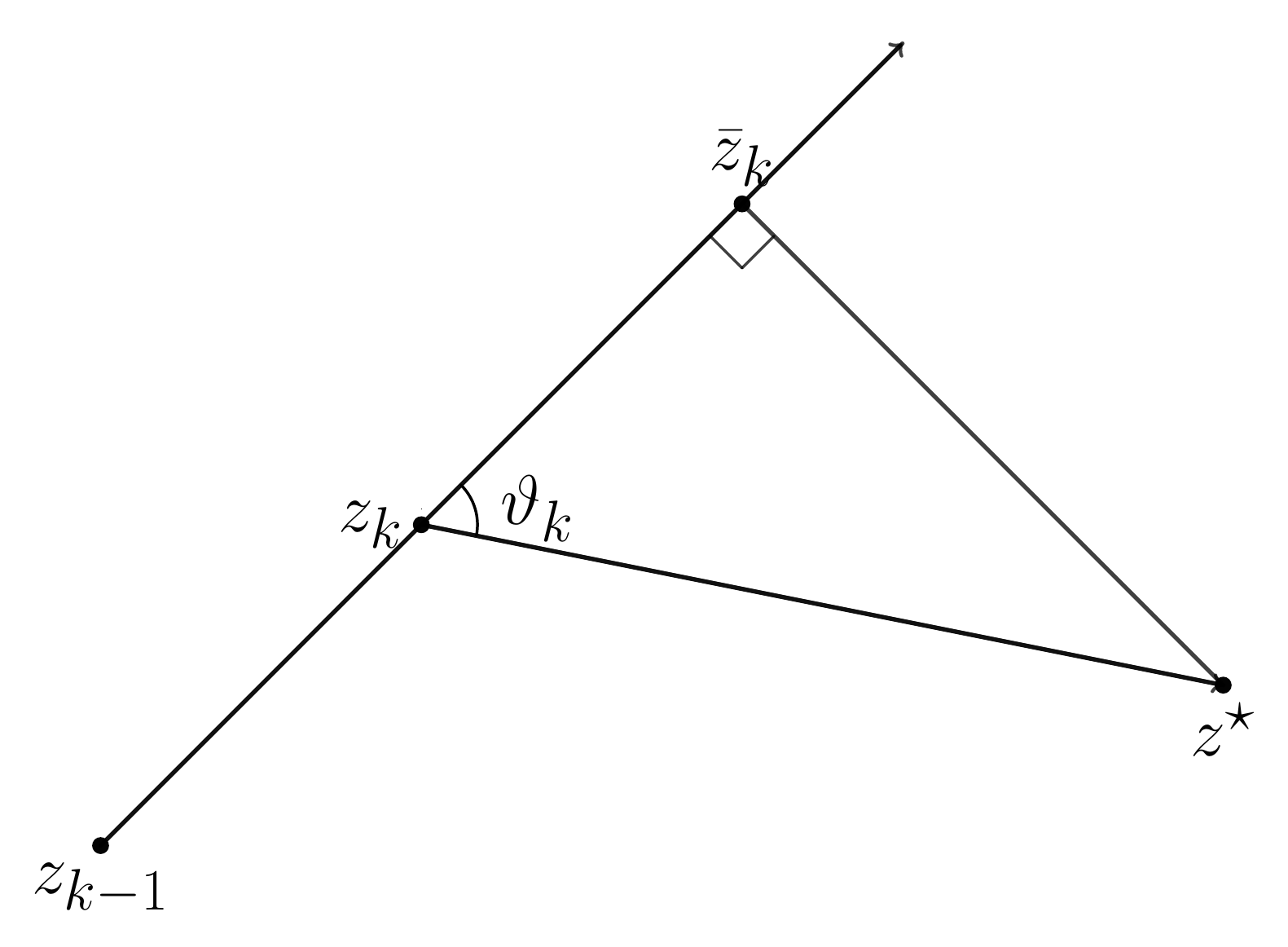}      \\
	\caption{Illustration between inertial direction and the direction to the limiting point in $\bbR^2$.} 
	\label{fig:inertial_triangle}  
\end{figure}

In the following, we consider the following linear system in $\bbR^2$:
\[
\zk = M \zkm 
\]
with $\zk$ converging to some $\zsol$, and study the property of $\vartheta_k$, under different $M$ that corresponds to the three types of trajectories. 

\begin{figure}[!ht]
	\centering
	\subfloat[Eventual straight line: $(\sigma_1, \sigma_2) = 0.9(1, 0.6)$]{ \includegraphics[width=0.425\linewidth]{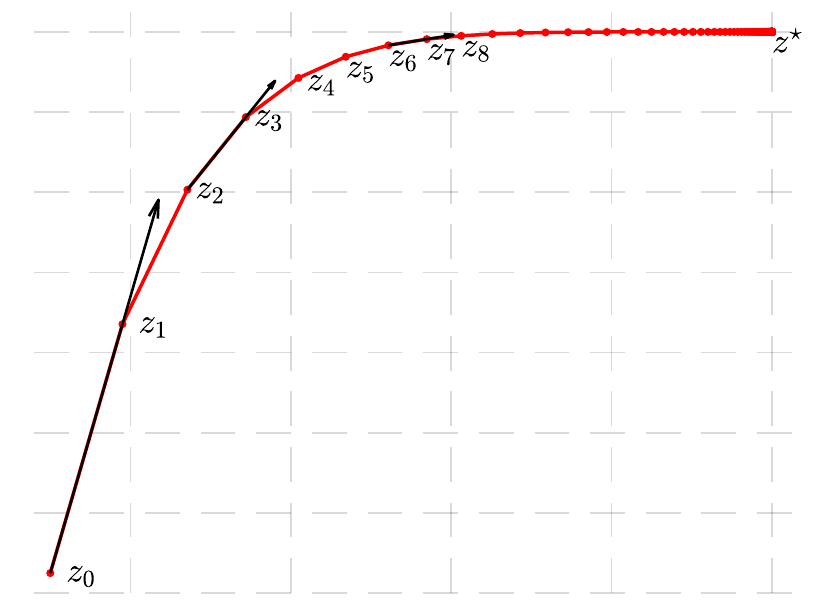} }     {\hspace{6pt}}
	\subfloat[Eventual straight line: $(\sigma_1, \sigma_2) = 0.9(1, - 0.6)$]{ \includegraphics[width=0.425\linewidth]{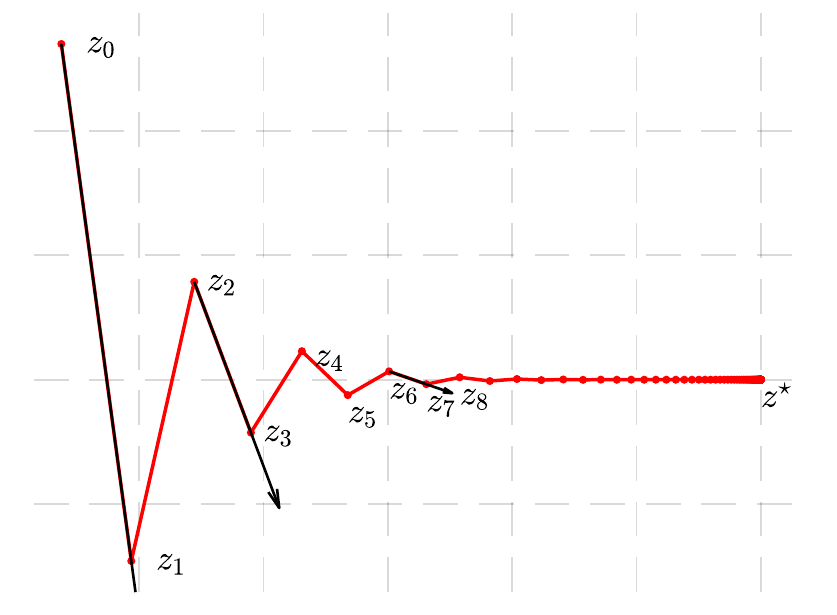} }     \\[-1mm]
	\subfloat[Logarithmic spiral]{ \includegraphics[width=0.425\linewidth]{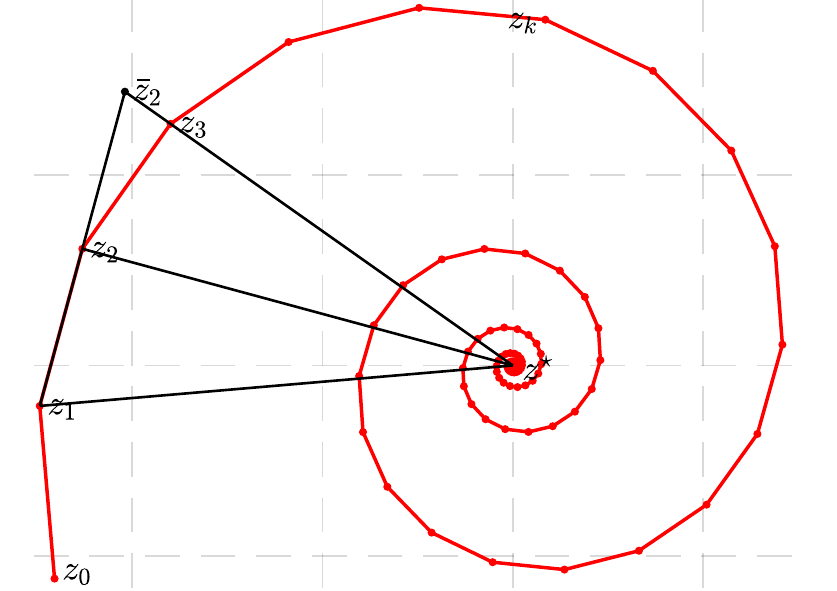} }  {\hspace{6pt}}
	\subfloat[Elliptical spiral]{ \includegraphics[width=0.425\linewidth]{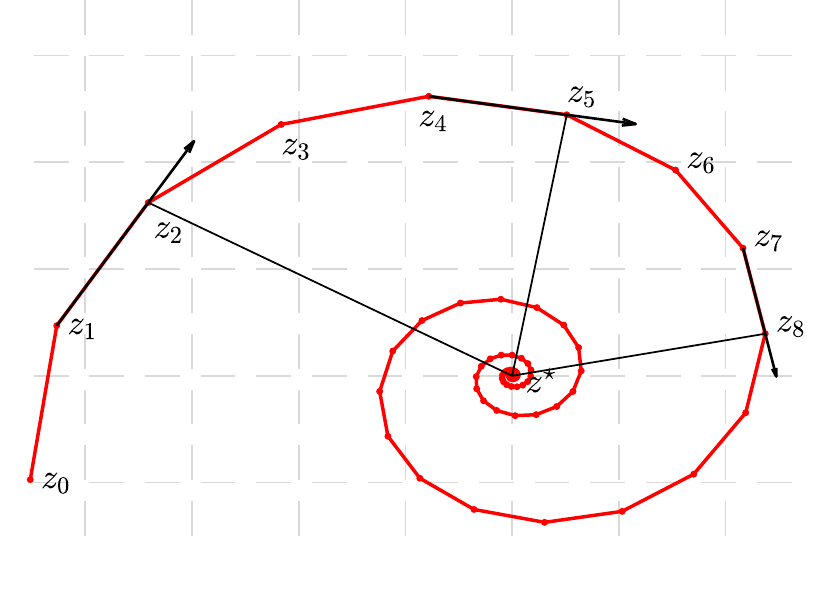} }      \\
	\caption{Graphical illustration of the direction of $\zk-\zkm$ for different types of sequence trajectory.}
	\label{fig:illustration_of_trajectory_and_inertial_type}
\end{figure}


\paragraph{Straight line trajectory} 
Let $U$ be a unitary $2\times 2$ matrix and $\sigma_1,\sigma_2$ be such that $0<\abs{\sigma_2} < \sigma_1 <1$, and let $M$ of the form
\[
M = 
U \begin{bmatrix} \sigma_1 & 0 \\ 0 & \sigma_2 \end{bmatrix} U^T . 
\]
It is immediate that $\zsol = 0$. 
Denote $\yk = U^{T} \zk$, we have
\[
\begin{aligned}
\yk 
= \begin{bmatrix} \sigma_1 & 0 \\ 0 & \sigma_2 \end{bmatrix} \ykm 
= \begin{bmatrix} \sigma_1 & 0 \\ 0 & \sigma_2 \end{bmatrix}^{k} y_{0} 
= \begin{bmatrix} \sigma_1 & 0 \\ 0 & 0 \end{bmatrix}^{k} y_{0} + \begin{bmatrix} 0 & 0 \\ 0 & \sigma_2 \end{bmatrix}^{k} y_{0}
= \sigma_{1}^{k} \begin{pmatrix} a \\ \eta^k b \end{pmatrix} ,
\end{aligned}
\] 
where $\eta = \sigma_2 /\sigma_1 < 1$. Assume that $y_0 = \begin{pmatrix} a \\ b \end{pmatrix}$, then
\[
\begin{aligned}
\cos( \vartheta_k )
= { \sfrac{\iprod{\zk-\zkm}{-\zk}}{\norm{\zk-\zkm}\norm{-\zk}} }
= { \sfrac{\iprod{\yk-\ykm}{-\yk}}{\norm{\yk-\ykm}\norm{-\yk}} }  
\textstyle = \sfrac{(1 - \sigma_1) a^2 + (1 - \sigma_1\eta) \sigma_1\eta^{2k-1} b^2}{\sqrt{(\sigma_1-1)^2a^2+(\sigma_1\eta-1)^2\eta^{2(k-1)}} \sqrt{a^2 + \sigma_1^2\eta^{2k}}} . 
\end{aligned}
\]
Let $k \to +\infty$ we get $\cos( \vartheta_k ) \to 1$ which means $\vartheta_{k} \to 0$.

The above result implies that, eventually $\zk-\zkm$ points towards the limiting point $\zsol$. Therefore, moving certain distance along $\zk-\zkm$ is useful to improve the convergence speed. 
To demonstrate the above result, we consider two different choices of $(\sigma_1, \sigma_2)$ that $(\sigma_1, \sigma_2) = 0.9(1, \pm 0.6)$. 
The trajectory of $\seq{\zk}$ and the property of $\vartheta_k$ are shown in Figure \ref{fig:illustration_of_trajectory_and_inertial_type} (a) and (b). It can be observed from both figures that: from the beginning, $\zk-\zkm$ is not pointing towards $\zsol$ but eventually almost directly to $\zsol$.




\paragraph{Logarithmic spiral trajectory}

For logarithmic spiral, we can show that inertial always slows down the convergence. For this case, we have $M$ of the form
\[
M = \cos(\theta) \begin{bmatrix} \cos(\theta) & \sin(\theta) \\ -\sin(\theta) & \cos(\theta) \end{bmatrix} ,
\]
for some $\theta \in ]0, \pi/2[$. 
For the sequence $\seq{\zk}$, we also have $\zk \to 0$ as $\rho(M)=\cos(\theta) < 1$. 
Given any $k$, we have $\norm{\zk} = \norm{M\zkm} = \norm{\zkm} \cos(\theta)$, then consider the inner product
\[
\begin{aligned}
\iprod{\zk-\zkm}{\zk}
= \norm{\zk}^2 - \iprod{\zkm}{\zk}
&= \norm{\zkm}^2 \cos^2(\theta) - \norm{\zkm}\norm{\zk} \cos(\theta) \\
&= \norm{\zkm}^2 \cos^2(\theta) - \norm{\zkm}^2 \cos^2(\theta) = 0 ,
\end{aligned}
\]
which means $\cos(\vartheta_k) \equiv 0$ and $\vartheta_k \equiv \pi/2$. 

A graphic illustration is provided in Figure \ref{fig:illustration_of_trajectory_and_inertial_type} (c). Let $k = 2$, and $\zbar_{2} = 2z_{2} - z_{1}$. It can be proved that the three points $z_{1}, \zbar_{2}$ and $\zsol$ form an isosceles triangle with $\norm{z_1-\zsol}=\norm{\zbar_2-\zsol}$. Moreover, $\zbar_2,z_3$ and $\zsol$ are in the same line. 
This in turn indicates that for all the point $z$ in the segment of $z_{2}$ and $\zbar_2$, we have $$\norm{z_2-\zsol} < \norm{z-\zsol} . $$
As a result, applying inertial will slows down the performance. 



\paragraph{Elliptical spiral trajectory}
For elliptical spiral, we consider the following form of $M$
\[
M = \cos(\theta) \begin{bmatrix} \cos(\theta) & \tfrac{l}{s}\sin(\theta) \\ -\tfrac{s}{l}\sin(\theta) & \cos(\theta) \end{bmatrix} ,
\]
for some $\theta \in ]0, \pi/2[$ and $l,s > 0$. 
The property of $\vartheta_k$ becomes more complicated for the elliptical spiral, as $\vartheta_k$ varies in an interval $[\underline{\vartheta}, \overline{\vartheta}] \subset ]0, \pi[$. Though the expressions of $\underline{\vartheta}, \overline{\vartheta}$ can be obtained explicitly based on the result of Section \ref{subsec:type-iii}, here we only provide descriptive explanation. 

As we can observe from Figure \ref{fig:illustration_of_trajectory_and_inertial_type} (d), that the angle $\vartheta_k$ varies in an interval $[\underline{\vartheta}, \overline{\vartheta}]$ where $\underline{\vartheta} < \pi/2$ and $\overline{\vartheta} > \pi/2$. This means that the direction $\zk-\zkm$ only points towards $\zsol$ for \emph{acute} $\vartheta_k$. 
In turn, inertial provides acceleration when $k$ are such that $\vartheta_k$ is acute and does not for the others. As a result, the overall performance of inertial is not clear in general.



\begin{remark}

 {
Although the above discussion is in $\RR^2$, the behavior describes the asymptotic behavior of the fixed point sequences generated by  Forward--Backward-splitting and Douglas--Rachford splitting for 2 polyhedral terms and Primal--Dual splitting for 2 polyhedral terms are  (up to orthogonal transformation) block diagonal matrices. Indeed, as we show in the appendix,  the corresponding  linearization matrices in these three cases are  (up to orthogonal transformation) block diagonal matrices, where each block is a $2\times 2$ matrix. Asymptotically, the behavior of the associated fixed point sequences will be driven by the leading block if the leading eigenvalue is unique. 
}

\end{remark}

\section{A$\!^2$FoM: adaptive acceleration for first-order methods}\label{sec:lp}

The trajectory property implies that, the sequence generated by first-order method eventually settles onto a regular path, \ie straight line or spiral. In turn, we can use such regularity to design adaptive acceleration for first-order methods, which is called ``A$^2$FoM'' and described in Algorithm \ref{alg:a2fom}.

\subsection{Trajectory following adaptive acceleration}

We describe how to use the regularity of the trajectory to design a linear prediction scheme for acceleration. 
Recall the general inertial scheme \eqref{eq:efom} for first-order method
\beqn
\begin{aligned}
	\zbark &= \calE(\zbarkm, \zk, \zkm, ...) , \\
	\zkp &= \calF(\zbark, \zk, \zkm, ...)  .
\end{aligned}
\eeqn
From our discussion in the last section, to provide acceleration, the extrapolation operator $\calE$ should be able to adapt itself to the trajectory of the sequence $\seq{\zk}$. 
To this end, we propose a \emph{trajectory following linear prediction strategy}, which locally fits the trajectory of the sequence $\seq{\zk}$ and predict the future points. 
The basic idea of linear prediction is: let $q \in \bbN_+$ be a positive integer, given $\{z_{k-j}\}_{j=0}^{q+1}$ and $v_j\eqdef z_j - z_{j-1}$, forecast the future iterates by considering how the past directions $v_{k-1},\ldots, v_{k-q}$ approximate the latest direction $\vk$. 
More precisely, 
\begin{itemize}
	\item First use $\{v_{k-j}\}_{j=1}^{q}$ to represent $\vk$, which is a least square problem. Denote $V_{k-1} \eqdef \begin{bmatrix} v_{k-1} , \dotsm , v_{k-q}  \end{bmatrix} \in \bbR^{n \times q}$, and let 
\[
c_k \in \Argmin_{c\in \RR^q}\norm{V_{k-1} c - v_k}^2 = \norm{ \msum_{j=1}^q c_j v_{k-j} - v_k}^2  . 
\]
Then we have $\vk \approx V_{k-1} c_k$, and the representation if perfect if $v_k \in \ran\pa{V_{k-1}}$. 

	\item Suppose we know $\zkp$, then follow the first step we have $\vkp \approx V_{k} c_{k+1}$.  Since the trajectory locally is regular, we have $c_{k+1} \approx c_{k}$, this means we have the approximation $\vkp \approx V_k c_k$. As a result, we obtain an approximation of $\zkp$ which is $$ z_{k+1} \approx \bar z_{k,1} \eqdef z_k + V_k c_{k} . $$ 
    
    \item By iterating the second step $s$ times, we obtain an approximation of $z_{k+s}$ which is $\bar z_{k,s}\approx z_{k+s}$.
	
\end{itemize}

%
\begin{figure}[!ht]
	\centering
	\includegraphics[width=0.375\linewidth]{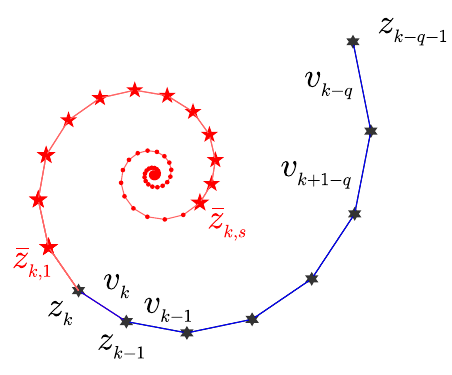} \\
	\caption{Illustration of linear prediction.}
	\label{fig:illustration_lp}
\end{figure}
%

\vgap

In Figure \ref{fig:illustration_lp}, we provide a graphical illustration of linear prediction: black dots are the given $q+2$ points, and red star points are the outputs of linear prediction from $1$-step prediction to $s$-step prediction. If we run the prediction until $s=+\infty$, we obtain a complete spiral.

It can be observed that the above procedure is totally linear, therefore we can derive a simple matrix representation for linear prediction. 
Given a vector $c \in \bbR^{q}$, define the mapping $H$ by
\begin{equation}\label{eq:coeff_mtx}
H(c) 
= 
\left[
\begin{array}{c|c}
c_{1:q-1}  & \Id_{q-1} \\
c_{q} & 0_{1,q-1}
\end{array}
\right]
\in \bbR^{q \times q} .
\end{equation}
Let $C_k = H(c_k)$, note that $V_{k} = V_{k-1} C_k$. Denote $\bar V_{k,0} \eqdef V_k$ and for $s \geq 1$,  define
\[
\bar{V}_{k,s} \eqdef \bar{V}_{k,s-1} C_k \eqdef V_{k} C_k^s ,
\]
where $C_k^s$ is the power of $C_k$. 
Let $(C)_{(:, 1)}$ be the first column of matrix $C$, then
\beq\label{eq:zk_s}
\begin{aligned}
\zbar_{k,s} 
= \zk + \msum_{i=1}^{s} (\bar{V}_{k,i})_{(:, 1)} 
= \zk + \msum_{i=1}^{s} V_{k} (C_{k}^{i})_{(:, 1)}  
&= \zk + V_{k} \bPa{ \msum_{i=1}^{s} C_{k}^i }_{(:, 1)}  ,
\end{aligned}
\eeq
which is the desired trajectory following extrapolation scheme. Now define the extrapolation parameterized by $s, q$ as 
\[
\calE_{s, q}(\zk, \dotsm, z_{k-q-1}) 
\eqdef  V_{k} \bPa{ \msum_{i=1}^{s} C_{k}^i }_{(:, 1)}  ,
\]
we obtain the following trajectory following adaptive acceleration for first-order method. 

\vspace{-1mm}

\begin{center}
\begin{minipage}{0.975\linewidth}
\begin{algorithm}[H]
\caption{A$\!^2$FoM: Adaptive Acceleration for First-order Methods} \label{alg:a2fom}
\KwIn{Let $s\geq1, q \geq 1$ be integers.}
{\noindent{\bf{Initial:}}} Let $\zbar_0 = z_0 \in \bbR^n$ and $V_{0} = 0 \in \bbR^{n \times q}$. \\ 
{\noindent{\bf{Repeat:}}} 
\begin{itemize}[leftmargin=2em]
\item \textrm{If $\textrm{mod}(k, q+2)=0$: } \textrm{Compute $C_k$ as described above, if $\rho(C_k)<1$:}  
			$$\zbark = \zk + \ak \calE_{s, q}(\zk, \dotsm, z_{k-q-1}).$$
\item \textrm{If $\textrm{mod}(k, q+2)\neq0$: } $ \zbark = \zk $. 

\item For $k \geq 1$: 
\[
\zkp = \calF(\zbark) , \enskip \vkp = \zkp - \zk \qandq V_{k+1} = [\vkp| \vk | \dotsm | v_{k-q+2}]  .
\]

\end{itemize}
{\noindent{\bf{Until:}}} $\norm{\vk} \leq \tol$.
\end{algorithm}
\end{minipage}
\end{center}

\vspace{-3mm}

\begin{remark} $~$
\begin{itemize}[leftmargin=2em, topsep=-4pt, itemsep=2pt]
\item When $\textrm{mod}(k, q+2)\neq0$, one can also consider $\zbark = \zk + \ak(\zk-\zkm)$ with properly chosen $\ak$. Instead of every $q+2$ steps, one can also consider $q+i$ with $i\geq 2$.  



\item A$\!^2$FoM carries out $q+2$ standard FoM iterations to set up the extrapolation step $\calE_{s, q}$. As $\calE_{s, q}$ contains the sum of the powers of $C_k$, it is guaranteed to be convergent if $\rho(C_k) < 1$. Therefore, we only apply $\calE_{s, q}$ when the spectral radius $\rho(C_k) < 1$ is true. In this case, there is a closed form expression for $\calE_{s, q}$ when $s=+\infty$; See Eq. \eqref{eq:lp-mpe}. 

\item The purpose of adding $\ak$ in front of $\calE_{s, q}(\zk, \dotsm, z_{k-q-1})$ is so that we can control the value of $\ak$ to ensure the convergence of the algorithm; See Section \ref{sec:convergence_a2fom}. 

\item Though in this paper, we restricted ourselves with finite dimensional Euclidean space, our Algorithm \ref{alg:a2fom} and its global convergence (Theorem \ref{thm:convergence-lp}) are readily extended to general Hilbert space. 
\end{itemize}

\end{remark}


\begin{remark}
In \eqref{eq:zk_s} we need to consider the sum of the power of $C_k$, 
 \[
 S_{s} = \msum_{i=1}^{s}  C_{k}^i .
 \]
 Suppose that $\Id - C_{k}$ is invertible, recall the Neumann series $\pa{\Id - C_{k}}^{-1} = \msum_{i=0}^{\pinf}  C_{k}^i$.
 Therefore, for $s = \pinf$, 
 \beq\label{eq:Ssinf}
 S_{\pinf}
 = \pa{\Id - C_{k}}^{-1} - \Id
 = C_{k} \pa{\Id - C_{k}}^{-1} .
 \eeq
 In turn, for the finite $s$, we have $ S_{s} = (C_{k} - C_{k}^{s+1}) \pa{\Id - C_{k}}^{-1} $. 

\end{remark}

\subsection{Convergence of \aafom}\label{sec:convergence_a2fom}

In this part we study the global convergence property of \aafom. We first show that the \aafom can be treated as a perturbation of the original fixed-point iteration, and then discuss its convergence properties. Let $\varepsilon_k \in \bbR^n$ whose value takes
\[
\varepsilon_k =
\left\{
\begin{aligned}
0 &: \mathrm{mod}(k, q+2) \neq 0  \enskip\textrm{or}\enskip  \mathrm{mod}(k, q+2) = 0 ~~\&~~ \rho(C_k)\geq1, \\
\ak\calE_{s, q}(\zk, \dotsm, z_{k-q-1})  &: \mathrm{mod}(k, q+2) = 0  \enskip\&\enskip \rho(C_k) < 1  .
\end{aligned}
\right. 
\] 
Then the Algorithm \ref{alg:a2fom} can be written as
\beq\label{eq:perturbation-km}
\zkp = \calF(\zk + \varepsilon_k) .
\eeq
Based on the above reformulation, we have the following convergence result for Algorithm \ref{alg:a2fom} which is based on the classic convergence result of inexact \KM fixed-point iteration \cite[Proposition 5.34]{bauschke2011convex}.

\begin{theorem}\label{thm:convergence-lp}
	For Algorithm \ref{alg:a2fom}, suppose that the fixed-point operator $\calF$ is averaged non-expansive whose set of fixed-points is non-empty. If the perturbation error is absolutely summable, \ie $\sum_{k} \norm{\varepsilon_k} < \pinf$, then there exists a $\zsol \in \fix(\calF)$ such that $\zk \to \zsol$. 
 
\end{theorem}
\begin{proof}
From \eqref{eq:perturbation-km}, we have that
\[
\begin{aligned}
\zkp 
= \calF(\zk + \varepsilon_k) 
&= \calF(\zk) + \Pa{ \calF(\zk + \varepsilon_k) - \calF(\zk) } .
\end{aligned}
\]
Given any $\zsol \in \fix(\calF)$, there holds
\[
\begin{aligned}
\norm{\zkp - \zsol}
\leq \norm{ \calF(\zk) - \calF(\zsol)} + \norm{ \calF(\zk + \varepsilon_k) - \calF(\zk) } 
&\leq \norm{\zk - \zsol} + \norm{ \varepsilon_k }  ,
\end{aligned}
\]
which means that $\seq{\zk}$ is quasi-Fej\'er monotone with respect to $\fix(\calF)$. Then invoking \cite[Proposition 5.34]{bauschke2011convex} we obtain the convergence of the sequence $\seq{\zk}$.  \qedhere

\end{proof}

\begin{remark}
	The perturbation perspective of \aafom implies that we can incorporate other errors in the iteration, as long as the error is absolutely summable. One such case is where $\calF(\zbark)$ is computed approximately, and the accuracy is increasing along the iteration. 
\end{remark}

The above convergence result indicates that we need a proper strategy to ensure that $\ak\calE_{s, q}(\zk, \dotsm, z_{k-q-1})$ is absolutely summable. This can be achieved via a safeguarded version of \aafom, which is inspired by \cite{alvarez2001inertial}.

\begin{center}
	\begin{minipage}{0.975\linewidth}
		\begin{algorithm}[H]
			\caption{\aafom with safeguard} \label{alg:safeLP}
			\KwIn{Let $a,b,\delta > 0$ and $s\geq1, q \geq 1$ be integers.}
			{\noindent{\bf{Initial:}}} Let $z_0 \in \bbR^n$ and $\zbar_0 = z_0$, set $V_{0} = 0 \in \bbR^{n \times q}$; \\
			{\noindent{\bf{Repeat:}}} 
			\begin{itemize}
				\item \textrm{If $\textrm{mod}(k, q+2)=0$: }\textrm{Compute $C_k$ as described above, if $\rho(C_k)<1$:}
				\[
				\begin{aligned}
				\ak &= \min\Ba{ a,~ \sfrac{b}{ k^{1+\delta} \norm{ \calE_{s, q}(\zk, \dotsm, z_{k-q-1}) }  } }	,	\\
				\zbark &= z_k + \ak \calE_{s, q}(\zk, \dotsm, z_{k-q-1})  .
				\end{aligned}
				\]
				\item \textrm{If $\textrm{mod}(k, q+2)\neq0$: } $ \zbark = \zk $.
				
				\item For $k \geq 1$:  
				\[
				\begin{aligned}
				\zkp &= \calF(\zbark) , \enskip
				\vkp &= \zkp - \zk \qandq
				V_{k+1} = [\vkp| \vk | \dotsm | v_{k-q+2}] . 
				\end{aligned}
				\]
				
			\end{itemize}
			{\noindent{\bf{Until:}}} $\norm{\vk} \leq \tol$.
		\end{algorithm}
	\end{minipage}
\end{center}

\section{Acceleration guarantees of \aafom}
\label{sec:acc-gua}

As our \aafom is motivated by the local trajectory of the sequence $\seq{\zk}$, in this part we turn to the local perspective and study the local acceleration guarantees of \aafom. 
We first recall several well established vector extrapolation methods in the field of numerical analysis, build connection with our linear prediction and then discuss the acceleration guarantees.

\subsection{Vector extrapolation techniques}

Vector extrapolation techniques provide a generic recipe for the acceleration of sequences, without  specific knowledge of how the sequence is generated.

In the following, we describe two popular techniques for vector extrapolation of a sequence $\ens{\xk}_k$. Let $\uk \eqdef \xkp - \xk$ and define the matrix
\begin{equation}
U_j \eqdef \begin{bmatrix}
\uk|\ukp|\cdots |u_{k+j}
\end{bmatrix}.
\end{equation}

\paragraph{Idea of vector extrapolation methods}
Suppose we are observing a sequence $\{\xk\}_k$ generated by
\begin{equation}\label{eq:T}
\xkp = T\xk +d
\end{equation}
where $T$ is a matrix and $d$ is a vector, which  are possibly unknown. Assume that $\rho(T)<1$ so that $\lim_{k\to\pinf} \xk \eqdef \xst$ exists. We say that $P$ is a minimal polynomial of $T$ with respect to a vector $v$ if it is the monic polynomial of least degree such that $P(T)v = 0$.

It is known \cite{sidi2017vector} that if $P(\la) = \sum_{i=0}^r c_i \la^i$ is  a minimal polynomial  of $T$ with respect to $\xk- \xst$, it is also minimal with respect to $\xkp-\xk$. Moreover, if $\sum_{i=0}^r c_i \neq 0$, then $\xst = \frac{1}{\sum_i c_i }\sum_{i=0}^r c_i u^{k+i}$ where $\uk\eqdef \xkp-\xk$. Therefore, we can compute $\xst$ from finitely many values of this sequence provided that the minimal polynomial coefficients are known. 
The key observation is that these coefficients can be computed without knowledge of $\xst$, since
\begin{align*}
0 = P(T) \uk = \msum_{i=0}^r c_i T^i \uk =  \msum_{i=0}^r c_i u^{k+i}.
\end{align*}
As $c_r = 1$, we can write this equation as $U_r c = 0$ and $U_{r-1} c' = -u^{k+r}$, where $c' = (c_0,\ldots\, c_{r-1})^\top$. Note that this is an overdetermined system if $r\leq d$, and is consistent and has a unique solution. Finally, setting $\gamma_i = \frac{c_i}{\sum c_i}$, we have $\xst = \sum_i \gamma_i x_{k+i}$.
This process of computing the coefficients is known as minimal polynomial extrapolation, and is summarized below.


\begin{center}
	\begin{minipage}{0.975\linewidth}
	\begin{algorithm}[H]
	\caption{Minimal polynomial extrapolation (MPE)} \label{alg:mpe}
	%
            \begin{enumerate}
            \item Choose  integers $r$ and $k$ and input the vectors $\xk, \xkp, \ldots,x_{k+r+1}$.
            \item Compute the vectors $\uk,\ukp,\ldots, u_{k+r}$ and the matrix $U_{r-1}=\begin{bmatrix}
            \uk|\ukp|\cdots |u_{k+r-1}
            \end{bmatrix}$.
            \item Find $c'\eqdef \begin{bmatrix}
            c_0,\ldots, c_{r-1}
            \end{bmatrix}^\top$ as the least squares solution to  $U_{r-1}c' =  - u^{k+r}$. Set $c_r \eqdef 1$ and $\gamma_i = c_i/\sum_{j=0}^{r} c_j$ for $i=0,\ldots, r$.
            \item Compute $s_{k,r}\eqdef \sum_{i=0}^r \gamma_i x_{k+i}$ as an approximation to $\lim_{k\to+\infty}\xk = \xst$.
            \end{enumerate}
	%
	\end{algorithm}
	\end{minipage}
\end{center}

In general, if $r$ is chosen too small, $\sum_i c_i$ might be zero and MPE will fail. To circumvent this, reduced rank extrapolation was introduced, where step 3 is replaced with a constrained minimization problem.


\begin{center}
	\begin{minipage}{0.975\linewidth}
	\begin{algorithm}[H]
	\caption{Reduced rank extrapolation (RRE)} \label{alg:rre}
	%
            \begin{enumerate}
            \item Choose  integers $r$ and $k$ and input the vectors $\xk, \xkp, \ldots,x_{k+r+1}$.
            \item Compute the vectors $\uk,\ukp,\ldots, u_{k+r}$ and for the matrix $U_{r}$.
            \item Let $\gamma \in \argmin_\gamma~ \norm{U_r \gamma} \text{ subject to } \msum_{i=0}^r \gamma_i = 1$. 
            \item Compute $s_{k,r}\eqdef \sum_{i=0}^r \gamma_i x_{k+i}$ as an approximation to $\lim_{k\to+\infty}\xk = \xst$.
            \end{enumerate}
	%
	\end{algorithm}
	\end{minipage}
\end{center}

Another form of convergence acceleration technique is Anderson  acceleration \cite{Anderson}, whose formulation is similar to that of RRE (and is equivalent in the linear setting). Further details about its relation to vector extrapolation technique can be found in \cite{brezinski2018shanks}. There has been recent work on applying such extrapolation techniques to accelerate first order algorithms \cite{scieur2016regularized,zhang2018globally,peng2018anderson}. One of the challenges of applying these methods is that while $s_{k,r}\to s$ as $r\to \infty$, choosing large  values for $r$ could lead to   $U_r$  being ill-conditioned, \cite{scieur2016regularized} suggested to circumvent this issue using regularization techniques when solving step 3. However, naive regularization could actually slow down convergence, and an adaptive choice of the regularization parameter may lead to many evaluations of the objective function which may be costly. 

\subsection{Equivalence between \aafom and MPE}

%

%

%

We now build a connection between our \aafom with MPE/RRE for the case of $s = \pinf$, that is when the linear prediction is taken for infinite steps.

Owing to \eqref{eq:Ssinf}, when $s=\pinf$, from \eqref{eq:zk_s} we get
\beq\label{eq:lp-mpe}
\begin{aligned}
\zbar_{k,\infty} 
\eqdef \zk + V_{k} \Pa{ (\Id - C_{k})^{-1} - \Id }_{(:, 1)}  
= \zk - \vk + V_{k} \Pa{ (\Id - C_{k})^{-1} }_{(:, 1)}   
&= \zkm + V_{k} \Pa{ (\Id - C_{k})^{-1} }_{(:, 1)}   \\
&= \sfrac{1}{1- \sum_{i=1}^s c_{k,i}} \Pa{ z_k - \msum_{j=1}^{q-1} c_{k,j} z_{k-j}} ,
\end{aligned}
\eeq
which turns out to be MPE, with the slight difference of taking the weighted sum of $\ba{z_j}_{j=k-q+1}^k$ as opposed to the weighted sum of $\ba{z_j}_{j=k-q}^{k-1}$.  
Let $b \in \bbR^{q+1}$ be such that
\[
b\in\argmin_{b\in\RR^{q+1}, \sum_j b_j =1}\norm{ \msum_{j=0}^{q} b_j v_{k-j} }
\]
and $b_0 \neq 0$. Define $c_j\eqdef - b_j/b_0$ for $j=1,\ldots, q$, then we have
$$
\Pa{ 1-\msum_{i=1}^{q} c_i }^{-1} =\sfrac{b_0}{ b_0 + \sum_{j=1}^{q} b_j} = b_0,
$$
and 
$ \zbar_{k,\infty} = \msum_{j=0}^{q-1} b_j z_{k-j} $ 
which is precisely the RRE update (again with the slight difference of summing over iterates shifted by one iteration).

\begin{remark}
Based on the structure of $C$ in \eqref{eq:coeff_mtx}, simple calculation yields 
\begin{align*}
\small
\Id - C = \begin{bmatrix}
(1-c_1) &  -1& 0& \cdots &0\\
- c_2 & 1&-1 & \ddots&\vdots\\
 - c_3 &0 & \ddots&\ddots &0\\
 \vdots &\vdots & \ddots&\ddots&-1 \\
 -c_q&0 & \cdots &0&1\\
\end{bmatrix}_{q\times q} 
~\textrm{and}\enskip
(\Id - C)^{-1} = \sfrac{1}{1- \sum_{i=1}^s c_i} \begin{bmatrix}
1&1&1&\cdots&\cdots&1&1\\
b_2 & \bar b_2 & \bar b_2&\cdots&\cdots& \bar b_2& \bar b_2\\
b_3 & b_3 & \bar b_3&\cdots&\cdots& \bar b_3& \bar b_3\\
b_4 & b_4 & b_4& \bar b_4&\cdots& \bar b_4& \bar b_4\\
\vdots &&&&&&\vdots\\
b_q&b_q & \cdots&\cdots&\cdots&b_q& \bar b_q
\end{bmatrix}_{q\times q} 
\end{align*}
where $b_j \eqdef \sum_{i\geq j} c_i$ and $\bar b_j = 1- \sum_{i<j}c_i$ such that $\bar b_j-b_j = 1-\sum_j c_j$. 
\end{remark}


\subsection{Acceleration guarantees of \aafom}

We are now ready to discuss the acceleration guarantees of \aafom. We first characterize the prediction error of our proposed \aafom, and then discuss its acceleration guarantees based on the relation with MPE/RRE.

\subsubsection{Prediction error of \aafom}

To discuss the prediction error of \aafom, we need to rewrite \eqref{eq:linearization-zkp-zk} first. Denote $f_{k} = o(\norm{\zk-\zkm})$
and 
\[
F_{k} = [f_{k}| f_{k-1} | \dotsm | f_{k-q+1}] \in \RR^{n\times q}.
\]
Recall that $\vk \eqdef \zk-\zkm$ and $V_{k} = [\vk| \vkm | \dotsm | v_{k-q+1}] \in \RR^{n\times q}$, from \eqref{eq:linearization-zkp-zk} we have $v_k = \mF (v_{k-1}) + f_{k-1}$ and 
\begin{equation}\label{eq:linear_rel}
V_{k} = \mF V_{k-1} + F_{k-1} .
\end{equation}
By virtue the definition of the coefficients matrix $C_k\eqdef H(c_k)$ of \eqref{eq:coeff_mtx},
define $E_{k,j} \eqdef V_k C_k^j - V_{k+j}$ for $j\geq 1$ and
\begin{equation}\label{eq:e0}
E_{k,0}\eqdef V_{k-1} C_k - V_{k} = \begin{bmatrix}
(V_{k-1} c_k - v_{k}) & 0 & \cdots & 0
\end{bmatrix}.
\end{equation}
We arrive at the following relation between the extrapolated point $\zbar_{k,s}$ and the $(k+s)$'th point of $\seq{\zk}$
$$
\zbar_{k,s} = z_{k} + \msum_{j=1}^s \pa{v_{j+k} + (E_{k,j})_{(:,1)}}
= z_{k+s} + \msum_{j=1}^s (E_{k,j})_{(:,1)} .
$$
Here given a matrix $E$, $E_{(:,j)}$ denotes its $j$'th column. 
As a result, we derive the following proposition on the prediction error $\bar z_{k,s} - \zsol$.


\begin{theorem}[Prediction error]\label{thm:extrap-error}
Given a first-order method of the form \eqref{eq:fom}, let \eqref{eq:linearization-zkp-zk} be its local linearization. 
For Algorithm \ref{alg:a2fom}, when the linear prediction is applied, we have the following error bounds: 
Let 
$$
B_{k,s} \eqdef \max_{i\in \{0,1\},t\leq s}\Ba{\norm{\msum_{\ell=i}^{t} \mF^\ell}, \norm{\msum_{\ell=i}^t C_{k}^\ell} }
$$
and define the coefficients fitting error as $\epsilon_k \eqdef \norm{\sum_{i=1}^{q} c_{k,i} v_{k-i} - v_{k}}$. 
Then, the prediction error $\bar z_{k,s} - \zsol$ satisfies
\begin{align*}
&\norm{\bar z_{k,s} - \zsol} 
\leq  \norm{\mF^s \pa{z_k - \zsol}} + \norm{\msum_{\ell=0}^{s-1} \mF^\ell} \norm{ \hat f_k} + B_{k,s}  \pa{ \epsilon_k + \norm{F_{k-1}}},
\end{align*}
where $\hat f_k \eqdef \mF (z_{k-1}-\zsol) - (z_k-\zsol)$.

In the case of $s=+\infty$, there holds
\[
\begin{aligned}
&\norm{\bar z_{k,+\infty} - \zsol} 
\leq  \norm{(\Id - \mF)^{-1}} \norm{\hat f_k} 
+ \epsilon_k \; \frac{ \sum_{\ell=1}^{+\infty} \norm{\mF^\ell} }{ {1-\sum_i c_{k,i}} } + \sum_{\ell=0}^{+\infty} \norm{\mF^\ell} \norm{F_{k-1} ((\Id - C_{k})^{-1} - \Id)_{(:,1)}}  .
\end{aligned}
\]
\end{theorem}

The proof of the theorem can be found in Section \ref{sec:proof_acceleration} of the appendix.

\begin{remark}$~$
	\begin{itemize}
		\item The fact that $B_{k,s}$ is uniformly bounded in $s$ if $\rho(\mF)<1$, and $\rho(C_{k})<1$ follows because this implies that $\sum_{\ell=1}^{+\infty} \norm{\mF^\ell}<+\infty$ thanks to the Gelfand formula, and $\sum_{i=0}^{+\infty} C_{k}^i = (\Id - C_{k})^{-1}$ and its $(1,1)^{th}$ entry is precisely $\frac{1}{1-\sum_i c_{k,i}}$.
		
		\item In Theorem \ref{thm:extrap-error}, the prediction error consists of two main sources: coefficient fitting error of $\epsilon_k$ and linearization error of $F_{k-1}$ which corresponds to the small $o$-terms. 
		When the small $o$-term in \eqref{eq:linearization-zkp-zk} vanishes,  that is $F_{k-1} = 0$, then it follows from the proof that
		$$
		\norm{\bar z_{k,s} - \zsol} \leq \norm{z_{k+s}-\zsol} + B_{k,s} \epsilon_k
		$$
		and if the spectral radius $\rho(\mF)<1$ and $\rho(C_{k})<1$, then
		$$
		\norm{\bar z_{k,+\infty} - \zsol} \leq \msum_\ell \norm{\mF^\ell}    \sfrac{ \epsilon_k }{ 1-\sum_i c_{k,i} }  .
		$$
	\end{itemize}

\end{remark}

%

\subsubsection{Acceleration guarantees}

As shown in Theorem \ref{thm:extrap-error}, a key quantity governing the amount of acceleration is the coefficient fitting error $\epsilon_k$. For the case that the small $o$-terms vanish, this error can be bounded using existing results of vector extrapolation. In the following, we assume that \eqref{eq:fom} can be linearized without small $o$-term and derive acceleration guarantees for Algorithm \ref{alg:a2fom}. 

\begin{theorem}[Acceleration guarantees]\label{thm:acc-guarantee}
Given a first-order method of the form \eqref{eq:fom}, suppose there exists a linear matrix $\mF$ such that it can be linearized of the form $$\zkp-\zk = \mF (\zk-\zkm).$$ 
%
%
Suppose that $\mF$ is diagonalizable. Let $\{\lambda_j\}_j$ denote its distinct eigenvalues ordered such that $\abs{\lambda_j}\geq \abs{\lambda_{j+1}}$ and $\abs{\lambda_1}= \rho(\mF)<1$. Suppose that $\abs{\lambda_{q}}>\abs{\lambda_{q+1}}$.  Then we have the following bounds on $\epsilon_k$
\begin{itemize}[leftmargin=2.25em]
\item
Asymptotic bound (fixed $q$ and as $k\to +\infty$): $\epsilon_k =O \pa{\abs{\lambda_{q+1}}^k}$. 
\item
Non-asymptotic bound (fixed $q$ and $k$):
Suppose that $\lambda(\mF)$ is real-valued and contained in the interval  $[\alpha,\beta]$ with $-1<\alpha<\beta<1$.
Then,
\begin{equation}\label{eq:opt-err}
\qfrac{\epsilon_k}{1-\sum_i c_{k,i}} \leq K \beta^{k-q} \Pa{\tfrac{\sqrt{\eta}-1}{\sqrt{\eta}+1}}^q
\end{equation}
where
  $K\eqdef {2   \norm{z_0-\zsol}}\norm{(\Id-M)^{\frac12}}$ and $\eta = \frac{1-\alpha}{1-\beta}$. 
\end{itemize}
%
\end{theorem}

\begin{remark}$~$
\begin{itemize}[leftmargin=2em]
\item As we have seen in Section \ref{sec:trajectory-fom}, when $R$ and $J$ are both polyhedral,  when the optimization problem is locally polyhedral around the solution, the small $o$-term vanishes and we have a perfect local linearization. Hence, the conditions of Theorem \ref{thm:acc-guarantee} holds for all $k$ large enough. 

	\item Combined with Theorem  \ref{thm:extrap-error}, this result shows that the extrapolated point $\zbar_{k,s}$ moves along the true trajectory as $s$ increases, up to the fitting error $\epsilon_k$. Note that the same error bounds also holds for MPE, see for instance \cite{sidi2017vector}, and as discussed previously, our update $\zbar_{k,+\infty}$ is essentially an MPE update. However, this theorem offers a further geometric interpretation of these extrapolation methods in terms of following the ``sequence trajectory'', and combined with our local analysis of \fom, provides justification of these methods for the acceleration of non-smooth optimization problems.

\end{itemize}

\end{remark}

\begin{remark}[Acceleration guarantee and the choice of $q$]$~$
\begin{itemize}
	\item For Forward--Backward splitting method, as the angle $\theta_k$ converges to $0$. For the coefficient fitting error we have $\epsilon_k = o(\rho(\mF))$, which indicate that \aafom can provide acceleration with $q=1$. Since $q=1$ corresponds to the inertial scheme, our result complies with the current literature on inertial Forward--Backward splitting methods. 
	
	\item Theorem \ref{thm:acc-guarantee} (ii) shows that extrapolation improves the convergence rate from $O (\abs{\lambda_1}^k)$ to $O (\abs{\lambda_{q+1}}^k)$, and the non-asymptotic bound shows that the improvement of extrapolation is optimal in the sense of Nesterov \cite{nesterov83}. 
	Take Douglas--Rachford splitting for example, in the case of two non-smooth polyhedral terms,
	we must have $\abs{\lambda_{2j-1}} = \abs{\lambda_{2j}} > \abs{\lambda_{2j+1}}$ for all $j\geq 1$. Hence, no acceleration can be guaranteed or observed when $q=1$, while the choice of $q=2$ provides guaranteed acceleration. 
\end{itemize}
\end{remark}


\begin{remark}[Dealing with small $o$-terms]  We now consider the coefficients fitting error of the perturbed problem $v_k = \mF (v_{k-1}) + f_{k-1}$. Let $v_k^0 = \mF(v_{k-1}^0)$ with $v_{k-q}^0 = v_{k-q}$ and let $c^0_k\in\RR^q$ and $C^0_k = H(c_k^0)$ be the associated coefficients and coefficients matrix. Let $\epsilon^0$ be the coefficients fitting error for this unperturbed problem, then 
\begin{align*}
\epsilon 
= \min_{c\in\RR^q} \norm{\msum_{j=1}^q c_j v_{k-j} - v_k} 
&\leq \epsilon^0 +  \norm{\msum_{j=1}^q c_{k,j}^0 (v_{k-j} - v_{k-j}^0 )- v_k-v_k^0}\\
&\leq \epsilon^0 + \norm{  \msum_{i=1}^q c_{k,i}^0 \Pa{
\msum_{\ell=1}^{q-i} \mF^{\ell-1} f_{k-i-\ell} } - \Pa{
\msum_{\ell=1}^{q} \mF^{\ell-1} f_{k-\ell} }}\\
&\leq \epsilon^0 + (1+\norm{c^0}_1) \max_{i=1}^q \norm{f_{k-i}} \msum_{\ell=1}^q \norm{\mF^\ell}
\end{align*}
where we have used
\begin{align*}
v_{k-i} - v_{k-i}^0 = \mF^{q-i}(v_{k-q} - v_{k-q}^0) + 
\sum_{\ell=1}^{q-i} \mF^{\ell-1} f_{k-i-\ell}  = \sum_{\ell=1}^{q-i} \mF^{\ell-1} f_{k-i-\ell} .
\end{align*}
Therefore, even with the presence of small $o$-terms, the coefficients fitting error can be bounded in terms of the small $o$-terms and  the coefficients fitting error under exact linearization.
\end{remark}


%


\section{Implementation and numerical experiments}\label{sec:exp}

Our proposed adaptive acceleration scheme Algorithm \ref{alg:a2fom} is quite abstract in the sense it is only presented for fixed-point iteration. While for first-order methods, as we have seen in Section \ref{sec:trajectory-fom}, each method has a unique fixed-point characterization. Therefore, in section we discuss how to implement \aafom for different algorithms and provide numerical tests to demonstrate the performance of our acceleration scheme.

\subsection{Gradient descent}
We first consider the comparisons of gradient descent method on least square problems. Since gradient descent is a special case of Forward--Backward splitting with $R$ being $0$ (hence $\prox_{\gamma R} = \Id$), we refer to Algorithm \ref{alg:a2fb} for the specialization of \aafom to gradient descent algorithm. 

For least square problem, gradient descent results in the linear system of \eqref{eq:T}. Therefore, in this example we compare the performance of the following methods: 
\begin{itemize}
	\item Gradient descent, FISTA \cite{fista2009}, restarting FISTA \cite{o2012adaptive}.
	\item Our proposed scheme (LP) with $(q,s) = (6, +\infty)$.
	\item MPE \cite{cabay1976polynomial}, RRE\cite{eddy1979extrapolating} and regularized non-linear acceleration (RNA) \cite{scieur2016regularized}. 
\end{itemize}
%
The following least square problem is considered
\[
\min_{x\in\bbR^{50}}~ \sfrac{1}{2}\norm{Ax - f}^2 ,
\]
where three different choices of $A$ are implemented
\begin{itemize}
	\item Tridiagonal matrix with main diagonal elements equal to $2$, and the elements of the first diagonal below and above main diagonal equal to $-1$;
	\item $A = \texttt{rand}(51,50)$ is generated from uniform distribution in $[0, 1]$.
	\item $A = \texttt{randn}(51,50)$ is generated from normal Gaussian distribution. 
\end{itemize}
The performance comparison of different methods are shown in Figure \ref{fig:cmp_lse}, from which we observe that
\begin{itemize}
	\item Among all the algorithms, gradient descent (gray line) is the slowest, while ``FISTA'' is the 2nd slowest for the {\tt rand} and {\tt randn} cases of $A$. 
	\item ``Restarting FISTA'' shows the best overall performance, especially for the case of tridiagonal $A$ as it is significantly faster than all the the other algorithms. 
	\item For the vector extrapolation based algorithms (MPE/RRE, RNA and our proposed algorithm), except for {\tt rand} $A$ where linear prediction is slower than the others, their performances are quite close. 
\end{itemize}
In light of our analysis, since gradient descent has an eventual straight-line trajectory, inertial is expected to perform well. Indeed, we observe here is restarted FISTA is the fastest and is rather impressive given its simplicity and easy implementation.


\begin{figure}[!ht]
	\centering
	\subfloat[$A$ is a tridiagonal matrix]{ \includegraphics[width=0.315\linewidth]{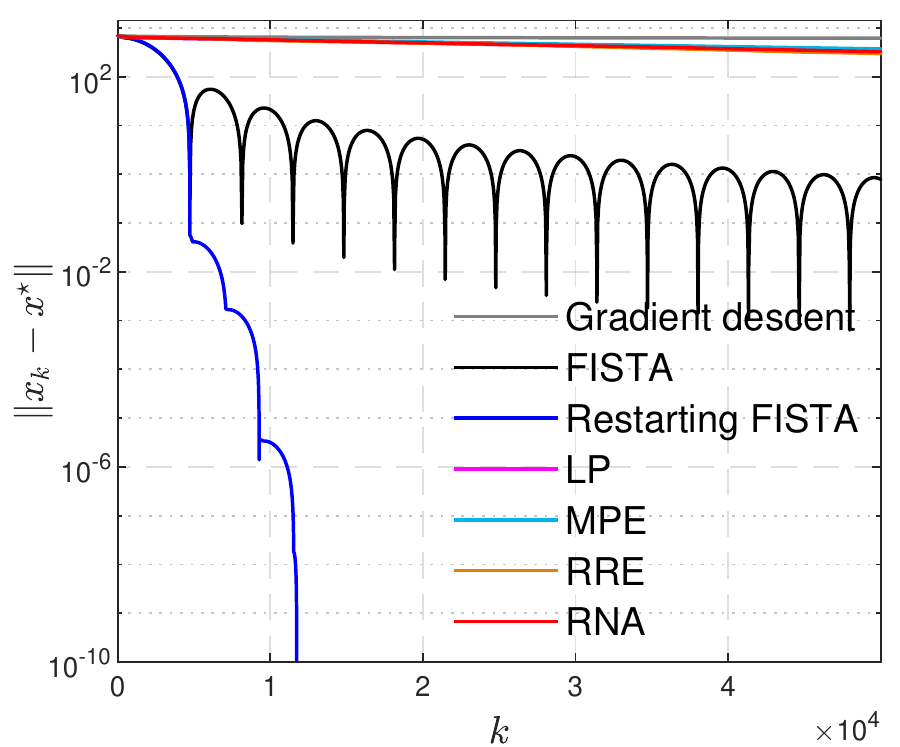} } \hspace{0pt}
	\subfloat[$A = \texttt{rand}(50, 50)$]{ \includegraphics[width=0.315\linewidth]{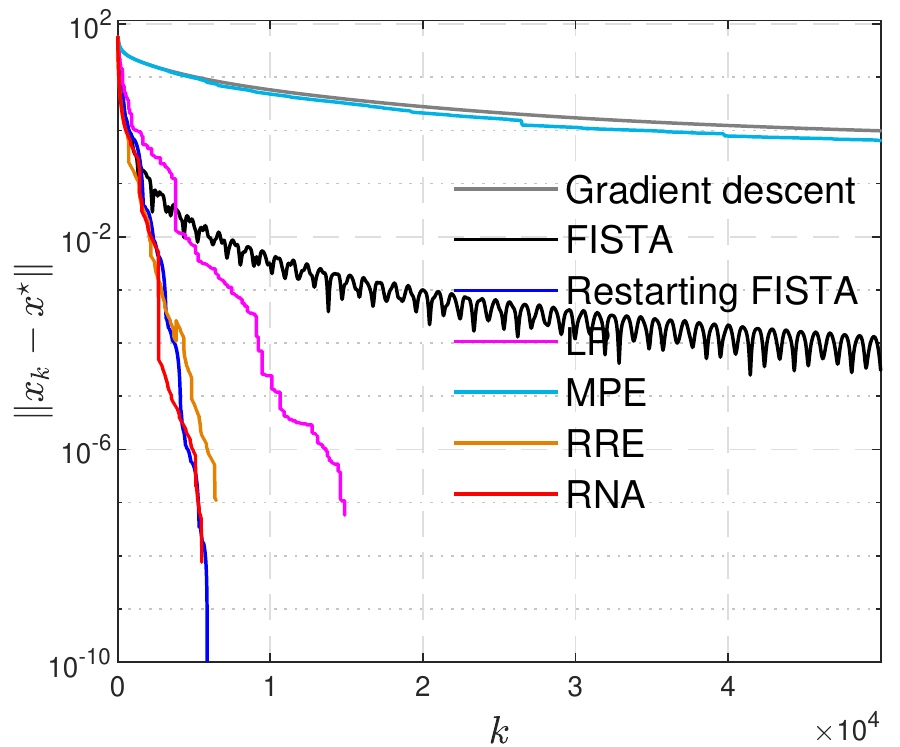} } \hspace{0pt}
	\subfloat[A = \texttt{randn}(50, 50)]{ \includegraphics[width=0.315\linewidth]{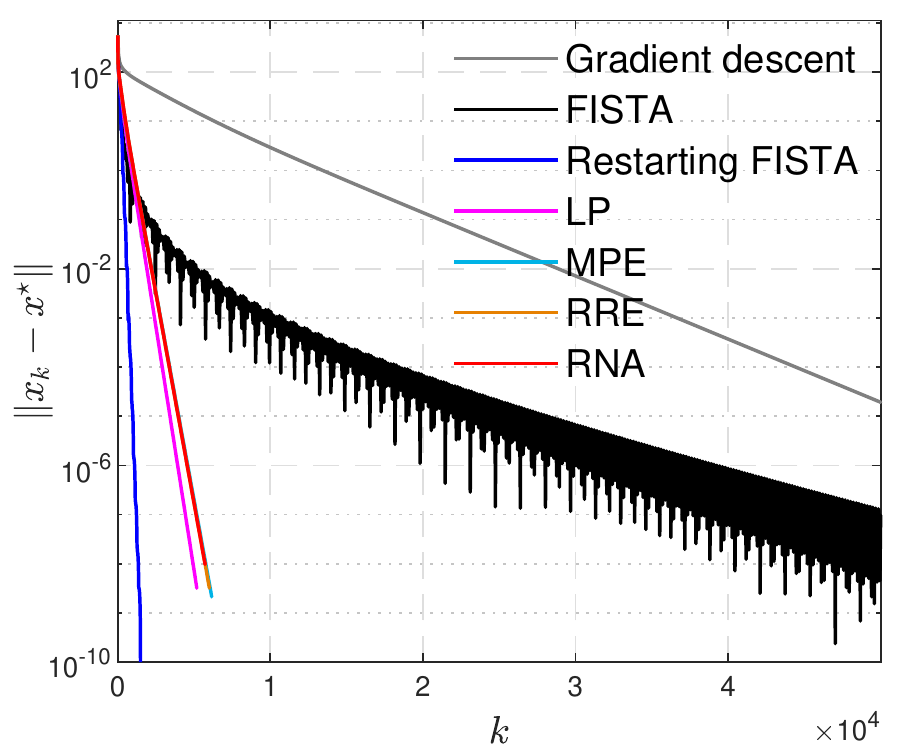} } \\
	\caption{Performance comparison among different schemes under different choices of $A$.} 
	\label{fig:cmp_lse}
\end{figure}


\begin{remark}
	Note that for gradient descent, also the Forward--Backward splitting method to be discussed below, line search can be applied to all the methods compared above. However, we decided not to provide the comparison here since line search can provide acceleration for all these algorithms. Moreover, line search needs the evaluation of objection function values which will increase the overall wall clock time and make it not practical for large scale problems. 
\end{remark}


%
%
%
%
%
%

\subsection{Forward--Backward splitting}

We start with the Forward--Backward splitting algorithm, adapt \aafom to this method we obtain the following adaptive accelerated Forward--Backward splitting scheme. 

\begin{center}
\begin{minipage}{0.975\linewidth}
\begin{algorithm}[H]
\caption{\aafb: Adaptive Acceleration for Forward--Backward splitting} \label{alg:a2fb}
\KwIn{$\gamma \in ]0, 2/L[$. Let $s\geq1, q \geq 1$ be integers. }
{\noindent{\bf{Initial}}}:$\xbar_{0} = x_0 \in \bbR^n$. Set $V_{0} = 0 \in \bbR^{n \times q}$. \\ 
{\noindent{\bf{Repeat}}}: 
\begin{itemize}[leftmargin=2em]
\item \textrm{If $\textrm{mod}(k, q+2)=0$: } \textrm{Compute $C_k$ via \eqref{eq:coeff_mtx}, if $\rho(C_k)<1$ and $\tcb{\angle(\vk, \calE_{s, q}(\xk, \dotsm, x_{k-q-1})) \leq \frac{\pi}{2}}$: }  $$ \xbark = \xk + \ak \calE_{s, q}(\xk, \dotsm, x_{k-q-1}). $$
			
\item \textrm{If $\textrm{mod}(k, q+2)\neq0$: } $ \xbark = \xk $. 

\item For $k \geq 1$: 
\[
\begin{aligned}
\xkp &= \prox_{\gamma R}\Pa{\xbark - \gamma \nabla F (\xbark)} ,\\
\vkp &= \xkp - \xk \qandq V_{k+1} = [\vkp| \vk | \dotsm | v_{k-q+2}]  .
\end{aligned}
\]

\end{itemize}
{\noindent{\bf{Until}}}: $\norm{\vk} \leq \tol$. 
\end{algorithm}
\end{minipage}
\end{center}

\begin{remark}
Note that for the above scheme, we have an extra check on the angle between $\vk=\xk-\xkm$ and the extrapolated direction $\calE_{s, q}(\xk, \dotsm, x_{k-q-1})$, and the value of $\theta$ is chosen close to $0$. This is due to the fact that the trajectory of $\seq{\xk}$ is eventually a straight-line, so we only accept $\calE_{s, q}(\xk, \dotsm, x_{k-q-1})$ if the angle $\angle(\vk, \calE_{s, q}(\xk, \dotsm, x_{k-q-1})) < \frac{\pi}{2}$. 

\end{remark}

\subsubsection{LASSO-type problem}

Next we consider regularized least square problem of the form
\beq\label{eq:R-quadF}
\min_{x \in \bbR^{n}}~  R(x) + \qfrac{1}{2}\norm{\calK x - f}^2 ,
\eeq
where $R$ is regularization term, and $\calK \in \bbR^{m \times n}$ is drawn from random Gaussian ensemble. $f$ is the observation of some $\xob$ under $\calK$ contaminated by noise $w$,
\beq\label{eq:observation}
f = \calK \xob + w  .
\eeq
In this experiment, three different cases of $R$ are considered: sparsity promoting $\ell_{1}$-norm, group sparsity promoting $\ell_{1,2}$-norm and low-rank promoting nuclear norm.
The detailed settings of each example are
\begin{description}[leftmargin=3.0cm]
\item[{$\ell_{1}$-norm}] $(m,n)=(768,2048)$, $\xob$ has $176$ non-zero elements.
\item[{$\ell_{1,2}$-norm}] $(m,n)=(640,2048)$, $\xob$ has $35$ non-zero blocks of size $4$.
\item[{Nuclear norm}] $(m,n)=(640,1024)$, $\xob \in \bbR^{32\times 32}$ and $\rank(\xob) = 4$.
\end{description}
%
The following scheme are compared
\begin{itemize}
\item Forward--Backward splitting, FISTA and restarting FISTA.
\item Our proposed scheme (LP) with $(q,s) = (4,+\infty)$.
\end{itemize}
The finite activity identification of $\seq{\xk}$ and the angle $\cos(\theta_k)$ of Forward--Backward splitting is provided in the first row of Figure \ref{fig:cmp_fb}, the observations are quite close to those of Example \ref{eg:lasso}.

The comparison of the above methods is presented in the second row of Figure \ref{fig:cmp_fb}, and we observe that
\begin{itemize}
	\item Similar to the least square example, Forward--Backward splitting method is the slowest one. However, note that in terms of local linear convergence rate, FISTA is the slowest one --- see the local slope of the gray and black line. The problem of Forward--Backward splitting method is that it needs much longer time to identified the underlying manifold. 
\item Restarting FISTA (blue line) is the fastest among all methods, our proposed linear prediction is as fast as restarting FISTA for the first two examples and slightly slower for the last example. 
\end{itemize}

\begin{figure}[!ht]
	\centering
	\subfloat[$\ell_{1}$-norm]{ \includegraphics[width=0.31\linewidth]{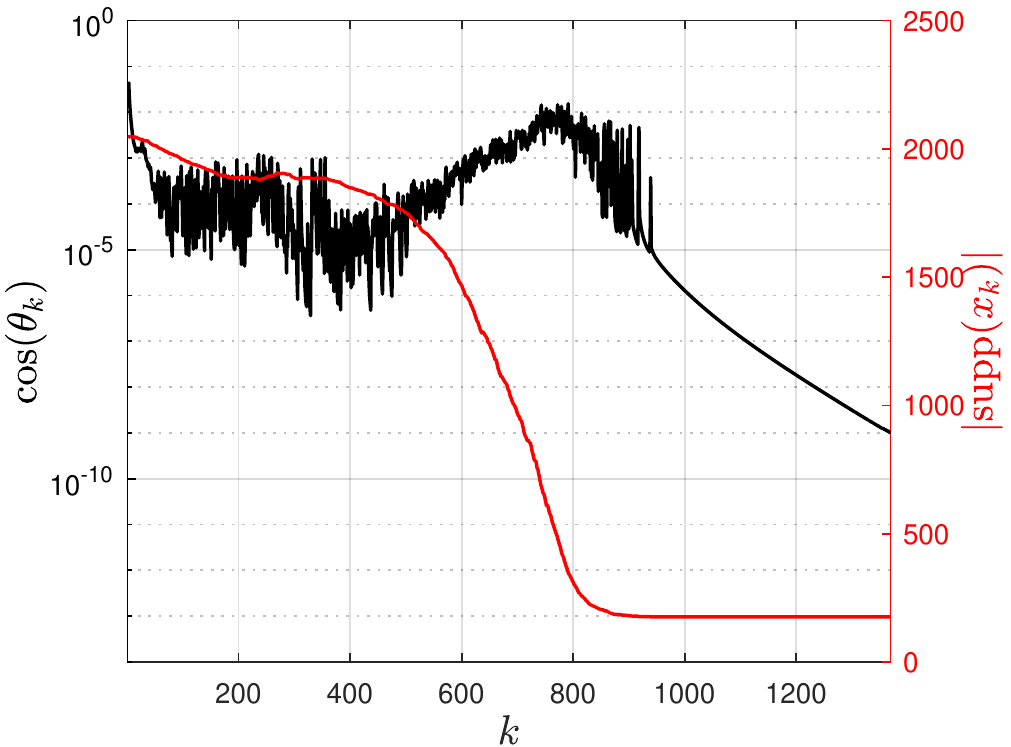} } \hspace{2pt}
	\subfloat[$\ell_{1,2}$-norm]{ \includegraphics[width=0.31\linewidth]{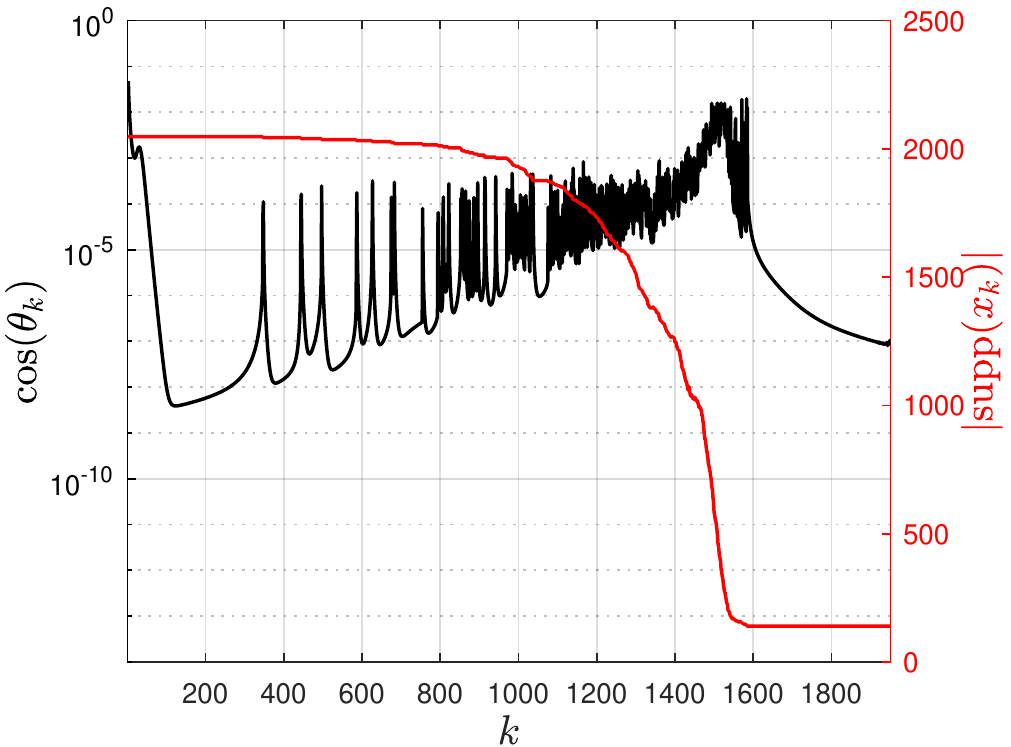} } \hspace{2pt}
	\subfloat[Nuclear norm]{ \includegraphics[width=0.31\linewidth]{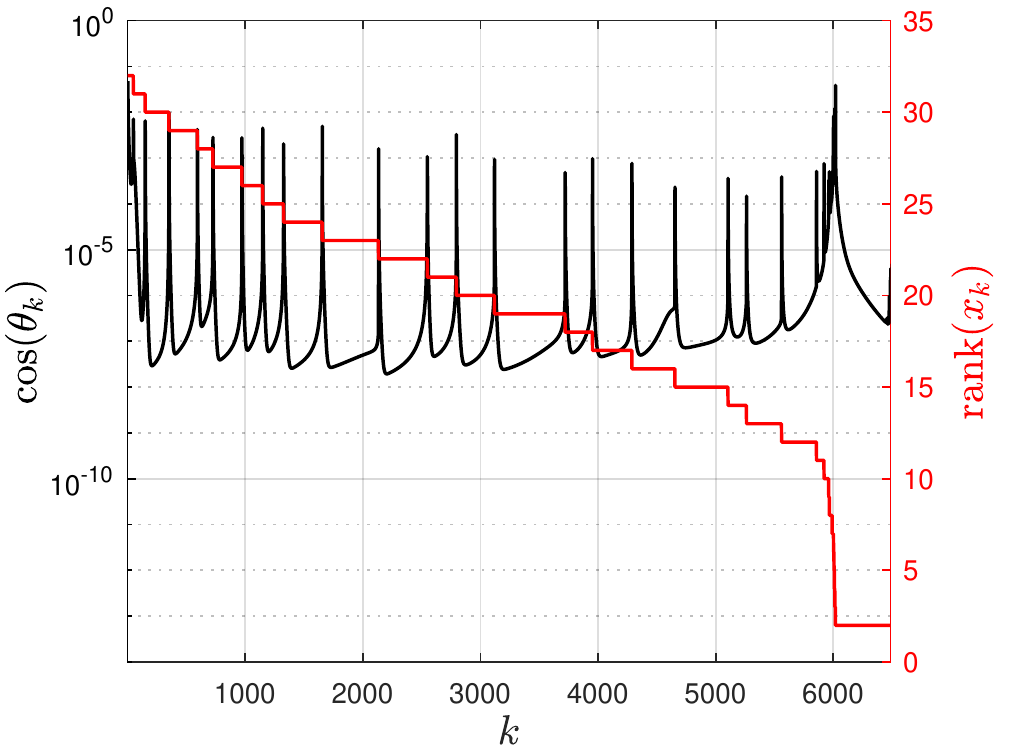} } \\ 
	\subfloat[$\ell_{1}$-norm]{ \includegraphics[width=0.31\linewidth]{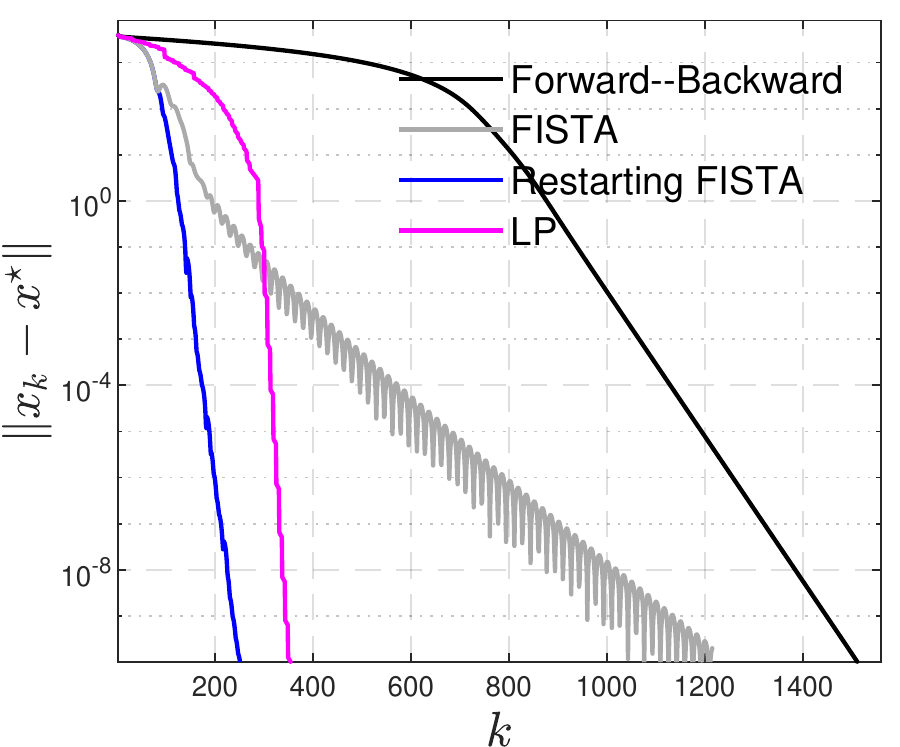} } \hspace{2pt}
	\subfloat[$\ell_{1,2}$-norm]{ \includegraphics[width=0.31\linewidth]{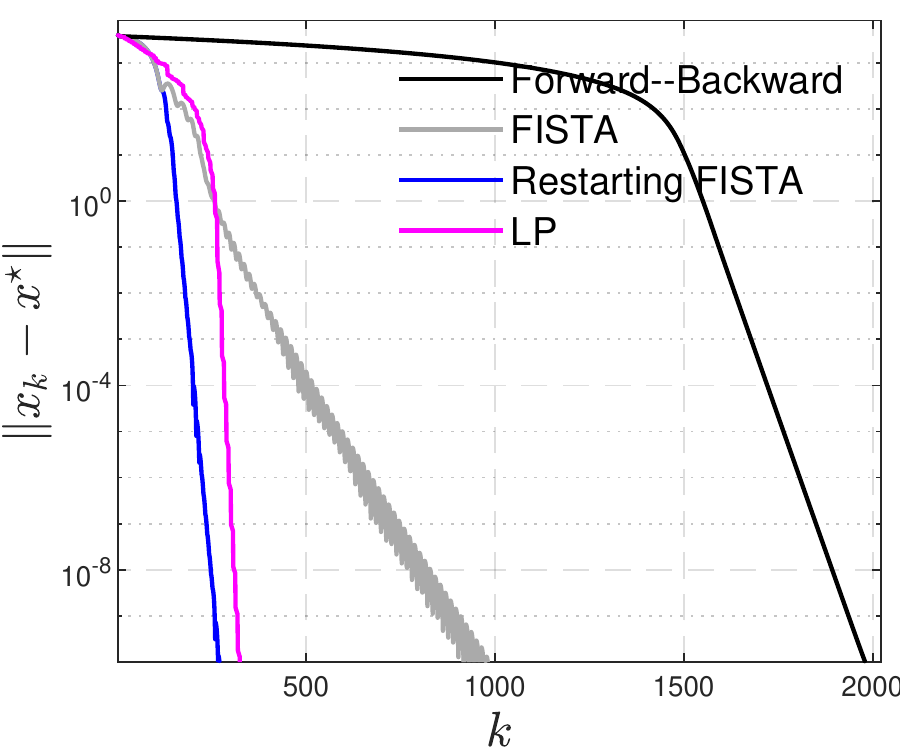} } \hspace{2pt}
	\subfloat[Nuclear norm]{ \includegraphics[width=0.31\linewidth]{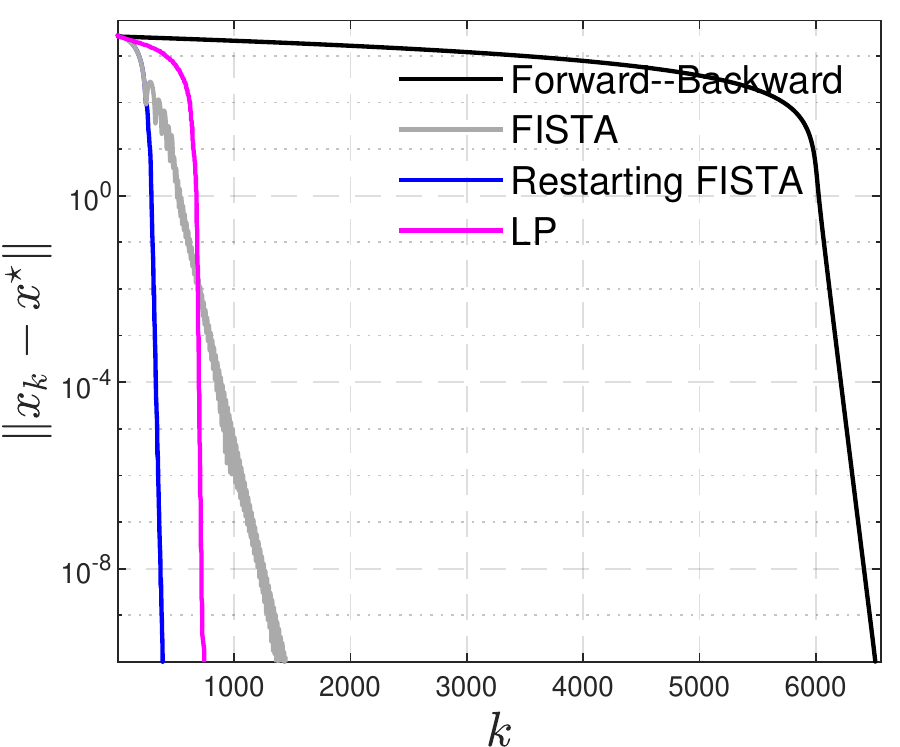} } \\
	\caption{Comparison between methods for solving regularized least square. } 
	\label{fig:cmp_fb}
\end{figure}

\subsection{Douglas--Rachford splitting}

Now we turn to the Douglas--Rachford splitting method, for which we obtain an adaptive acceleration scheme described in Algorithm \ref{alg:a2dr}. 

\begin{center}
	\begin{minipage}{0.975\linewidth}
		\begin{algorithm}[H]
			\caption{\aadr: Adaptive Acceleration for Douglas--Rachford splitting} \label{alg:a2dr}
			\KwIn{$\gamma > 0$. Let $s\geq1, q \geq 1$ be integers.}
			{\noindent{\bf{Initial}}}: $\zbar_{0} = z_0 \in \bbR^n ,~ x_0 = \prox_{\gamma J} (\zbar_0)$. Let $V_{0} = 0 \in \bbR^{n \times q}$. \\ 
			{\noindent{\bf{Repeat}}}: 
			\begin{itemize}[leftmargin=2em]
				\item \textrm{If $\textrm{mod}(k, q+2)=0$: } \textrm{Compute $C_k$ via \eqref{eq:coeff_mtx}, if $\rho(C_k)<1$: $\zbark = \zk + \ak \calE_{s, q}(\zk, \dotsm, z_{k-q-1})$.}  
				
				\item \textrm{If $\textrm{mod}(k, q+2)\neq0$: } $ \zbark = \zk $. 
				
				\item For $k \geq 1$: 
				\[
				\begin{aligned}
				\xk &= \prox_{\gamma J} (\zbark) , \\
				\ukp &= \prox_{\gamma R}\pa{2\xk - \zbark} , \\
				\zkp &= \zbark + \ukp - \xk , \\
				\vkp &= \zkp - \zk \qandq V_{k+1} = [\vkp| \vk | \dotsm | v_{k-q+2}]  .
				\end{aligned}
				\]
				
			\end{itemize}
			{\noindent{\bf{Until}}}: $\norm{\vk} \leq \tol$. 
		\end{algorithm}
	\end{minipage}
\end{center}

\begin{figure}[!ht]
	\centering
	\subfloat[$\ell_{1}$-norm]{ \includegraphics[width=0.315\linewidth]{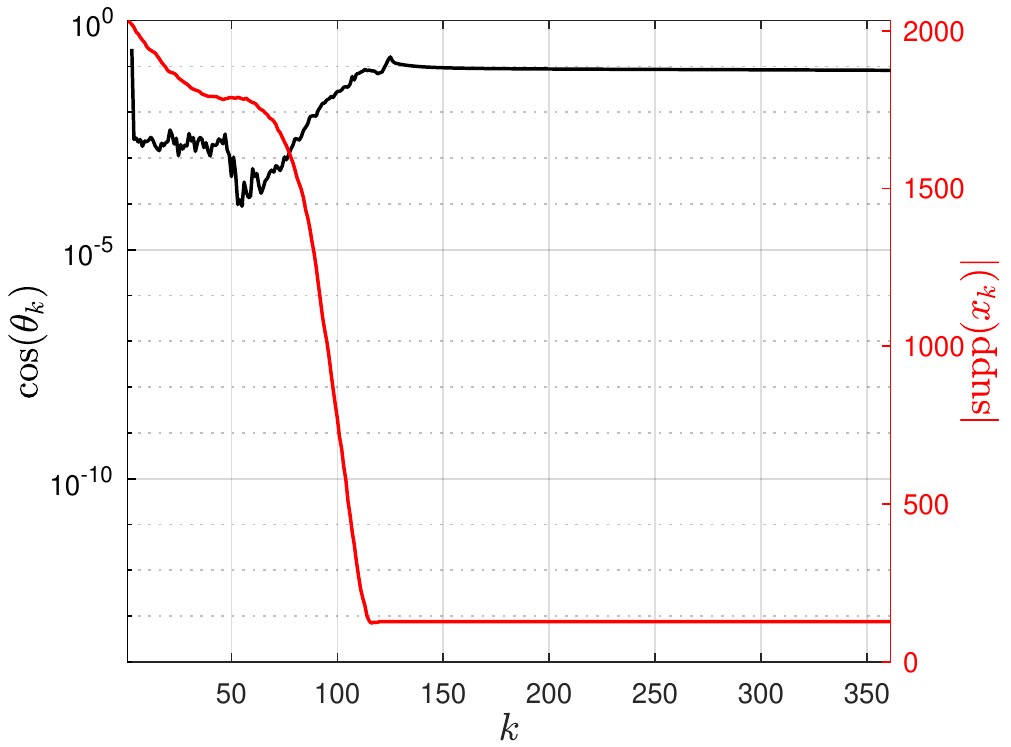} } \hspace{0pt}
	\subfloat[$\ell_{1,2}$-norm]{ \includegraphics[width=0.315\linewidth]{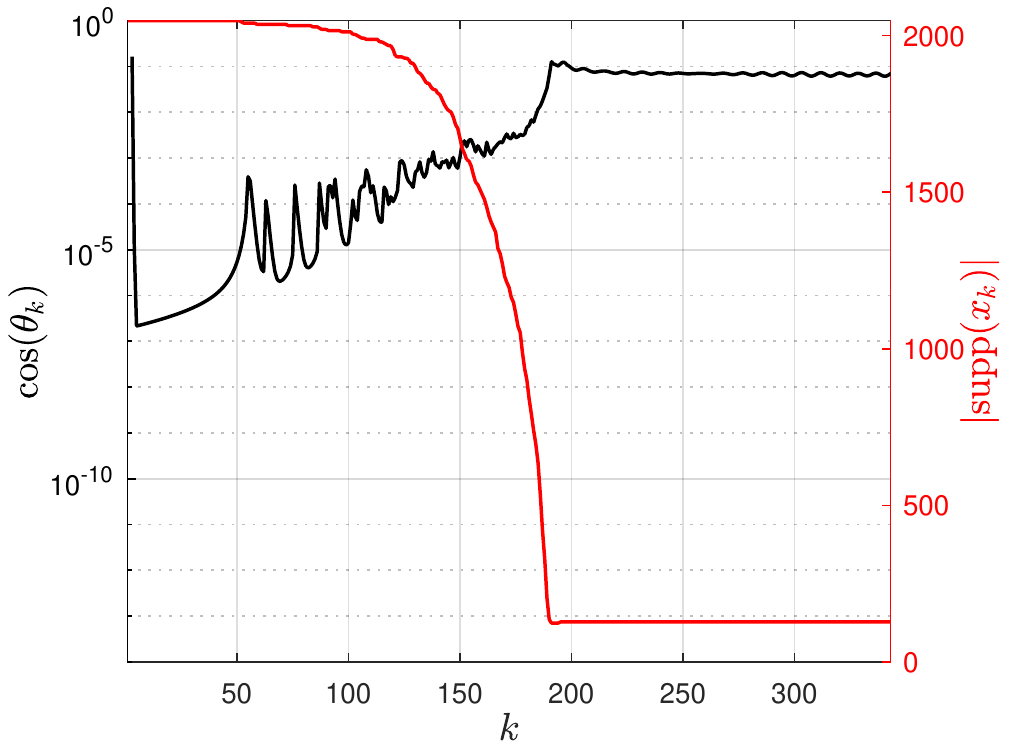} } \hspace{0pt}
	\subfloat[Nuclear norm]{ \includegraphics[width=0.315\linewidth]{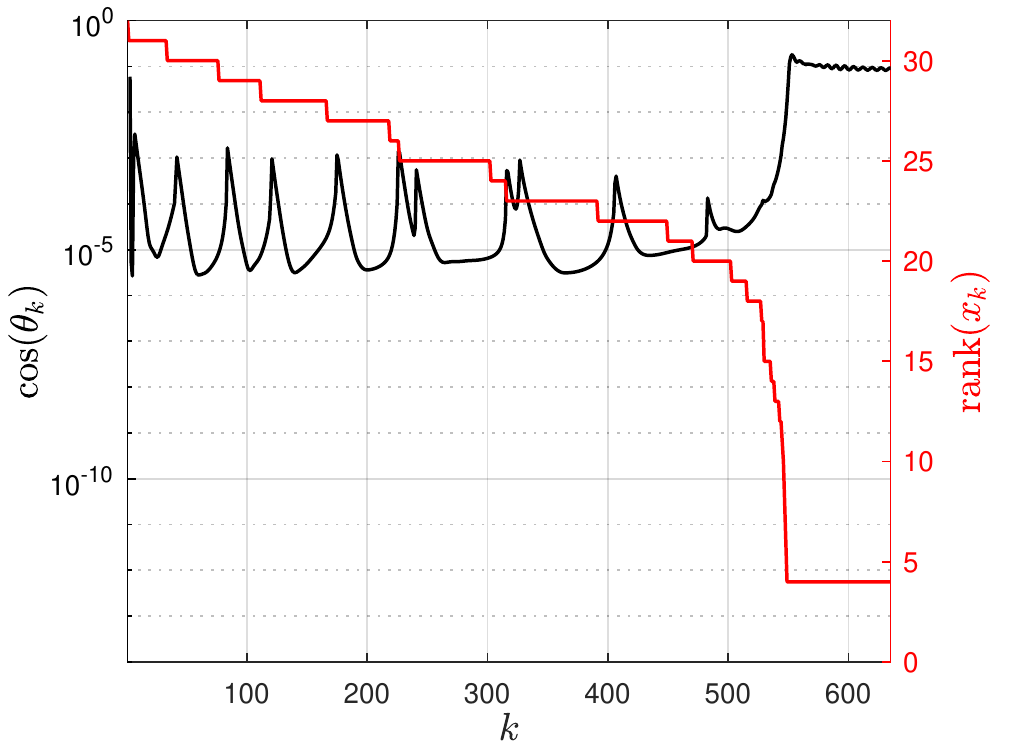} } \\
	\subfloat[$\ell_{1}$-norm]{ \includegraphics[width=0.315\linewidth]{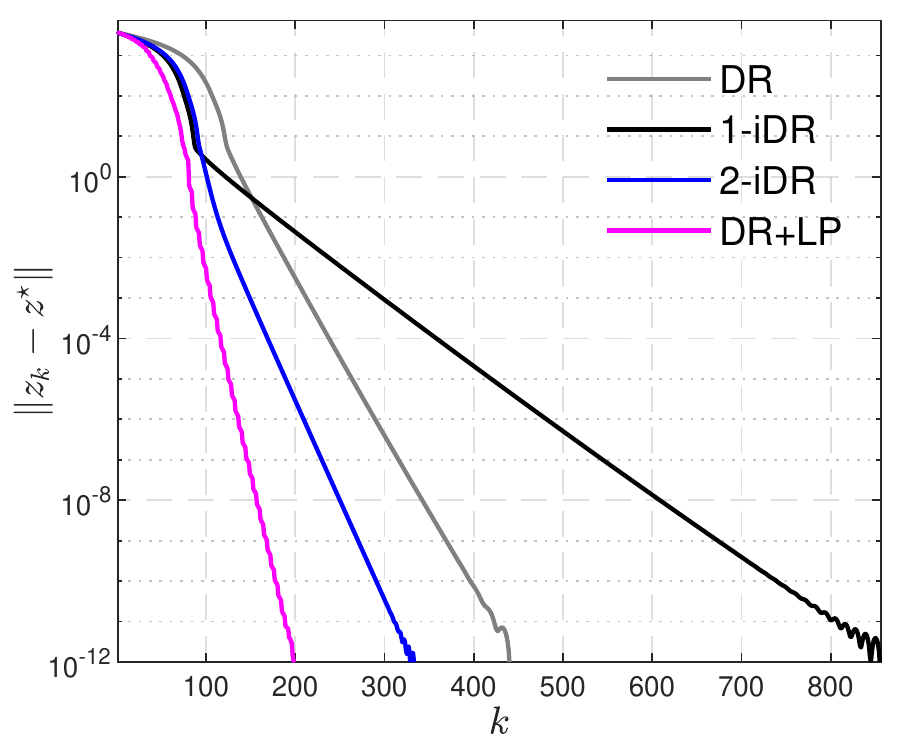} } \hspace{0pt}
	\subfloat[$\ell_{1,2}$-norm]{ \includegraphics[width=0.315\linewidth]{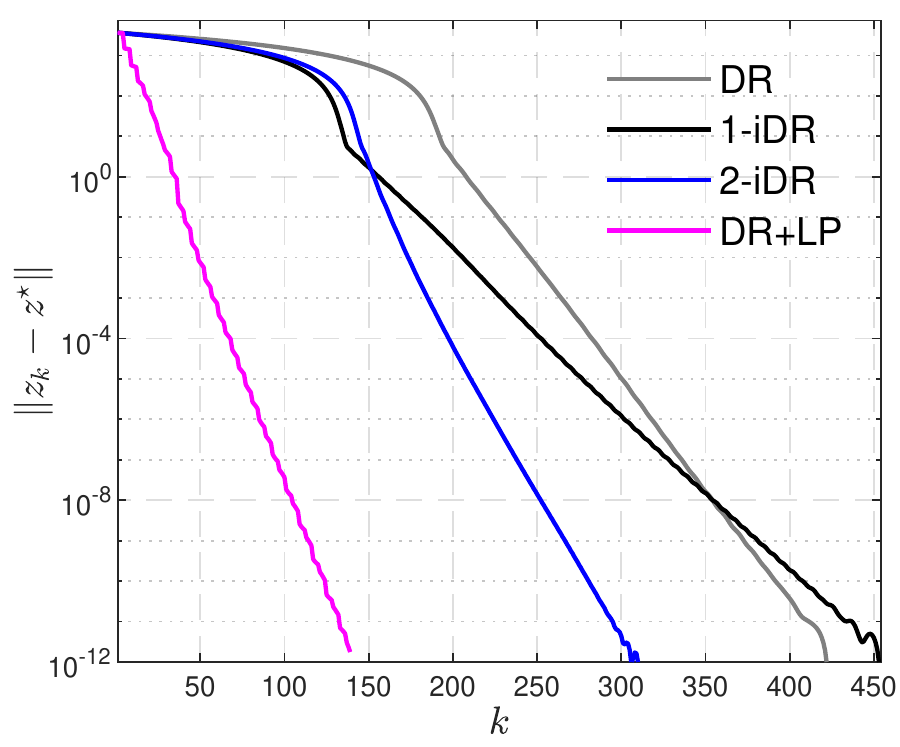} } \hspace{0pt}
	\subfloat[Nuclear norm]{ \includegraphics[width=0.315\linewidth]{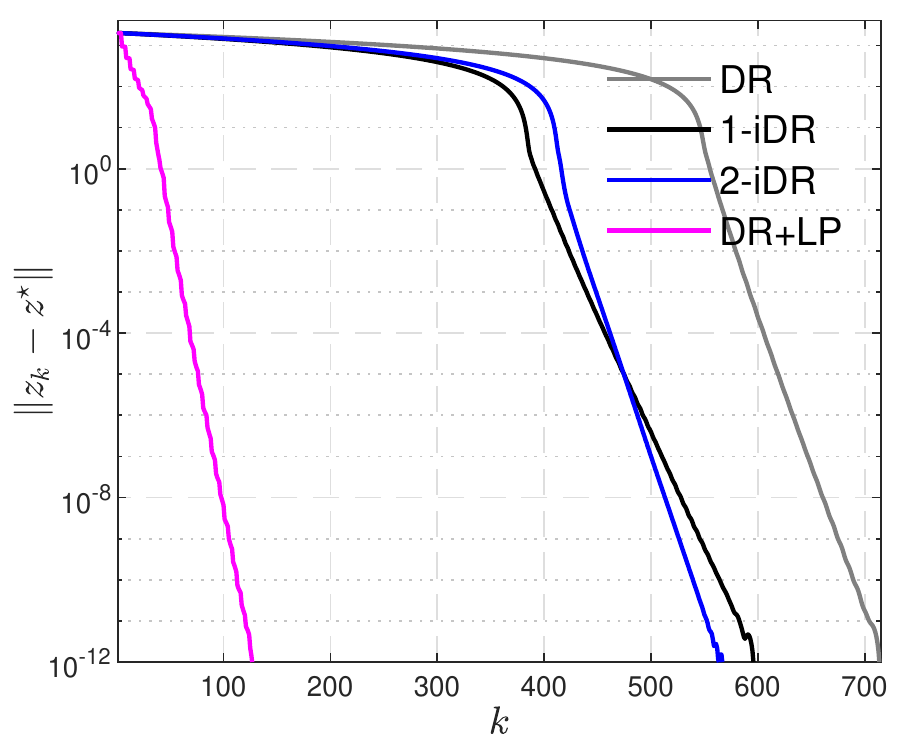} } \\
	\caption{Comparison between methods for solving basis pursuit type problem. } 
	\label{fig:cmp_dr}
\end{figure}


\subsubsection{Basis pursuit type problems}
Now suppose that there is no noise in the observation model \eqref{eq:observation}, \ie $w = 0$. 
Then instead of solving \eqref{eq:R-quadF}, the following equality constrained problem should be considered
\[
\min_{x \in \bbR^{n}}~ R(x)  \qsubjq \calK x = \calK \xob  .
\]
Furthermore, the above constrained problem can  be formulated as 
\beq \label{eq:problem-noise-free}
\min_{x \in \bbR^{n}}~ R(x) + J(x) ,
\eeq
where $J = \iota_{\Omega}(\cdot)$ is the indicator function of the constraint $\Omega \eqdef \ba{x \in \bbR^{n}: \calK \xob  = \calK  x } = \xob+\ker(\calK)$. 
As both functions $R$ and $J$ are non-smooth, a proper choice to solve \eqref{eq:problem-noise-free} is the Douglas--Rachford splitting.
The proximity operator of $J$ is the projection operator onto $\Omega$, which reads $\prox_{\gamma J}(x) = x + \calK^{+}(f - \calK x)$ where $\calK^{+} = \calK^T (\calK\calK^T)^{-1}$ is the Moore-Penrose pseudo-inverse of $\calK$.
%
For $R$, again three examples are considered: $\ell_{1}, \ell_{1,2}$ and nuclear norm, and the settings of each example are
The detailed settings of each example are
\begin{description}[leftmargin=3.0cm]
	\item[{$\ell_{1}$-norm}] $(m,n)=(768,2048)$, $\xob$ has $128$ non-zero elements.
	\item[{$\ell_{1,2}$-norm}] $(m,n)=(640,2048)$, $\xob$ has $32$ non-zero blocks of size $4$.
	\item[{Nuclear norm}] $(m,n)=(640,1024)$, $\xob \in \bbR^{32\times 32}$ and $\rank(\xob) = 4$.
\end{description}
%
The following schemes are compared
\begin{itemize}
\item Douglas--Rachford splitting (DR), the standard two-point inertial DR \eqref{eq:idr} (1-iDR) with $\ak = 0.3$, the three-point inertial DR \eqref{eq:three_point} (2-iDR) with $(\ak,\bk)=(0.5,-0.25)$.
\item Our proposed scheme (LP) with $(q,s) = (4, 100), (4,+\infty)$. 
\end{itemize}
The finite activity identification of $\xk$ and the angle $\cos(\theta_k)$ of Douglas--Rachford splitting is provided in the first row of Figure \ref{fig:cmp_dr}, the observations are quite close to those of Example \ref{eg:exp-dr}. 
The performance comparison of the above methods is presented in the second row of Figure \ref{fig:cmp_dr}, and we observe that
\begin{itemize}
\item For the two inertial schemes: 1-iDR and 2-iDR. Only ``2-iDR'' shows constant better performance than DR. Locally, the convergence speed of ``1-iDR'' is the slowest among all schemes. This observation comply with our discussion in Section \ref{sec:geo_inertial}.  
\item Linear prediction is the fastest among all the schemes, especially for $\ell_{1,2}$-norm and nuclear norm. The main advantage of LP is that it needs much shorter time to identify the manifolds. 
\end{itemize}

\subsubsection{LASSO problem}

We also consider the LASSO problem \eqref{eq:lasso} to demonstrate the performance. 
Three data sets from LIBSVM\footnote{\url{https://www.csie.ntu.edu.tw/~cjlin/libsvmtools/datasets/}} are considered: \texttt{australian}, \texttt{mushrooms} and \texttt{covtype}.  
The observation are shown in Figure~\ref{fig:cmp_dr_lasso}, we can see that linear prediction shows clear advantages over the compared ones. 

Note that for the two inertial schemes, they are better than the standard Douglas--Rachford splitting method for all the three examples, which is different from the affine constrained problem considered above. 

\begin{figure}[!ht]
	\centering
	\subfloat[\texttt{australian}]{ \includegraphics[width=0.315\linewidth]{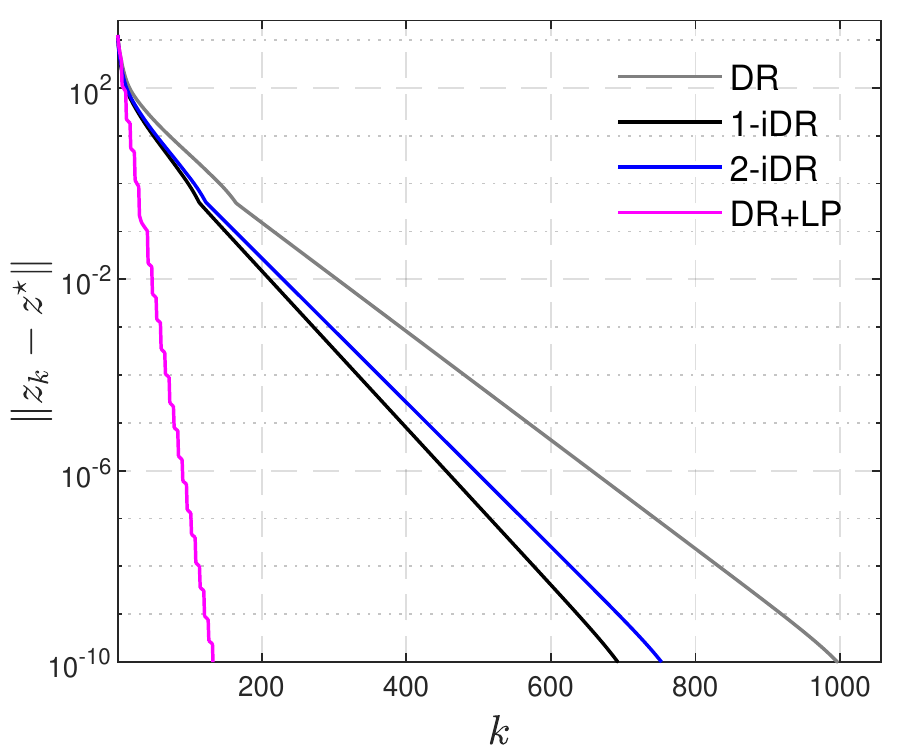} }  \hspace{0pt}
	\subfloat[\texttt{mushrooms}]{ \includegraphics[width=0.315\linewidth]{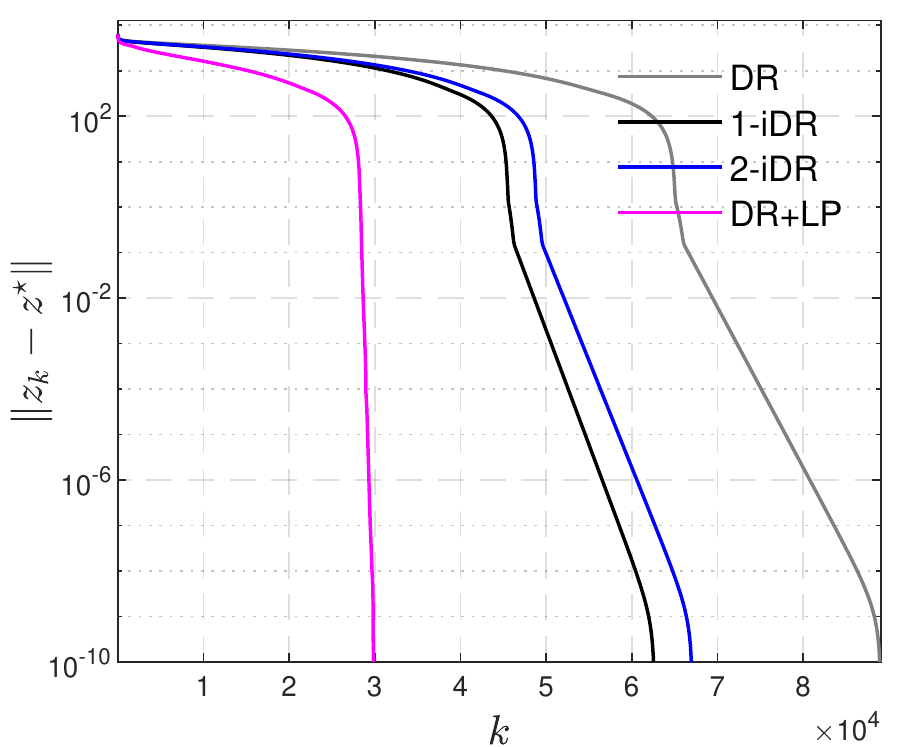} }  \hspace{0pt}
	\subfloat[\texttt{covtype}]{ \includegraphics[width=0.315\linewidth]{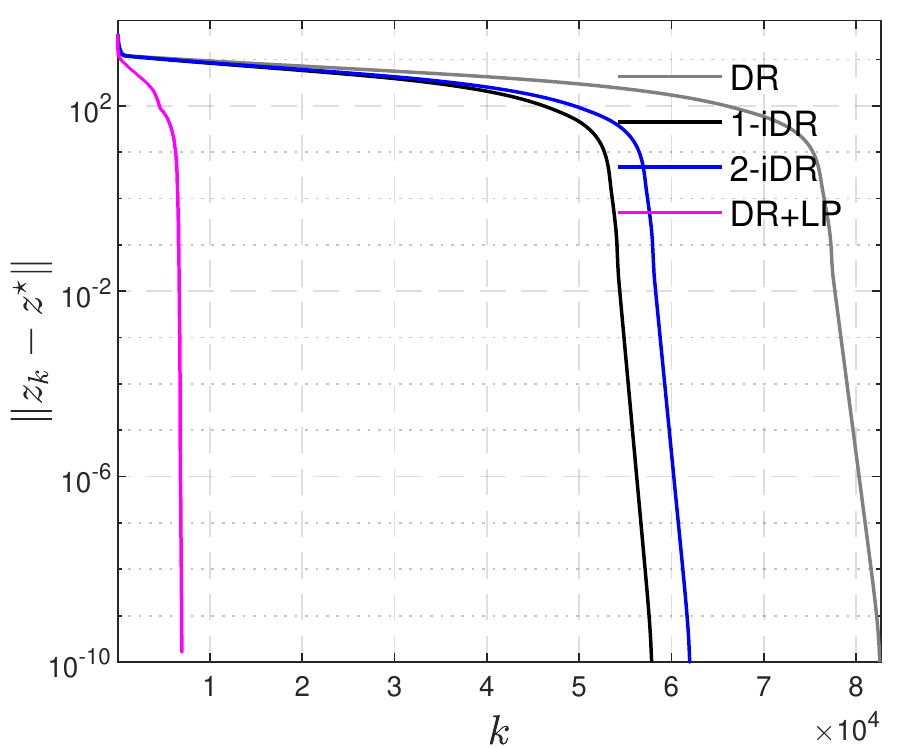} }  \\
	\caption{Comparison of Douglas--Rachford schemes for solving LASSO problem.}
	\label{fig:cmp_dr_lasso}
\end{figure}


\subsection{Primal--Dual splitting}

The third example we consider is the Primal--Dual splitting method. Adapt \aafom to the method, we obtain an adaptive acceleration scheme for Primal--Dual splitting method which is described in Algorithm \ref{alg:a2pd}. Note that the fixed-point sequence of Primal--Dual splitting method is the augmented variable $\zk$ defined in \eqref{eq:A-B-calV}.

To demonstrate the performance of Algorithm \ref{alg:a2pd}, a medical imaging problem of the following form is considered
\[
\min_{x\in\bbR^n}~ \norm{\calW x}_1+ \sfrac{\lambda}{2}\norm{\calK x - f}^2 ,
\]
where $\calK$ is a subsampled Fourier transform operator, $f$ is the measurement and $\calW$ a redundant wavelet frame. 
We compare the standard Primal--Dual splitting, inertial Primal--Dual splitting and our proposed accelerated one with $(q,s) = (1, \pinf) , (2, \pinf)$, 
The numerical result is shown in Figure \ref{fig:cmp_pd_mri}. 
\begin{itemize}
	\item Image quality wise, LP provides much better reconstruction than the plain Primal--Dual splitting and its inertial version, especially for $q=2$.
	\item In terms of PSNR in Figure \ref{fig:cmp_pd_mri} (d), LP also yields better PSNR value than the (inertial) Primal--Dual splitting methods. 
\end{itemize}

\begin{center}
	\begin{minipage}{0.975\linewidth}
		\begin{algorithm}[H]
			\caption{\aapd: Adaptive Acceleration for Primal--Dual splitting} \label{alg:a2pd}
			\KwIn{$\gammaR, \gammaJ > 0$ such that $\gammaR\gammaJ\norm{L}^2 < 1$ and $\tau \in [0,1]$. Let $s\geq1, q \geq 1$ be integers.}
			{\noindent{\bf{Initial}}}: $\xbar_{0} = x_0 \in \bbR^n ,~ \wbar_0 = w_0 \in \bbR^m$. Let $\textstyle z_{0} = \begin{pmatrix} x_{0} \\ v_{0} \end{pmatrix}$ and $V_{0} \in \bbR^{(m+n)\times q}$. \\ 
			{\noindent{\bf{Repeat}}}: 
			\begin{itemize}[leftmargin=2em]
				\item \textrm{If $\textrm{mod}(k, q+2)=0$: Compute $C_k$ via \eqref{eq:coeff_mtx}, if $\rho(C_k)<1$: $\ek = \calE_{s, q}(\zk, \dotsm, z_{k-q-1})$.} $$ \begin{aligned} \xbark = \xk + \ak e_{k,(1:n)} \qandq  \wbark = \wk + \ak e_{k,(n+1:m+n)} . \end{aligned} $$ 
				
				\item \textrm{If $\textrm{mod}(k, q+2)\neq0$: } $ \xbark = \xk $ and $\wbark = \wk$. 
				
				\item For $k \geq 1$: 
				\[
				\begin{aligned}
				\xkp &= \prox_{\gammaR R} \pa{ \xbark - \gammaR L^T \wbark } , \\
				\tilde{x}_{k+1} &= \xkp + \tau(\xkp - \xbark) , \\
				\wkp &= \prox_{\gammaJ J^*} \pa{ \wbark  + \gammaJ L \tilde{x}_{k+1} } , \\
				\zkp &= \begin{pmatrix} \xkp \\ \wkp \end{pmatrix} ,\quad \vkp = \zkp - \zk \qandq V_{k+1} = [\vkp| \vk | \dotsm | v_{k-q+2}] .
				\end{aligned}
				\]
				
			\end{itemize}
			{\noindent{\bf{Until}}}: $\norm{\vk} \leq \tol$. 
		\end{algorithm}
	\end{minipage}
\end{center}

\begin{figure}[!ht]
	\centering
	\subfloat[Original phantom]{ \includegraphics[width=0.3\linewidth]{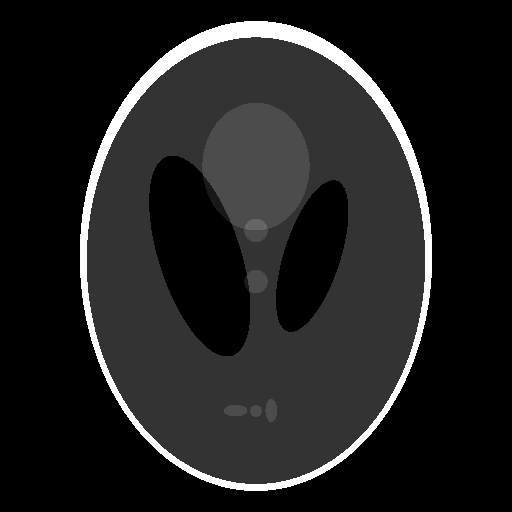} }  \hspace{2pt}
	\subfloat[Primal--Dual splitting]{ \includegraphics[width=0.3\linewidth]{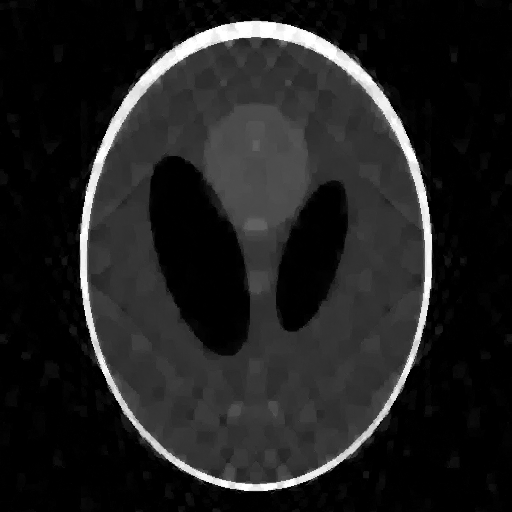} }  \hspace{2pt}
	\subfloat[Inertial Primal--Dual]{ \includegraphics[width=0.3\linewidth]{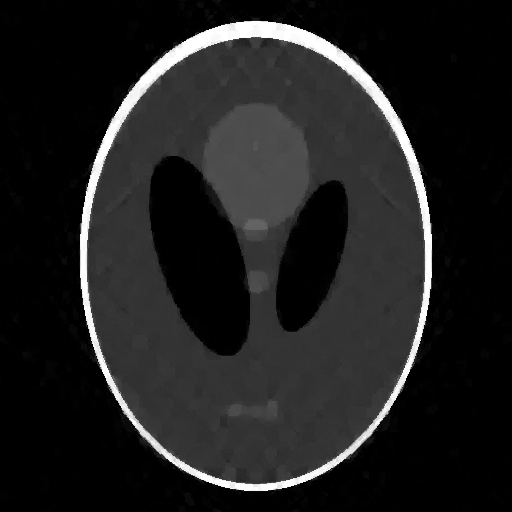} }  \\
	\subfloat[PSNR]{ \includegraphics[width=0.325\linewidth]{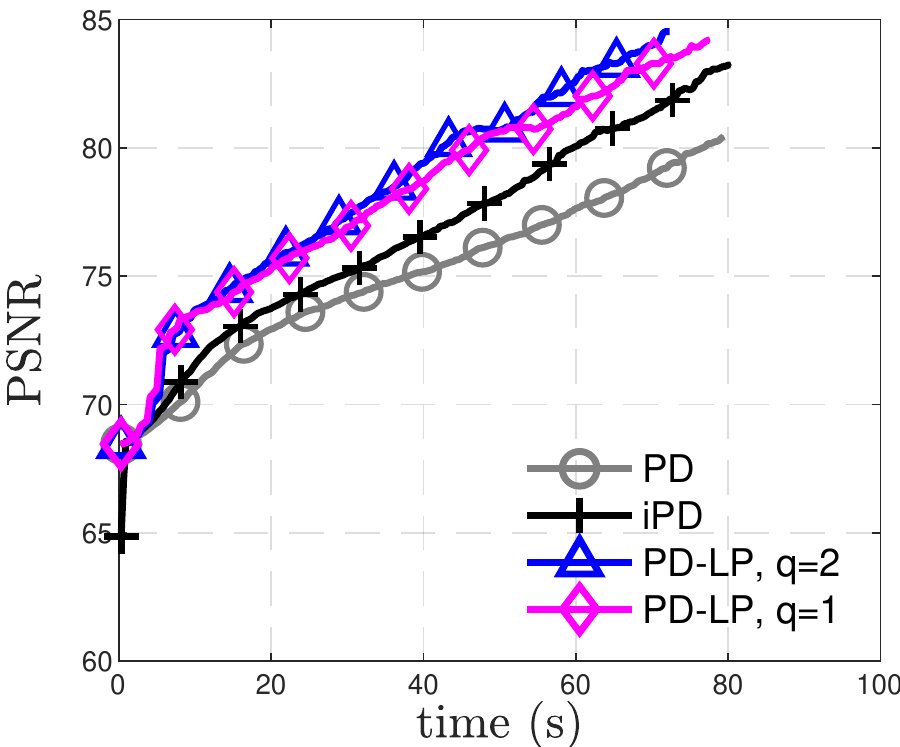} }  \hspace{2pt}
	\subfloat[LP, $q = 1$]{ \includegraphics[width=0.3\linewidth]{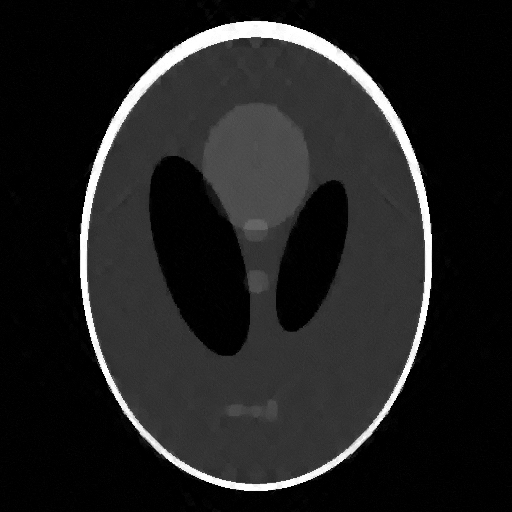} }  \hspace{2pt}
	\subfloat[LP, $q = 2$]{ \includegraphics[width=0.3\linewidth]{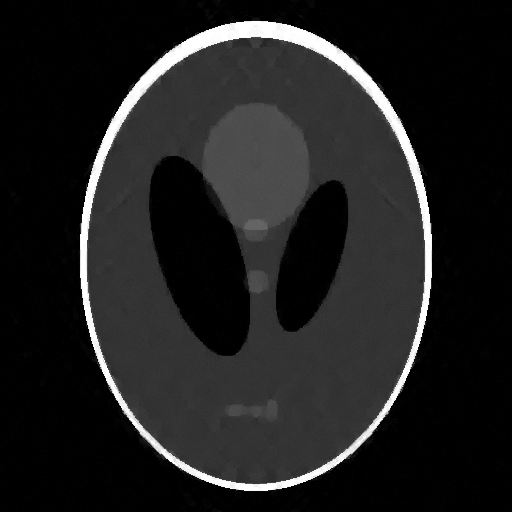} }  \\
	\caption{Comparison of Primal--Dual schemes for MRI reconstruction. (d) Original Shepp–Logan phantom; (b) Output of Primal--Dual splitting; (c) Output of inertial Primal--Dual splitting; (d) PSNR comparison; (e) Output of Algorithm \ref{alg:a2pd} with $(q,s) = (1,+\infty)$; (c) Output of Algorithm \ref{alg:a2pd} with $(q,s) = (2,+\infty)$. }
	\label{fig:cmp_pd_mri}
\end{figure}



\subsection{Generalized Forward--Backward splitting}

For the problem \eqref{eq:problem-fb}, suppose now there are more than $1$ non-smooth functionals: let $\mm$ be a positive integer and consider
\beq\label{eq:problem-gfb}\tag{$\mathcal{P}_{\mathrm{GFB}}$}
\min_{x\in\bbR^n} \Ba{ \Phi_{\mm}(x)  \eqdef  F(x) + \msum_{\ii=1}^{\mm} R_{\ii}(x) }  ,
\eeq
where $F$ is continuous differentiable with $\nabla F$ being $L$-Lipschitz and $R_{\ii} \in \lsc(\bbR^n)$ for each $\ii=1,...,\mm$.

Forward--Backward splitting is no longer feasible for this problem, as in general there is no close form solution for the proximity mapping of $\sum_{\ii=1}^{\mm} R_{\ii}(x)$ even if each $R_{\ii}$ is simple. In \cite{gfb2011}, the authors proposed a generalized Forward--Backward splitting algorithm (GFB) to overcome the challenge. GFB achieves the full splitting of the evaluation of the proximity operator of each $R_{\ii}$. 
Let $(\omega_{\ii})_{\ii}\in ]0,1[^{\mm}$ such that $\sum_{\ii=1}^{\mm}\omega_{\ii}=1$, choose $\gamma \in ]0, 2\beta[$: 
\beq\label{eq:gfb}
\begin{aligned}
& \textrm{from $\ii={1}$ to $\mm$:} \\
&\left\lfloor
\begin{aligned}
u_{\ii,k+1} &= \prox_{\frac{\gamma}{\omega_{\ii}} R_{\ii}} \Pa{ 2\xk-z_{\ii,k}-\gamma \nabla F (\xk)  }  \\
z_{\ii,k+1} &= z_{\ii,k} +  \pa{ u_{\ii,k+1}  - \xk }  
\end{aligned}
\right. \\
&\xkp = \msum_{\ii=1}^{\mm} \omega_{\ii} z_{\ii,k+1}  .
\end{aligned}
\eeq
Under a properly defined product space $\calH$, there exists a non-expansive operator $\fGFB: \calH \to \calH$ such that
\[
\bmz_{k+1} = \fGFB (\bmz_{k}) 
\]
with $\bmz_{k} = \begin{pmatrix} z_{1,k} \\ \vdots \\ z_{\mm,k} \end{pmatrix}$. 
We refer to \cite{gfb2011} for more details of the GFB algorithm. 
Specializing \aafom to the case of GFB, we obtain the accelerated GFB scheme described in Algorithm \ref{alg:a2gfb}. 

\begin{center}
\begin{minipage}{0.975\linewidth}
\begin{algorithm}[H]
\caption{\aagfb: Adaptive Acceleration for generalized Forward--Backward splitting} \label{alg:a2gfb}
\KwIn{$(\omega_{j})_{j}\in ]0,1[^{\mm}~\mathrm{such~that}~\sum_{j=1}^{\mm}\omega_{j}=1$, $\gamma\in]0,2/L[$. Let $s\geq1, q \geq 1$ be integers.}
{\noindent{\bf{Initial}}}: for $i=1,...,r$, $\zbar_{i,0} = z_{0} \in \bbR^n$ and $x_{0} = \msum_{\ii=1}^{\mm} \omega_{\ii} \zbar_{\ii,0}$, $V_{i,0} = 0 \in \bbR^{n\times q}$\;
{\noindent{\bf{Repeat}}}: 
\begin{itemize}[leftmargin=2em]
\item \textrm{If $\textrm{mod}(k, q+2)=0$: for $\ii={1},...,\mm$, compute $C_k^{i}$ via \eqref{eq:coeff_mtx}, if $\rho(C_k^{i})<1$: } $$\zbar_{\ii,k} = z_{\ii,k} + \ak^{i} \calE_{s, q}(z_{\ii,k}, \dotsm, z_{\ii, k-q-1}).$$
	
\item \textrm{If $\textrm{mod}(k, q+2)\neq0$: } $ \zbar_{\ii,k} = z_{\ii,k}$. 

\item For $k \geq 1$: 
\[
\begin{aligned}
&\xk = \msum_{\ii=1}^{\mm} \omega_{\ii} \zbar_{\ii,k}  , \\
& \textrm{from $\ii={1}$ to $\mm$:} \\
&\left\lfloor
\begin{aligned}
u_{\ii,k+1} &= \prox_{\frac{\gamma}{\omega_{\ii}} R_{\ii}} \Pa{ 2\xk-\zbar_{\ii,k}-\gamma \nabla F (\xk)  } ,   \\
z_{\ii,k+1} &= \zbar_{\ii,k} +  \pa{ u_{\ii,k+1}  - \xk }  , \\
v_{\ii,k+1} &= z_{\ii,k+1} - z_{\ii,k} \qandq V_{i,k+1} = [v_{\ii,k+1}| v_{\ii,k} | \dotsm | v_{\ii,k-q+2}] . 
\end{aligned}
\right. \\
\end{aligned}
\]

\end{itemize}
{\noindent{\bf{Until}}}: $\sum_{i}\norm{\vk^{i}} \leq \tol$.
\end{algorithm}
%
%
\end{minipage}
\end{center}
%

We consider the Principal Component Pursuit (PCP) problem \cite{candes2011robust} to demonstrate the performance comparison. %
Different from \eqref{eq:observation}, the forward observation model of PCP problem reads, 
\[
b = \cx_{{_L}} + \cx_{{_S}} + \omega,
\]
where $\cx_{{_L}}$ is low-rank, $\cx_{{_S}}$ is sparse, and $b, \omega$ are the observation and noise respectively.
The PCP proposed in \cite{candes2011robust} attempts to provably recover $(\cx_{{_L}},\cx_{{_S}})$ up to a good approximation, by solving a convex optimization. 
Here, we also add a non-negativity constraint to the low-rank component, which leads to the following convex problem
\beq
\label{eq:rpca}
\min_{x_{_L}, x_{_S} \in \bbR^{n\times n}}~ \qfrac{1}{2}\norm{b-x_{_L}-x_{_S}}^2+\mu_1\norm{x_{_S}}_1 + \mu_2\norm{x_{_L}}_*+\iota_{P_+}(x_{_L}) .
\eeq
%
Observe that for given an $x_{_L}$, the minimizer of \eqref{eq:rpca} is $x_{_S}^\star=\prox_{\mu_1{\norm{\cdot}_1}}(b - x_{_L})$. Thus, \eqref{eq:rpca} is equivalent to
\beq
\label{eq:rpcame}
\min_{x_{_L}\in\bbR^{n\times n}}~^1\Pa{\mu_1\norm{\cdot}_1}(b-x_{_L}) + \mu_2\norm{x_{_L}}_*+\iota_{P_+}(x_{_L}),
\eeq
where $^1\Pa{{\mu_1\norm{\cdot}_1}}(b-x_{_L})$ is the Moreau Envelope of $\mu_1{\norm{\cdot}_1}$.

The numerical comparison on a synthetic example is shown below in Figure \ref{fig:cmp_gfb}, and the observations are very similar to those of the previous examples, that linear prediction shows clear advantages over the standard method and its inertial version. 

\begin{figure}[!ht]
\centering
\subfloat[Iteration comparison]{ \includegraphics[width=0.45\linewidth]{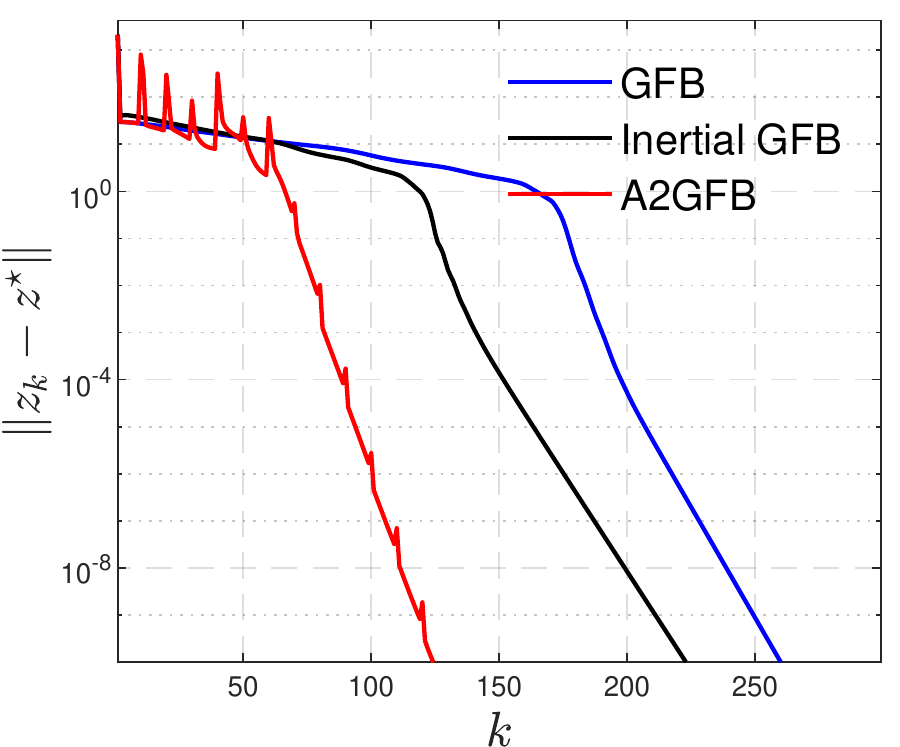} } \hspace{2pt}
\subfloat[CPU time]{ \includegraphics[width=0.45\linewidth]{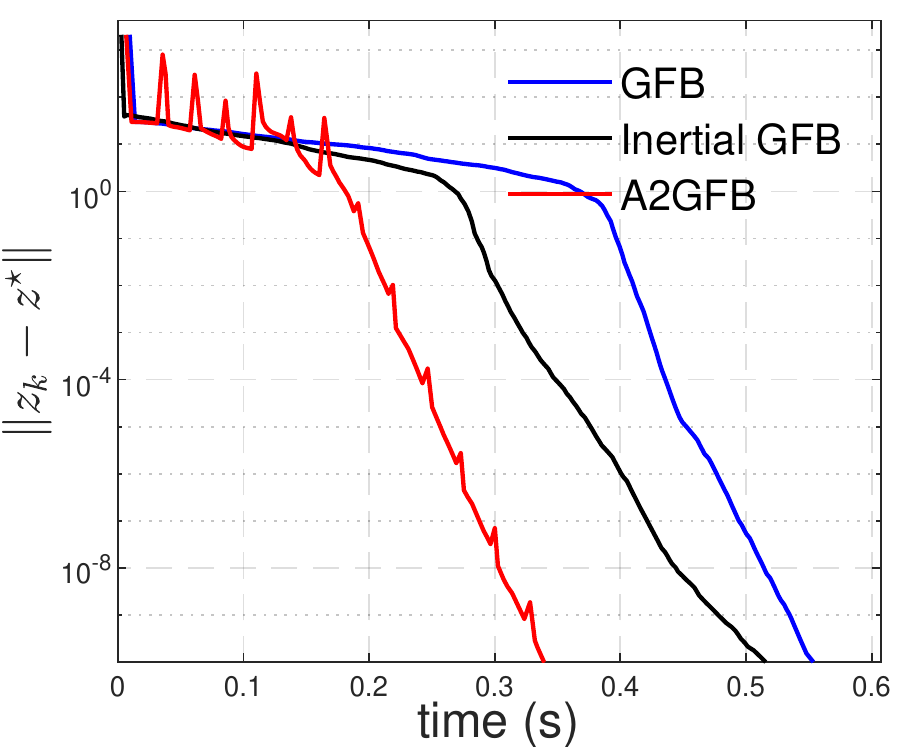} } \hspace{2pt}
\caption{Comparison of GFB and adaptive accelerated GFB on synthetic data.} 
\label{fig:cmp_gfb}
\end{figure}

\section{Conclusions} 

In this article, we studied the local geometry of first order methods for nonsmooth optimization. Our analysis provides insight as to when inertial techniques can be applied, in particular, the outcome of inertial depends not only on the structure of the optimization problem to solve but also the geometry of the first order method itself. Based on our trajectory analysis, we propose a trajectory-following acceleration scheme, which provides an alternative perspective on  classical vector extrapolation techniques.

\subsection*{Acknowledgement}

We would like to thank Arieh Iserles for pointing out the connection between linear prediction and vector extrapolation. 
Jingwei Liang was partly supported by Leverhulme trust, Newton trust, the EPSRC centre ``EP/N014588/1'' and the Cantab Capital Institute for the Mathematics of Information (CCIMI).



\begin{small}
\bibliographystyle{plain}
\bibliography{bib}
\end{small}


\begin{small}

\appendix 
\addtocontents{toc}{\protect\setcounter{tocdepth}{0}} 


\section{Trajectory of linear systems}\label{sec:trajectory-ls}
{In this section, we study the trajectories of 3 different types of matrices. These matrices correspond exactly to the  linearisation matrices in the cases of Forward--Backward splitting, Douglas--Rachford splitting and Primal--Dual splitting and hence, our analysis here can be used as a guide  for  the geometry of these methods.}

Let $M \in \bbR^{n\times n}$ be a bounded real matrix and consider the following linear system 
\beq\label{eq:lineq-xk}
\xkp = M \xk  ,
\eeq
which generates a train of sequence $\seq{\xk}$. Assumed $\seq{\xk}$ is convergent, \ie there exists an $\xsol \in \bbR^n$ such that $\xk \to \xsol$. The goal of this section is to investigate the properties of the trajectory formed by $\seq{\xk}$. 
To this end, define ${\vk} = \xk - \xkm$, it is immediate that \eqref{eq:lineq-xk} leads to the following iteration in terms of $\vk$,
\beq\label{eq:lineq-vk}
\vkp = M \vk ,
\eeq
and $\lim_{k\to+\infty} \vk = 0$ since $\xk \to \xsol$. 
%
%
%
%
To characterize the trajectory, we choose to use the angle between each two adjacent vectors $\vk$ and $\vkm$, which is denoted by $\theta_k$ and defined by
\beqn 
\theta_k 
\eqdef \angle(\vk, \vkm) 
= \arccos\Ppa{ \sfrac{\iprod{\vk}{\vkm}}{\norm{\vk}\norm{\vkm}} } .
\eeqn
For the rest of this section, we discuss the property of $\seq{\theta_k}$ under three different choices of matrix $M$. 


\subsection{Type I linear system}

We start with the simplest case, that $M \in \bbR^{n\times n}$ is symmetric. Let $(\sigma_{i})_{i=1,...,n} \in \bbR^n$ be the eigenvalues of $M$, which are all real owing to the symmetry of $M$.

\begin{definition}[Type I matrix]\label{def:type-i}
$M \in \bbR^{n\times n}$ is symmetric with all its eigenvalues in $]-1, 1]$, moreover $1 \geq \sigma_1 \geq \sigma_2 \geq \dotsm \geq \sigma_{n}$ and $\sigma_1 > \abs{\sigma_{n}} > 0$.
\end{definition}

Denote $\eta$ the ratio between the second largest eigenvalue in magnitude and $\sigma_1$, \ie $\eta = \frac{\max\ba{\sigma_2, \abs{\sigma_n}}}{\sigma_1}$. 

\begin{proposition}
\label{prop:type-i}
Consider the linear system \eqref{eq:lineq-vk} where $M$ is a Type I matrix defined in Definition \ref{def:type-i}, then there holds 
$ 1 -  \cos(\theta_k) = O(\eta^{2k}) $. 
\end{proposition}
\begin{remark} 
Proposition \ref{prop:type-i} implies eventually the trajectory of $\seq{\xk}$ is a straight line. If $\sigma_1 < \abs{\sigma_n} < 1$, then it can be shown that $\lim_{k\to\pinf} \theta_k = \pi$. 
\end{remark}

\begin{example}
Let $U$ be an orthogonal matrix in $\bbR^{3\times 3}$, and consider 
$M = U
\begin{bmatrix} a &  &  \\  & b &  \\  &  & c \end{bmatrix}
U^T $ 
where $-1 < c \leq b \leq a \leq 1$ are the eigenvalues. 
Two different choices of $(a,b,c)$ are considered
\[
(a,b,c) \in \Ba{ 0.99\times(1,0.98, 0.9), 0.99\times(1, 0.98,-0.75) }  .
\]
For both cases, we have $\eta = 0.98$, hence same convergence rates of $\cos(\theta_k)$ to $1$, see Figure \ref{fig:type-i} (a). 
The trajectories of $\seq{\xk}$ are shown in the other two figures of Figure \ref{fig:type-i}: for figure (b) all the three eigenvalues of $M$ are in $[0, 1]$, for figure (c) the smallest eigenvalue of $M$ is negative.

\begin{figure}[H]
\centering
\subfloat[${1-\cos(\theta_k)}$]{ \includegraphics[width=0.3\linewidth]{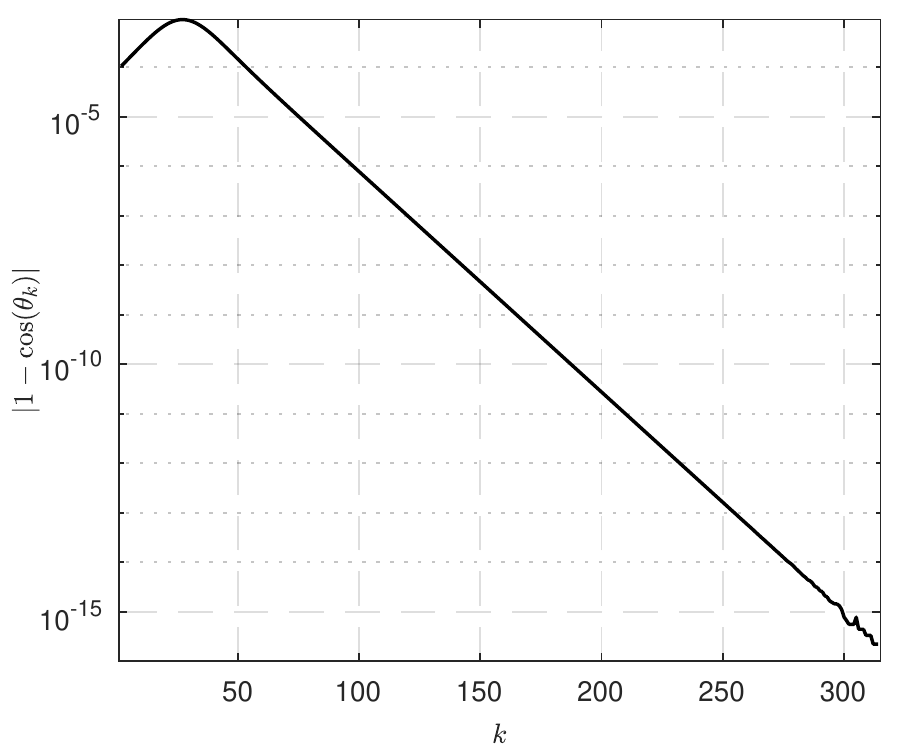} } {\hspace{2pt}}
\subfloat[{Eigenvalues in $[0, 1]$}]{ \includegraphics[width=0.3\linewidth]{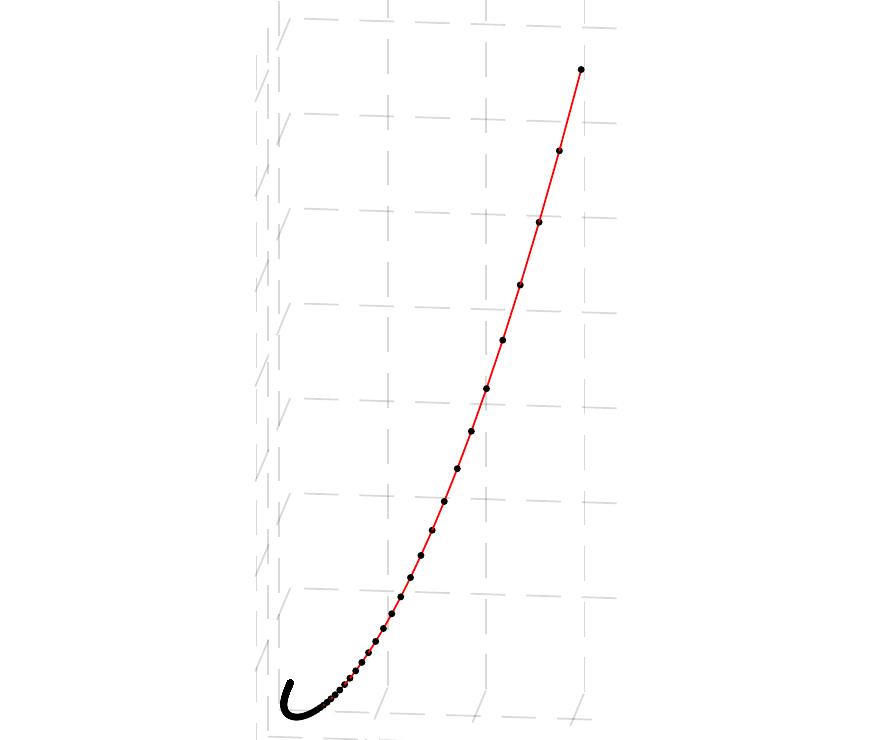} } {\hspace{2pt}}
\subfloat[{Eigenvalues in $]-1, 1]$}]{ \includegraphics[width=0.3\linewidth]{figure/tracjectory-type-i-B.pdf} }\\
\caption{{\small Convergence of $\cos(\theta_k)$ and trajectories of $\seq{\xk}$. (a): Convergence of $\cos(\theta_k)$ to $1$; (b) Trajectory of $\seq{\xk}$ when all the eigenvalue of $M$ are real; (c) Trajectory of $\seq{\xk}$ when $M$ has negative eigenvalues.}}
\label{fig:type-i}
\end{figure}

\end{example}

\begin{proof}[Proof of Proposition \ref{prop:type-i}]

Since $M$ is symmetric, there exists a real orthogonal matrix $U$ such that
$ M = U \Sigma U^T $ 
where $\Sigma = \diag((\sigma_i)_{i=1,...,n})$ is a diagonal matrix, and
$
\vk 
= M \vkm 
= M^k v_0 
= U \Sigma^k U^T v_0  
$. 
Let $\uk = U^T\vk$, then $\uk = \Sigma^k u_0$. 
Suppose there exists $d \in [2, n[$ such that $\sigma = \sigma_1=\sigma_2 = \dotsm = \sigma_d > \sigma_{d+1}$, we can consider the following decomposition of $\Sigma$
\beq\label{eq:Sigma-decom}
\Sigma_1 
\eqdef
\begin{bmatrix}
\diag((\sigma_i)_{i=1,...,d}) & 0_{d\times (n-d)} \\
0_{(n-d)\times d} & 0_{n-d}
\end{bmatrix}
\qandq
\Sigma_2
\eqdef
\begin{bmatrix}
0_d & 0_{d\times (n-d)} \\
0_{(n-d)\times d} & \diag((\sigma_i)_{i=d+1,...,n})
\end{bmatrix}  .
\eeq
It is immediate that $\uk = \Sigma_1^k u_0 + \Sigma_2^k u_0$, and that 
\[
\sfrac{1}{\sigma} \Sigma_1
= \begin{bmatrix}
\Id_{d} & 0_{d\times (n-d)} \\
0_{(n-d)\times d} & 0_{n-d}
\end{bmatrix}
\qandq
\sfrac{1}{\sigma}\Sigma_2
= \eta
\begin{bmatrix}
0_d & 0_{d\times (n-d)} \\
0_{(n-d)\times d} & \diag((\tfrac{\sigma_i}{\sigma_{d+1}})_{i=d+1,...,n})
\end{bmatrix}  .
\]
Moreover, there holds $\frac{1}{\sigma^k} \Sigma_1^k = \frac{1}{\sigma} \Sigma_1$ and $\frac{1}{\sigma^k}\Sigma_2^k = O(\eta^k)$. 
Consider the following orthogonal decomposition of $\frac{\uk}{\sigma^k}$, 
\[
s_k = \sfrac{1}{\sigma^k} \Sigma_1^k u_0  = \sfrac{1}{\sigma} \Sigma_1 u_0 
\qandq
t_k = \sfrac{\uk}{\sigma^k} - s_k = O(\eta^k) .
\] 
We get
\[
\iprod{\vk}{\vkm}
= \iprod{ U^T\vk  }{ U^T \vkm }
= \sigma^{2k-1} \iprod{ \tfrac{\uk}{\sigma^k}  }{ \tfrac{\ukm}{\sigma^{k-1}} }
= \sigma^{2k-1} \iprod{ s_k + \tk  }{ s_{k-1} + \tkm } ,
\]
and $\norm{\vk} = \norm{\uk} = \sigma^k (\norm{s_k + t_k}) = \sigma^k (\norm{s_k} + \norm{t_k})$. Consequently the value of $\cos(\theta_k)$ is, note that $s_k = s_{k-1}$
\beq\label{eq:cos-thetak-i}
\begin{aligned}
\cos(\theta_k)
= \sfrac{ \iprod{ \vk }{ \vkm } }{\norm{ \vk }\norm{ \vkm }}
= \sfrac{ \iprod{ s_k + t_k }{ s_{k-1} + t_{k-1} } }{\norm{ s_k + t_k }\norm{ s_{k-1} + t_{k-1} }}
&= \sfrac{ \iprod{ s_k }{ s_{k-1} } }{\norm{ s_k + t_k }\norm{ s_{k-1} + t_{k-1} }} + \sfrac{ \iprod{ t_k }{ t_{k-1} } }{\norm{ s_k + t_k }\norm{ s_{k-1} + t_{k-1} }}\\
&= \sfrac{ \norm{ s_k }^2 }{\norm{ s_k + t_k }\norm{ s_{k-1} + t_{k-1} }} + O(\eta^{2k-1})\\
&= \sfrac{ \norm{s_{k}}^2 }{\norm{ s_{k} }^2 + \norm{t_k }^2}  \times \sfrac{\norm{s_{k} + t_k}}{\norm{ s_{k} + t_{k-1} }} + O(\eta^{2k-1}) .
\end{aligned}
\eeq
Since we have
\[
\sfrac{ \norm{s_{k}}^2 }{\norm{ s_{k} }^2 + \norm{t_k }^2} = 1 - \norm{t_k }^2  + O(\norm{t_k }^4) = 1 + O(\eta^{2k}) 
\qandq
\sfrac{\norm{s_{k} + t_k}}{\norm{ s_{k} + t_{k-1} }} \to 1 .
\]
Combining with \eqref{eq:cos-thetak-i} leads to the claimed result.  \qedhere

\end{proof}

\subsection{Type II linear system}

From this part, we turn to linear systems which result in spiral trajectories. The first of this kind is the normal matrix.

\begin{definition}[Type II matrix]\label{def:type-ii}
$M \in \bbR^{n\times n}$ is normal matrix with all its eigenvalues lying in the complex unit disc. 
\end{definition}

According to \cite[Theorem 2.5.8]{horn1990matrix}, a normal matrix $M \in \bbR^{n\times n}$ is quasi-diagonalizable, that is there exists a real orthogonal matrix $U \in \bbR^{n\times n}$ such that
\[
M = 
U 
\begin{bmatrix}
B_1 & & \\
& \ddots &  \\
&  &   B_{m}\\
\end{bmatrix}
U^T .
\]
For each $i=1,...,m$, $B_i$ is either real valued scalar or $2 \times 2$ matrix of the form $\begin{bmatrix} a_i & b_i \\ -b_i & a_i \end{bmatrix}$ in which $b_{i} > 0$ and has eigenvalues $a_{i} \pm \mathrm{i} b_{i}$. 
We impose the following assumptions on $M$.

\begin{assumption}\label{assumption-type-ii}
For each $i=1,...,m$,
\begin{enumerate}[label={\rm {(\roman{*})}},leftmargin=3.5em]
\item if $B_i$ is scalar, then $B_i \in \ba{0, 1}$; 
\item if $B_i = \begin{bmatrix} a_i & b_i \\ -b_i & a_i  \end{bmatrix}$ with $b_i > 0$, then $a_i^2+b_i^2 < 1$. Moreover, there exists $1\leq q \leq d \leq m$ such that $B_1=\dotsm=B_q$ and $1 > a_1^2+b_1^2 = a_2^2+b_2^2 = \dotsm = a_q^2+b_q^2 > a_{q+1}^2+b_{q+1}^2 \geq \dotsm \geq a_d^2+b_d^2 > 0 $. 
\end{enumerate}
\end{assumption}

 
Let $\psi$ be the argument of $a_1+\mathrm{i} b_1$ and $\eta = \frac{\sqrt{a_{q+1}^2+b_{q+1}^2}}{\sqrt{a_q^2+b_q^2}}$. 
Under Assumption \ref{assumption-type-ii}, the power of $M$, \ie $M^k$, is convergent when $k$ goes to $+\infty$. Denote $\Minf \eqdef \lim_{k\to\pinf} M^k $.

\begin{proposition}
\label{prop:type-ii}
Consider the linear system \eqref{eq:lineq-vk} whose $M$ is a Type II matrix define in Definition \ref{def:type-ii}, suppose that Assumption \ref{assumption-type-ii} holds. Then 
\begin{enumerate}[label={\rm {(\roman{*})}}]
\item $\Minf$ is a symmetric matrix with eigenvalues being either $0$ or $1$, and $v_0 \in \ker(\Minf)$.  
\item $ \cos(\theta_k) - \cos\pa{\psi} = O(\eta^{2k}) $ for $\psi$ and $\eta$ defined above. 
\end{enumerate}
\end{proposition}
\begin{remark} $~$
\begin{itemize}
\item Proposition \ref{prop:type-ii} indicates that eventually $M$ performs circular rotation. 
\item If we have only $a_1^2+b_1^2 = a_2^2+b_2^2 = \dotsm = a_q^2+b_q^2$, and $B_i \neq B_j, 1\leq i,j \leq q, i\neq j$, then $\cos(\theta_k)$ will converge to some $\psi$ which depends on $(\psi_i)_{i=1,...,q}$ where $\psi_i$ is the argument of $a_i+\mathrm{i} b_i$. 

\end{itemize}
\end{remark}

\begin{figure}[H]
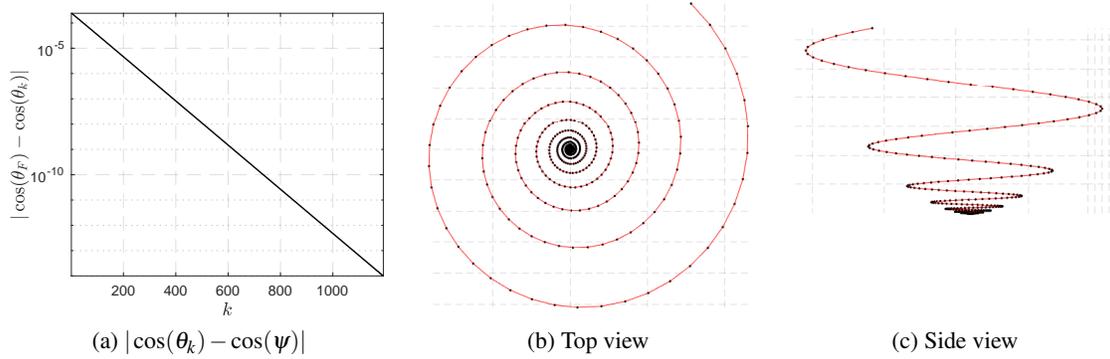

	\centering
	\subfloat[$\abs{\cos(\theta_k)-\cos(\psi)}$]{ \includegraphics[width=0.3\linewidth]{figure/angle-type-ii-A.pdf} } {\hspace{2pt}}
	\subfloat[{Top view}]{ \includegraphics[width=0.27\linewidth]{figure/tracjectory-type-ii-b-A.pdf} } {\hspace{2pt}}
	\subfloat[{Side view}]{ \includegraphics[width=0.265\linewidth]{figure/tracjectory-type-ii-a-A.pdf} }\\
	\caption{\small Convergence of $\cos(\theta_k)$ and trajectory of $\seq{\xk}$. (a): Convergence of $\cos(\theta_k)$ to $\cos(\psi)$; (b) Top view of trajectory of $\seq{\xk}$; (c) Side view of trajectory of $\seq{\xk}$. }
	\label{fig:type-ii}
\end{figure}

\begin{example}
Let $\alpha \in ]0, \pi/2]$ and define $a,b,c$ by
\[
a = 0.99\cos(\alpha) ,\enskip b = 0.99\sin(\alpha) \qandq c = \eta \sqrt{a^2+b^2} 
\]
for some $\eta \in ]0, 1[$.
Let $U$ be an orthogonal matrix in $\bbR^{3\times 3}$, and let 
$
M = U
\begin{bmatrix} a & b &  \\ -b & a &  \\  &  & c \end{bmatrix}
U^T 
$.
The matrix $M$ has three eigenvalues: $a+\mathrm{i} b , a-\mathrm{i} b$ and $c  $. 
Let $\psi$ be the argument of $a+\mathrm{i} b$ 
and $(\alpha, \eta) = (0.05, 0.96)$, 
The convergence of $\cos(\theta_k)$ and trajectories of $\seq{\xk}$ are provided in Figure \ref{fig:type-ii}. 
The first plot shows the convergence of $\cos(\theta_k)$ to $\cos(\phi)$, and the other two are different views of the trajectory of $\seq{\xk}$.


\end{example}

\begin{proof}[Proof of Proposition \ref{prop:type-ii}]
Owing to \cite[Theorem 2.5.8]{horn1990matrix} and Assumption \ref{assumption-type-ii}, we have the decomposition of $M$
\[
M = 
U 
\Sigma
U^T 
\enskip\textrm{with}\quad
\Sigma \eqdef
\begin{bmatrix}
\Id_r & & & &\\
& B_{1} & & &\\
& & \ddots & &  \\
& &  & B_{d} &\\
& &  & & 0_{r}\\
\end{bmatrix} 
\]
where $r$ denotes the multiplicities of eigenvalue $1$ in \ref{assumption-type-ii} (i), and $d$ denotes the number of $2\times 2$ blocks. For each $i=1,...,d$, we have $B_i = \begin{bmatrix} a_i & b_i \\ -b_i & a_i  \end{bmatrix}$ with $1 > a_1^2+b_1^2 \geq a_2^2+b_2^2 \geq \dotsm \geq a_d^2+b_d^2 > 0$. It is easy to show that 
$\lim_{k\to\pinf} B_i^k = 0 ,~ i=1,...,d $ 
since the spectral radius of each $B_i$, $\rho(B_i) = \sqrt{a_i^2+b_i^2} < 1$, is strictly smaller than $1$. 
This further implies that
\[
\Minf \eqdef
\lim_{k\to\pinf} M^k 
=
U 
\begin{bmatrix}
\Id_r & \\
& 0_{n-r}  
\end{bmatrix}
U^T  ,
\]
which verifies the first claim of the proposition. 

\vgap

Since $\vk \to 0$, we have from $\vk = M\vkm = M^k v_{0}$ that
\[
0 = \lim_{k\to\pinf} M^k v_0 = \Minf v_0 ,
\]
which means $v_0 \in \ker(\Minf)$ and moreover $\vk \in \ker(\Minf),~ k \in \bbN$. Consequently, we have
\[
\vk = M\vkm = (M-\Minf) \vkm
\]
Define $\tM \eqdef M-\Minf$, then there exists a real orthogonal matrix $V$ (actually a permutation of $U$) such that
\[
\tM = 
V 
\Gamma
V^T 
\enskip\textrm{with}\enskip
\Gamma \eqdef
\begin{bmatrix}
B_{1} & & &\\
& \ddots &  & \\
&  & B_{d} &\\
& & & 0_{2r}
\end{bmatrix}  
\]
and $B_i = \begin{bmatrix} a_i & b_i \\ -b_i & a_i \end{bmatrix}, i=1,...,d$. 
Suppose for some $1\leq q < d$, there holds
\[
a_1^2+b_1^2 = a_2^2+b_2^2 = \dotsm = a_q^2+b_q^2 > a_{q+1}^2+b_{q+1}^2 \geq \dotsm \geq a_d^2+b_d^2  .
\]
Consider the decomposition of $\Gamma$,
\[
\Gamma_1
= 
\begin{bmatrix}
B_{1} & & &\\
& \ddots &  & \\
&  & B_{q} &\\
& & & 0_{n-2q}
\end{bmatrix}
\qandq
\Gamma_2
= \Gamma - \Gamma_1 .
\]
Let $\sigma = \sqrt{a_1^2+b_1^2}$ and $\eta = \frac{\sqrt{a_{q+1}^2+b_{q+1}^2}}{\sigma}$, then $\frac{1}{{\sigma}^k} \Gamma_2^k = O(\eta^k) \to 0$. 
Let $\psi = \arccos(\frac{a_1}{\sigma})$, then for each $i=1,...,q$
\[
\sfrac{1}{{\sigma}} B_i = 
\begin{bmatrix} \cos(\psi) & \sin(\psi) \\ - \sin(\psi) & \cos(\psi) \end{bmatrix} 
\]
which is a circular rotation. Therefore, $\frac{1}{\sigma} \Gamma_1$ is a rotation with respect to the first $2q$ elements. 
Denote $\uk = V^T\vk$, then from $\vk = \tM \vkm$, we get $\uk = \Gamma \ukm = \Gamma^k u_0$. Consider the orthogonal decomposition of $\frac{\uk}{{\sigma}^k}$, 
\[
s_k = \sfrac{1}{{\sigma}^k} \Gamma_1^k u_0
\qandq
t_k = \sfrac{1}{{\sigma}^k} \Gamma_2^k u_0  .
\]
We have that $\norm{s_k} = \norm{s_{k-1}}$ and $\iprod{s_k}{s_{k-1}} = \norm{s_k}^2 \cos(\psi)$. 
As a result, for $\cos(\theta_k)$ we have
\beq\label{eq:cos-thetak-ii}
\cos(\theta_k)
= \sfrac{\iprod{ s_k  }{ s_{k-1} }}{\norm{ s_k + t_k }\norm{ s_{k-1} + t_{k-1} }} + \sfrac{\iprod{ t_k }{ t_{k-1} }}{\norm{ s_k + t_k }\norm{ s_{k-1} + t_{k-1} }} 
= \sfrac{ \norm{s_k}^2\cos(\psi) }{\norm{ s_k }^2 + \norm{t_k }^2} \times \sfrac{\norm{s_k + t_k}}{\norm{ s_{k-1} + t_{k-1} }} + O(\eta^{2k-1})  .
\eeq
Using the fact that 
$
\frac{ \norm{s_k}^2 \cos(\psi) }{\norm{ s_k }^2 + \norm{t_k }^2} = \cos(\psi) \Pa{1 - \norm{t_k }^2 + O(\norm{t_k }^4) } = \cos(\psi) + O(\eta^{2k}) 
$ and $
\frac{\norm{s_k+t_k}}{\norm{ s_{k-1} + t_{k-1} }} \to 1 
$
we conclude the convergence of $\theta_{k}$.  \qedhere

\end{proof}


\subsection{Type III linear system}\label{subsec:type-iii}

The last trajectory we discuss is the elliptical spiral which is more complicated. 
We first discuss the definition and properties of elliptical rotation, then discuss one type of matrix that leads to elliptical rotation.

\subsubsection{Elliptical rotation}

\begin{definition}[{\cite[Theorem~1]{ozdemir2016alternative}}]\label{lem:elliptical-rotation}
Let $\phi > 0$ and $l > s > 0$, then the following matrix
\[
\scrR_{l,s,\phi} = 
\begin{bmatrix} \cos(\phi) & \sfrac{{s}}{{l}} \sin(\phi) \\ - \sfrac{{l}}{{s}} \sin(\phi) & \cos(\phi) \end{bmatrix}
\]
is an elliptical rotation along the ellipse $\frac{x^2}{s^2} + \frac{y^2}{l^2} =  d$ 
with $d > 0$. 
\end{definition}

\begin{remark}\label{rmk:ellipses}
The definition is adopted from \cite{ozdemir2016alternative}.  
All similar ellipses have identical elliptical rotation matrices \cite{ozdemir2016alternative}. 
When $s=l$, then $\scrR_{l,s,\phi}$ simply becomes circular rotation. 
\end{remark}

Different from circular rotation which is isometry, elliptical rotation does not preserves angle and distance. Given any $x\in\bbR^n$ and its rotated point $x_{+} = \scrR_{l,s,\phi} x$, the angle between $x_{+}, x$ and the ratio $\frac{\norm{x_{+}}}{\norm{x}}$ depend on $x$. 
In the following, let $e = (\frac{s^2}{l^2}-1) / (\frac{s^2}{l^2}+1)$ and $\zeta = \arccos\pa{ - e\cos(\phi) } $.

\begin{proposition}\label{prop:elliptical-rotation}
Let $\scrR_{l,s,\phi}$ be an elliptical rotation for some $\phi \in ]0, \pi[$ and $l , s > 0$. Given an arbitrary point $x \neq 0$ and its rotated point $ x_{+} = \scrR_{l,s,\phi} x $, there holds
\begin{itemize}
\item The ratio $\frac{\norm{x_{+}}^2}{\norm{x}^2}
\in \big[ \frac{ e \cos(\zeta-\phi) + 1 }{ e \cos(\zeta+\phi) + 1 } , \frac{ e \cos(\zeta+\phi) + 1 }{ e \cos(\zeta-\phi) + 1  } \big]$ .  

\item Let $\chi$ be the angle between $x$ and $x_{+}$, we have $\chi \in [\chim, \chiM]$ 
with
$
\cos(\chiM)
= \frac{ a \cos(\phi) - b }{ \ssqrt{ \sin^2(\phi)  + \pa{a \cos(\phi) - b }^2 } }
$ and $
\cos(\chim)
= \frac{ a \cos(\phi) + b }{ \ssqrt{ \sin^2(\phi)  + \pa{a \cos(\phi) + b }^2 } }  
$ where $a = \frac{s}{2l} + \frac{l}{2s} , b = \abs{ \frac{s}{2l} - \frac{l}{2s} }$.

\end{itemize}
\end{proposition}
\begin{remark} 
When $s/l=1$, then $\scrR_{l,s,\phi}$ becomes circular rotation, consequently we get $\frac{\norm{x_{+}}^2}{\norm{x}^2} = 1$ and $\chi = \phi$. 

\end{remark}
\begin{example}\label{exp:elliptical_rotation}
In this example, we consider an elliptical rotation parameterized by $l = 2, s = 1$ and $\phi = \frac{\pi}{30.01}$. Consider the sequence $\seq{\xk}$ generated by the rotation $\xk = \scrR_{l,s,\phi} \xkm$ with $x_0$ chosen arbitrarily, we study the ratio $\frac{\norm{\xk}^2}{\norm{\xkm}^2}$ and the angle $\chi_k = \angle(\xk,\xkm)$ 
\begin{itemize}
   \item We have $e = (\frac{s^2}{l^2}-1) / (\frac{s^2}{l^2}+1) = \frac35$ and $\zeta = \arccos\pa{ - \frac{3\cos(\phi)}{5} } $, consequently $\frac{\norm{\xk}^2}{\norm{\xkm}^2} \in [0.5679 , 1.7608]$. 
   \item For the angle $\chi_k$, we have that $\chi_k \in [0.1980 ,  0.7564]$. 
\end{itemize}
The values of $\frac{\norm{\xk}^2}{\norm{\xkm}^2}$ and $\chi_k$ along $k$ are shown below in Figure \ref{fig:elliptical_rotation}. 

\begin{figure}[H]
\centering
\subfloat[${\norm{\xk}^2}/{\norm{\xkm}^2}$]{ \includegraphics[width=0.425\linewidth]{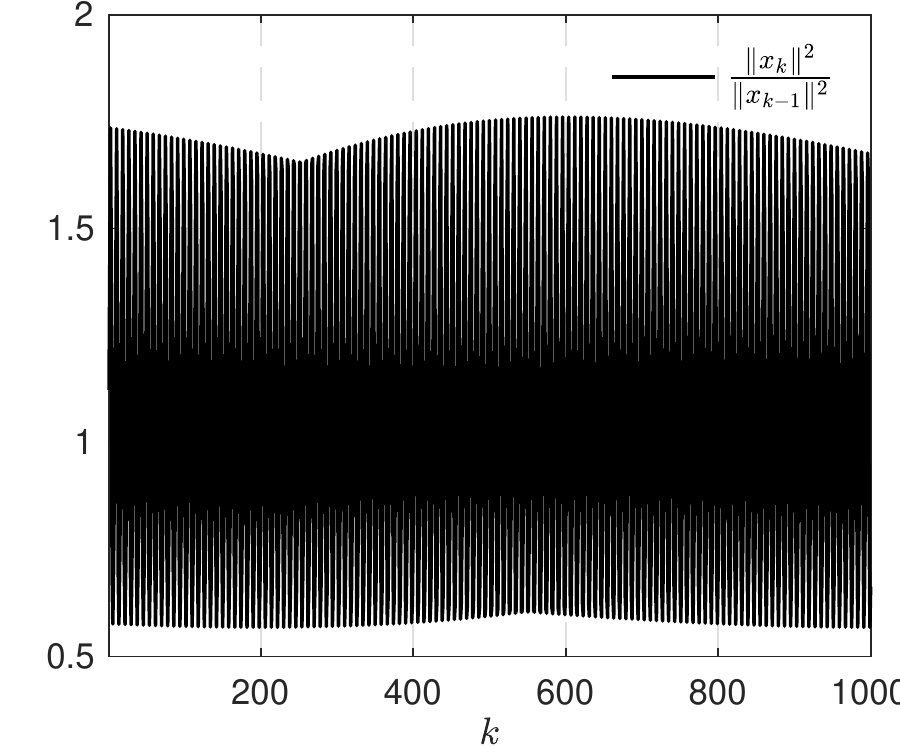} } {\hspace{2pt}}
\subfloat[${\cos(\chi_k)}$]{ \includegraphics[width=0.425\linewidth]{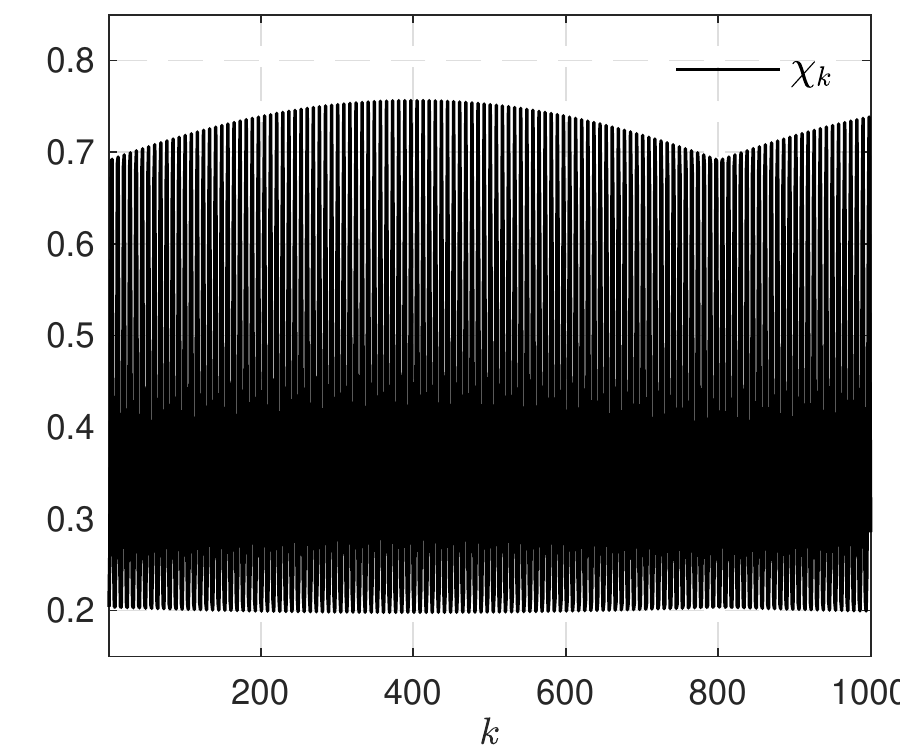} } \\ 
\caption{\small Range of $\frac{\norm{\xk}^2}{\norm{\xkm}^2}$ and $\cos(\chi_k)$ for elliptical rotation.}
\label{fig:elliptical_rotation}
\end{figure}

\end{example}
\begin{proof}[Proof of Proposition \ref{prop:elliptical-rotation}]
Since $x \neq 0$, there exit $\beta \geq 0$ and $L , S > 0$ such that $S/L = s/l$ and $x =\begin{pmatrix} S \cos(\beta) \\ L \sin(\beta) \end{pmatrix}$. 
Then 
\beq\label{eq:x_prime}
\begin{aligned}
x_{+} = \scrR_{l,s,\phi} x
= 
\begin{bmatrix} \cos(\phi) & \tfrac{s}{l} \sin( \phi )  \\ - \tfrac{l}{s} \sin( \phi ) & \cos(\phi) \end{bmatrix}
\begin{pmatrix} S \cos(\beta) \\ L \sin(\beta) \end{pmatrix}
=
\begin{pmatrix} S \cos(\phi) \cos(\beta) + S \sin( \phi ) \sin(\beta) \\ L \cos(\phi) \sin(\beta) - L \sin( \phi ) \cos(\beta)  \end{pmatrix}
=
\begin{pmatrix} S \cos(\phi - \beta)  \\ - L \sin(\phi-\beta)   \end{pmatrix}.
\end{aligned}
\eeq
We first prove the range of $\frac{\norm{x_{+}}}{\norm{x}}$,
\[
\begin{aligned}
q(\beta) = \sfrac{\norm{x_{+}}^2}{\norm{x}^2}
= \sfrac{ S^2 \cos^2(\beta) + L^2 \sin^2(\beta) }{ S^2 \cos^2(\phi - \beta) + L^2 \sin^2(\phi-\beta) } 
= \sfrac{ (\frac{s^2}{l^2}-1) \cos^2(\beta) + 1 }{ (\frac{s^2}{l^2}-1) \cos^2(\phi - \beta) + 1 } 
&= \sfrac{ (\frac{s^2}{l^2}-1) \cos(2\beta) + \frac{s^2}{l^2}+1 }{ (\frac{s^2}{l^2}-1) \cos(2\phi-2\beta) + \frac{s^2}{l^2}+1 }.
\end{aligned}
\]
Denote $e = (\frac{s^2}{l^2}-1) / (\frac{s^2}{l^2}+1)$, then we get from above that $q(\beta) = \frac{ e \cos(2\beta) + 1 }{ e \cos(2\beta-2\phi) + 1 }$
whose derivative with respective to $\beta$ reads
\[
\begin{aligned}
q'(\beta) 
&= \sfrac{ -2e^2\sin(2\phi) - 2e( \sin(2\beta) - \sin(2\beta-2\phi) ) }{ (e \cos(2\beta-2\phi) + 1)^2 }.
\end{aligned}
\]
Solving $ q'(\beta) = 0$ we get
\[
\begin{aligned}
0 = e\sin(2\phi) + \sin(2\beta) - \sin(2\beta-2\phi)
\quad\Longleftrightarrow\quad &
0 = 2\sin(\phi) \Pa{ e\cos(\phi) + \cos(2\beta - \phi) } \\
\quad\Longleftrightarrow\quad &
- e\cos(\phi) = \cos(2\beta - \phi) .
\end{aligned}
\]
Denote $\zeta = \arccos\pa{- e\cos(\phi) } $, 
then the choices of $\beta$ such that $q'(\beta)=0$ hence $q(\beta)$ reaches extreme values are
\[
\beta_{\max} = \sfrac{\zeta+\phi}{2}
\qandq
\beta_{\min} = \pi - \sfrac{\zeta-\phi}{2}  .
\]
Consequently we get
\[
\max_{\beta \in [0, 2\pi]} q(\beta) = q(\beta_{\max}) = \sfrac{ e \cos(\zeta+\phi) + 1}{ e \cos(\zeta-\phi) + 1}
\qandq
\min_{\beta \in [0, 2\pi]} q(\beta) = q(\beta_{\min}) = \sfrac{ e \cos(\zeta-\phi) + 1}{ e \cos(\zeta+\phi) + 1}  ,
\]
which is the range of $\frac{\norm{x_{+}}^2}{\norm{x}^2}$.

For the angle $\chi$ between $x_{+}$ and $x$, from \eqref{eq:x_prime} we get the following inner product
\beq\label{eq:iprod_xp_x}
\begin{aligned}
\iprod{x_{+}}{x}
= S^2 \cos(\beta) \cos(\phi-\beta) - L^2 \sin(\phi - \beta) \sin(\beta)  
&= \sfrac{S^2+L^2}{2} \cos(\phi) + \sfrac{S^2-L^2}{2} \cos(\phi-2\beta)  , 
\end{aligned}
\eeq
which means
\[
\begin{aligned}
\cos(\chi) 
= \sfrac{ \iprod{x_{+}}{x} }{ \norm{x_{+}} \norm{x} }
&= \sfrac{ \frac{s/l+l/s}{2} \cos(\phi) + \frac{s/l-l/s}{2} \cos(\phi-2\beta) }{ \ssqrt{ \sin^2(\phi-\beta) + s^2/l^2 \cos^2(\phi - \beta) } \ssqrt{ l^2/s^2 \sin^2(\beta) + \cos^2(\beta) } }  .
\end{aligned}
\]
Since $\phi$ and $s/l$ are constant, consider the function of $\beta$: 
\[
\begin{aligned}
\ell(\beta) 
&\eqdef \Pa{ \sin^2(\phi-\beta) + \tfrac{s^2}{l^2} \cos^2(\phi - \beta) } \Pa{ \tfrac{l^2}{s^2} \sin^2(\beta) + \cos^2(\beta) } \\
&= { \tfrac{l^2}{s^2} \sin^2(\beta) \sin^2(\phi-\beta) + \cos^2(\beta) \sin^2(\phi-\beta) + \sin^2(\beta) \cos^2(\phi - \beta) + \tfrac{s^2}{l^2} \cos^2(\beta) \cos^2(\phi - \beta)  } \\
&= \Pa{ \sin(\beta) \cos(\phi - \beta) + \cos(\beta) \sin(\phi-\beta) }^2 - 2 \sin(\beta) \cos(\phi - \beta) \cos(\beta) \sin(\phi-\beta) \\&\qquad + { \tfrac{l^2}{s^2} \sin^2(\beta) \sin^2(\phi-\beta) + \tfrac{s^2}{l^2} \cos^2(\beta) \cos^2(\phi - \beta)  } \\
&= \Pa{ \sin(\beta) \cos(\phi - \beta) + \cos(\beta) \sin(\phi-\beta) }^2  + \Pa{ \tfrac{s}{l} \cos(\beta) \cos(\phi - \beta) - \tfrac{l}{s} \sin(\beta) \sin(\phi-\beta) }^2 \\
&= \sin^2(\phi)  + \Pa{ \tfrac{s}{l} \cos(\beta) \cos(\phi - \beta) - \tfrac{s}{l} \sin(\beta) \sin(\phi-\beta) - (\tfrac{l}{s}-\tfrac{s}{l}) \sin(\beta) \sin(\phi-\beta) }^2 \\
&= \sin^2(\phi)  + \Pa{ \tfrac{s}{l} \cos(\phi) - (\tfrac{l}{s}-\tfrac{s}{l}) \sin(\beta) \sin(\phi-\beta) }^2 \\
&= \sin^2(\phi)  + \bPa{ \tfrac{s}{l} \cos(\phi) - (\tfrac{l}{s}-\tfrac{s}{l}) \sfrac{\cos(\phi-2\beta) - \cos(\phi)}{2} }^2 \\
&= \sin^2(\phi)  + \bPa{ \tfrac{s}{l} \cos(\phi) + (\tfrac{l}{s}-\tfrac{s}{l}) \sfrac{ \cos(\phi)}{2} - (\tfrac{l}{s}-\tfrac{s}{l}) \sfrac{\cos(\phi-2\beta)}{2} }^2 \\
&= \sin^2(\phi)  + \Pa{ \pa{ \tfrac{l}{2s} + \tfrac{s}{2l} } {\cos(\phi)} - \pa{ \tfrac{l}{2s} - \tfrac{s}{2l} } {\cos(\phi-2\beta)} }^2   .
\end{aligned}
\]
Therefore, we have
\beqn
\cos(\chi) 
= \sfrac{ (\frac{s}{2l} + \frac{l}{2s} ) \cos(\phi) + (\frac{s}{2l} - \frac{l}{2s} ) \cos(\phi-2\beta) }{ \ssqrt{\sin^2(\phi)  + \pa{ \pa{ \tfrac{l}{2s} + \tfrac{s}{2l} } {\cos(\phi)} + (\frac{s}{2l} - \frac{l}{2s} ) {\cos(\phi-2\beta)} }^2 } }  .
\eeqn
Consider the following function 
\[
f(x) = \sfrac{x}{\ssqrt{\sin^2(\phi) + x^2}} .
\]
It is easy to verify that the derivative $f'(x) = \frac{\sin^2(\phi)}{(\sin^2(\phi)+x^2)^{3/2}} > 0$ holds for all $x \in \bbR$, hence $f(x)$ is monotonically increasing. 
Now let $x = (\frac{s}{2l} + \frac{l}{2s} ) \cos(\phi) + (\frac{s}{2l} - \frac{l}{2s} ) \cos(\phi-2\beta)$, we have 
\[
x \in \left[ \Pa{ \tfrac{s}{2l} + \tfrac{l}{2s} } \cos(\phi) - \abs{ \tfrac{s}{2l} - \tfrac{l}{2s} } , \Pa{ \tfrac{s}{2l} + \tfrac{l}{2s} } \cos(\phi) + \abs{ \tfrac{s}{2l} - \tfrac{l}{2s} } \right]  ,
\] 
which further implies
\[
\sfrac{ \pa{ \frac{s}{2l} + \frac{l}{2s} } \cos(\phi) - \abs{ \frac{s}{2l} - \frac{l}{2s} } }{ \ssqrt{ \sin^2(\phi)  + \pa{\pa{ \frac{s}{2l} + \frac{l}{2s} } \cos(\phi) - \abs{ \frac{s}{2l} - \frac{l}{2s} } }^2 } }
\leq
\cos(\chi) 
\leq
\sfrac{ \pa{ \frac{s}{2l} + \frac{l}{2s} } \cos(\phi) + \abs{ \frac{s}{2l} - \frac{l}{2s} } }{ \ssqrt{ \sin^2(\phi)  + \pa{\pa{ \frac{s}{2l} + \frac{l}{2s} } \cos(\phi) + \abs{ \frac{s}{2l} - \frac{l}{2s} } }^2 } }  .
\]
Since $\cos(\chi)$ is monotonic decreasing in $[0, \pi/2]$, we obtain the claimed result. \qedhere

\end{proof}

Given an angle $\psi \in ]0, \pi]$, define the circular rotation $\scrR_{\psi} = \begin{bmatrix} \cos(\psi) & -\sin(\psi) \\ \sin(\psi) & \cos(\psi) \end{bmatrix}$. 
For the composite rotation $\scrR_{\psi} \scrR_{l,s,\phi}$, we have the following.

\begin{proposition}\label{prop:comp-circular-elliptical}
Let $\scrR_{l,s,\phi}$ be an elliptical rotation for $\phi \in ]0, \pi/2]$ and $l , s > 0$ and $\scrR_{\psi}$ be a circular rotation such that $$\Pa{ \pa{ \tfrac{l}{s} + \tfrac{s}{l} } \sin(\psi)\sin(\phi) + 2\cos(\psi)\cos(\phi) }^2 - 4 < 0 .$$ Given an arbitrary point $x \neq 0$ and its rotated point $ x_{+} = \scrR_{\psi} \scrR_{l,s,\phi} x $, 
\begin{itemize}
\item The ratio $\frac{\norm{x_{+}}^2}{\norm{x}^2}
\in \big[ \frac{ e \cos(\zeta-\phi) + 1 }{ e \cos(\zeta+\phi) + 1 } , \frac{ e \cos(\zeta+\phi) + 1 }{ e \cos(\zeta-\phi) + 1  } \big]$ as in Proposition \ref{prop:elliptical-rotation}.

\item Let $\chi_c = \angle(x,x_{+})$ be the angle between $x$ and $x_{+}$, then $\chi_c \in [\psi - \chiM, \psi - \chim]$ 
where $\chim, \chiM$ are as defined in Proposition~\ref{prop:elliptical-rotation}. 

\end{itemize}

\end{proposition}
 

\begin{excont}[Continued]
We continue Example \ref{exp:elliptical_rotation} by compositing $\scrR_{l,s,\phi}$ with a circular rotation. 
Let $\psi = \frac{\pi}{3}$ and consider the sequence $\seq{\yk}$ generated by the rotation $\yk = \scrR_{\psi} \scrR_{l,s,\phi} \ykm$ with $y_0 = x_0$, we consider the trajectory of the sequence $\seq{\yk}$ and angle $\chi_{c,k} = \angle(\yk,\ykm), \vartheta_{c,k} = \angle(\yk,\yk-\ykm)$: 
\begin{itemize}
   \item For the elliptical rotation and the composite rotation, trajectories of the sequences $\seq{\xk}, \seq{\yk}$ are shown below in Figure \ref{fig:elliptical_rotation_continue} (a). Note that the trajectory of $\seq{\yk}$ is not equal to rotating the that of $\seq{\xk}$ using $\scrR_{\psi}$. 
   \item For the angle $\chi_{c,k}$, we have that $\chi_{c,k} \in [\psi - \chiM, \psi - \chim] = [0.2908 ,  0.8492]$. 
\end{itemize}

\begin{figure}[H]
\centering
\subfloat[Trajectory of $\xk$]{ \includegraphics[width=0.4\linewidth]{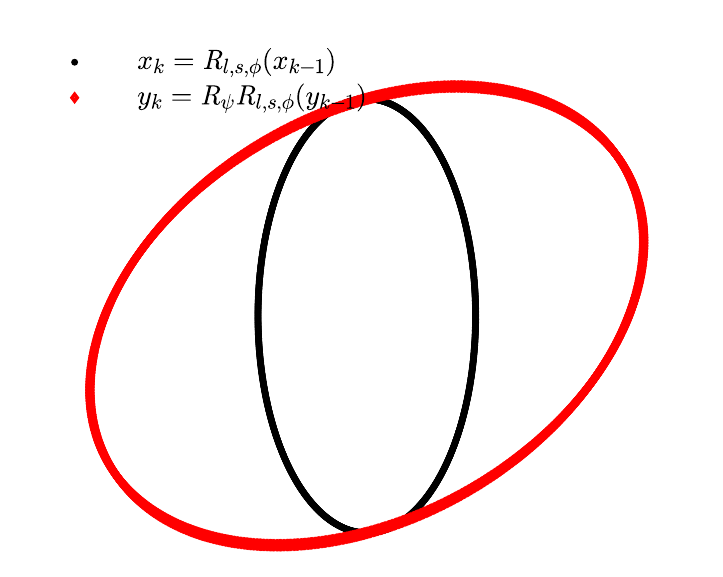} } {\hspace{2pt}}
\subfloat[{${\cos(\chi_{c,k})}$}]{ \includegraphics[width=0.4\linewidth]{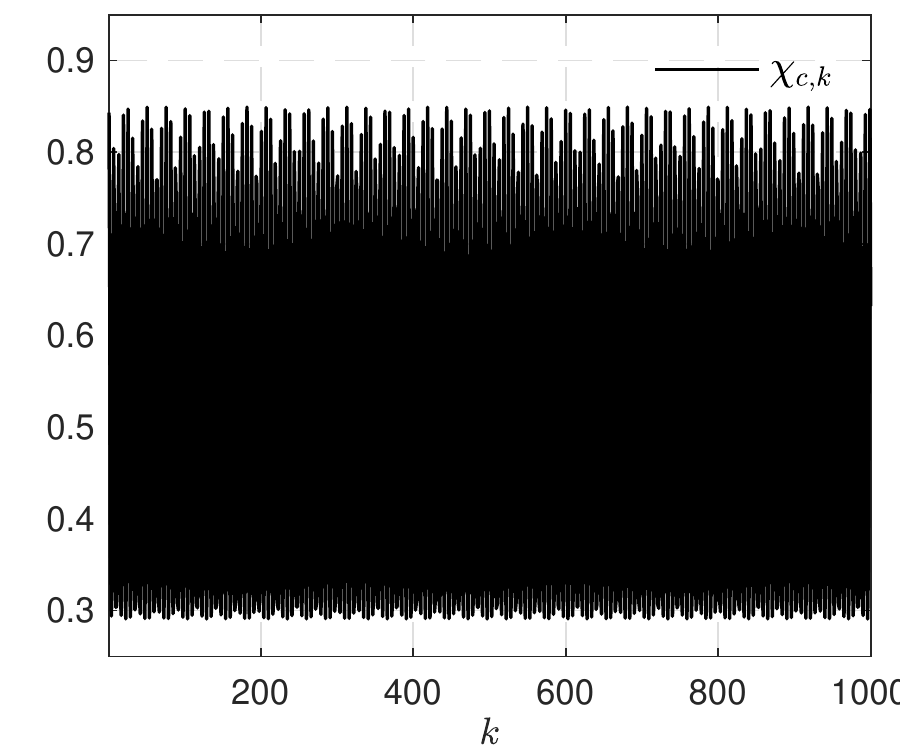} } \\ 
\caption{\small Trajectories of sequences $\seq{\xk}, \seq{\yk}$, and the range of $\chi_{c,k}$.}
\label{fig:elliptical_rotation_continue}
\end{figure}

\end{excont}
\begin{proof}
Denote $\scrR = \scrR_{\psi}\scrR_{l,s,\phi} $, then we have 
\[
\begin{aligned}
\scrR 
&= 
\begin{bmatrix} 
\cos(\psi)\cos(\phi) + \sfrac{{l}}{{s}} \sin(\psi)\sin(\phi)  & \sfrac{{s}}{{l}} \cos(\psi)\sin(\phi) - \sin(\psi)\cos(\phi) \\ 
\sin(\psi)\cos(\phi) - \sfrac{l}{s} \cos(\psi)\sin(\phi) & \sfrac{s}{l}\sin(\psi)\sin(\phi) + \cos(\psi)\cos(\phi) 
\end{bmatrix} 
= \begin{bmatrix} \scrR_{1,1} & \scrR_{1,2} \\ \scrR_{2,1} & \scrR_{2,2} \end{bmatrix} 
\end{aligned}
\]
and that
\[
\begin{aligned}
\scrR_{1,1}-\scrR_{2,2}
&= \Pa{ \sfrac{{l}}{{s}} - \sfrac{s}{l} } \sin(\psi)\sin(\phi)   ,\\
\scrR_{1,2}\scrR_{2,1}
&= \Pa{ \sfrac{{s}}{{l}} + \sfrac{l}{s} } \cos(\psi)\sin(\phi) \sin(\psi)\cos(\phi) - \cos^2(\psi)\sin^2(\phi) - \sin^2(\psi)\cos^2(\phi) .
\end{aligned}
\]
The characteristic polynomial of $\scrR$ reads
\[
x^2 - (\scrR_{1,1}+\scrR_{2,2}) x + \scrR_{1,1}\scrR_{2,2} - \scrR_{1,2}\scrR_{2,1} = 0 ,
\]
whose discriminant is
\[
\begin{aligned}
\Delta 
&=(\scrR_{1,1}+\scrR_{2,2})^2 - 4(\scrR_{1,1}\scrR_{2,2} - \scrR_{1,2}\scrR_{2,1})\\
&= (\scrR_{1,1}-\scrR_{2,2})^2 + 4 \scrR_{1,2}\scrR_{2,1}\\
&= \Pa{ \sfrac{{l}}{{s}} - \sfrac{s}{l} }^2 \sin^2(\psi)\sin^2(\phi)  + 4\Pa{ \sfrac{{s}}{{l}} + \sfrac{l}{s} } \cos(\psi)\sin(\phi) \sin(\psi)\cos(\phi) - 4\cos^2(\psi)\sin^2(\phi) - 4 \sin^2(\psi)\cos^2(\phi)  \\
&= \bPa{ \Pa{ \sfrac{{l}}{{s}} + \sfrac{s}{l} } \sin(\psi)\sin(\phi) + 2\cos(\psi)\cos(\phi) }^2 - 4.
\end{aligned}
\]
When $\Delta < 0$, $\scrR$ admits two complex eigenvalues, meaning that $\scrR$ is a rotation.  

Let $y = \scrR_{l,s,\phi} x$ and let $\beta$ be the angle between $x, y$, then owing to Proposition \ref{prop:elliptical-rotation}, we have $\beta \in [\chim, \chiM]$. The angle between $x_{+}$ and $y$ is $\psi$ which is straightforward owing to $\scrR_{\psi}$. Since $\scrR_{l,s,\phi}$ is clockwise rotation and $\scrR_{\psi}$ is counterclockwise, the claimed result follows immediately.  
%
\qedhere
\end{proof}



\subsubsection{Type III linear system}

Let $m_1, m_2 \in \bbN_+$ such that $m_1 + m_2 = n$, 
let $A \in \bbR^{m_1 \times m_1}, B \in \bbR^{m_2 \times m_2}$ be symmetric and $C \in \bbR^{m_2 \times m_1}$. 
Define the following $2\times2$ block matrix
\beq\label{eq:mtx-type-iii-a}
M \eqdef \begin{bmatrix} A & - \delta C^T \\ \tau C & B \end{bmatrix}  .
\eeq

\begin{definition}[Type III matrix]\label{def:type-iii-a}
$M \in \bbR^{n\times n}$ is a $2 \times 2$ block matrix defined by \eqref{eq:mtx-type-iii-a}, with all its eigenvalues lying in the complex unit disc. 
\end{definition}

For the sake of brevity, we assume henceforth that $m_1 = m_2 = m = \frac{n}{2}$\footnote{For the case of $m_1\neq m_2$ or $n$ is odd, we can apply the zero padding trick.}. 
Denote $S_A = (a_i)_{i=1,...,m}, S_B = (b_i)_{i=1,...,m}$ and $S_C = (c_i)_{i=1,...,m} \in \bbR^m$ the singular values of $A, B$ and $C$ in \emph{descending} order, respectively. 
For each $i=1,...,m$, define the $2\times 2$ matrix $D_i$ by
$D_i = 
\begin{bmatrix}
a_i & - \delta c_i \\ \tau c_i & b_i 
\end{bmatrix} $.
It is trivial to show that the eigenvalues of $D_i$ read $\frac{1}{2}\pa{ (a_i+b_i) \pm \ssqrt{(a_i-b_i)^2-4\delta\tau c_i^2} }$. 
We impose the following assumptions.

\begin{assumption}\label{assumption-type-iii-a}
Let $C = Y \diag(S_C) X^T$ be the SVD of $C$, suppose that $A$ and $B$ can be diagonalized by $X$ and $Y$ respectively, that is there holds $\diag(S_A) = X^T A X$ and $\diag(S_B) = Y^T B Y$. 
For each $i=1,...,m$:
\begin{enumerate}[label={\rm {(\roman{*})}}]
\item If the eigenvalues are real, \ie $(a_i-b_i)^2-4\delta\tau c_i^2\geq0$, then they are either $0$ or $1$;
\item If $(a_i-b_i)^2-4\delta\tau c_i^2 < 0$, then $\delta\tau c_i^2 + a_ib_i < 1$. Moreover, there exists $1\leq q\leq d \leq m$ such that $D_i=D_j, 1\leq i,j \leq q$ and  $\delta\tau c_1^2 + a_1b_1 = \dotsm = \delta\tau c_q^2 + a_qb_q  > \delta\tau c_{q+1}^2 + a_{q+1}b_{q+1}  \geq \dotsm \geq \delta\tau c_d^2 + a_db_d > 0 $.  
\end{enumerate}
\end{assumption}

Let $\sigma = \ssqrt{\delta\tau c_1^2 + a_1b_1}$ and $\eta = \frac{\sqrt{\delta\tau c_{q+1}^2 + a_{q+1}b_{q+1}}}{\sigma}$, we have the following proposition. Again, let $\Minf \eqdef \lim_{k\to\pinf} M^k $.

\begin{proposition}
\label{prop:type-iii-a}
Consider the linear system \eqref{eq:lineq-vk} whose $M$ is a Type III matrix define in Definition \ref{def:type-iii-a}, suppose that Assumption \ref{assumption-type-iii-a} holds. Then  
\begin{enumerate}[label={\rm {(\roman{*})}}]
\item $\Minf$ is a symmetric matrix with eigenvalues being either $0$ or $1$, and $v_0 \in \ker(M)$. 

\item $\theta_k \in \big[ \psi - \alphaM , \psi - \alpham \big]  $, 
where 
$
\cos(\chiM)
= \frac{ \pa{ \frac{s}{2l} + \frac{l}{2s} } \cos(\phi) - \abs{ \frac{s}{2l} - \frac{l}{2s} } }{ \ssqrt{ \sin^2(\phi)  + \pa{\pa{ \frac{s}{2l} + \frac{l}{2s} } \cos(\phi) - \abs{ \frac{s}{2l} - \frac{l}{2s} } }^2 } }
$ and $
\cos(\chim)
= \frac{ \pa{ \frac{s}{2l} + \frac{l}{2s} } \cos(\phi) + \abs{ \frac{s}{2l} - \frac{l}{2s} } }{ \ssqrt{ \sin^2(\phi)  + \pa{\pa{ \frac{s}{2l} + \frac{l}{2s} } \cos(\phi) + \abs{ \frac{s}{2l} - \frac{l}{2s} } }^2 } }  
$, and
$\psi = \mathrm{arccot}\Pa{ \frac{ (\tau-\delta)c_1 }{ b_1 - a_1 } } $,  
$\phi = \arccos\Pa{ \frac{(\delta+\tau) c_1 \sin(\psi) + (a_1+b_1)\cos(\psi) }{2 {\sigma}} } $
and 
$\sfrac{s}{l} = \frac{b_1}{{\sigma} \sin(\psi)\sin(\phi) } - \cot(\psi)\cot(\phi).
$


\end{enumerate}
\end{proposition}
\begin{remark} 
When $\delta = \tau = 1$, then we have $\psi = \frac{\pi}{2} ,
\phi = \arccos\Pa{ \frac{c_1}{ \sqrt{ c_1^2 + a_1b_1 } } } 
$ and $ 
\frac{s}{l} = \frac{b_1}{ \sin(\phi) \sqrt{ c_1^2 + a_1b_1 }} $. 
%
\end{remark}

\begin{example}
Let $\delta = \tau  = 1$ and $a,b,c,d > 0$, and let 
$
M = U
\begin{bmatrix} a & -c &  \\ c & b &  \\  &  & d \end{bmatrix}
U^T 
$ where $U$ is an orthogonal matrix in $\bbR^{3\times 3}$. 
$M$ has three eigenvalues: $\frac{(a+b) + \sqrt{(a-b)^2-4c^2}}{2} , \frac{(a+b) - \sqrt{(a-b)^2-4c^2}}{2}$ and $d$.
Note that the magnitude of both complex eigenvalues is $\sqrt{ab+c^2}$. We consider the following choice of $a,b,c,d$: $(a,b,c) = (0.95 , 0.75, 0.35)$ and $d = 0.99\sqrt{ab+c^2}$. The observations are shown in Figure \ref{fig:type-iii-a}, where the first figure shows the oscillation behavior of $\cos{\theta_k}$, and the trajectories of $\seq{\xk}$ from different perspectives are provided in the 2nd and 3rd figures.

\begin{figure}[H]
\centering
\subfloat[${\cos(\theta_k)}$]{ \includegraphics[width=0.3\linewidth]{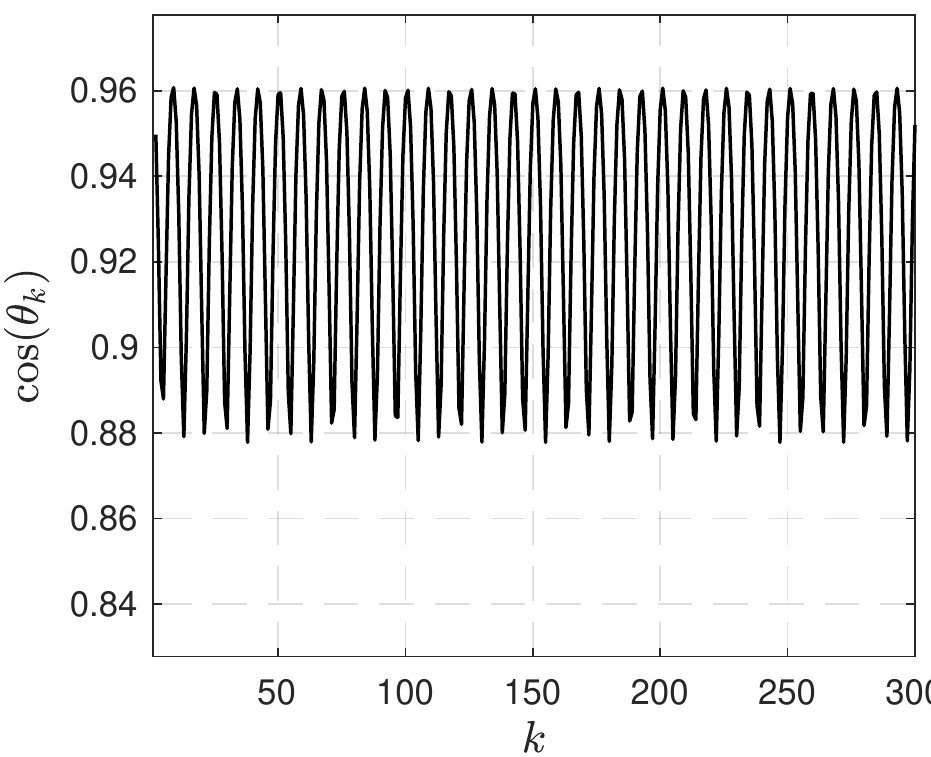} } {\hspace{2pt}}
\subfloat[{Top view}]{ \includegraphics[width=0.27\linewidth]{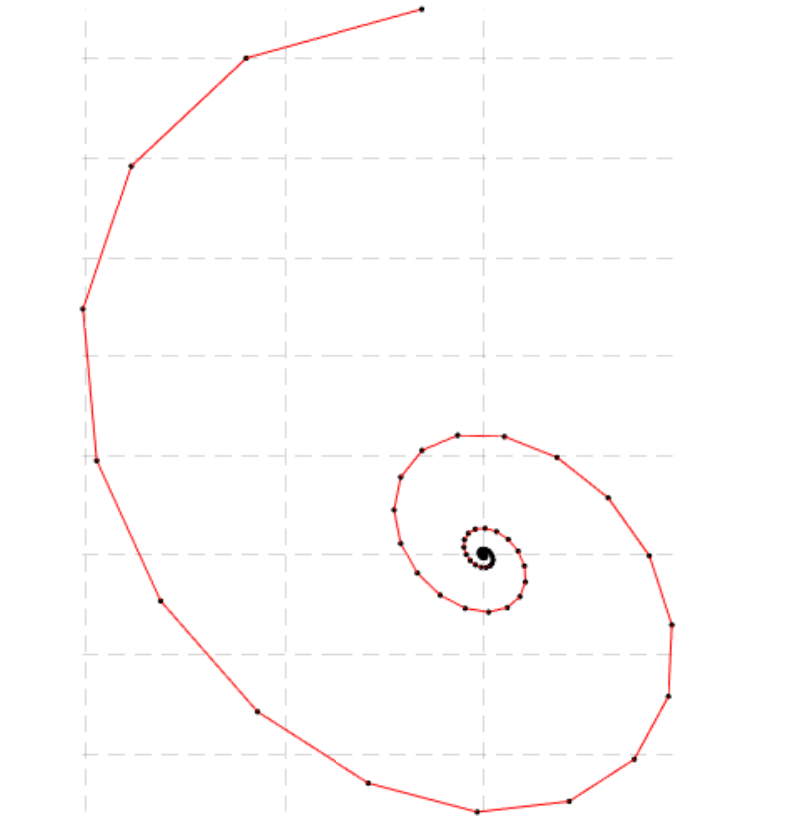} } {\hspace{2pt}}
\subfloat[{Side view}]{ \includegraphics[width=0.255\linewidth]{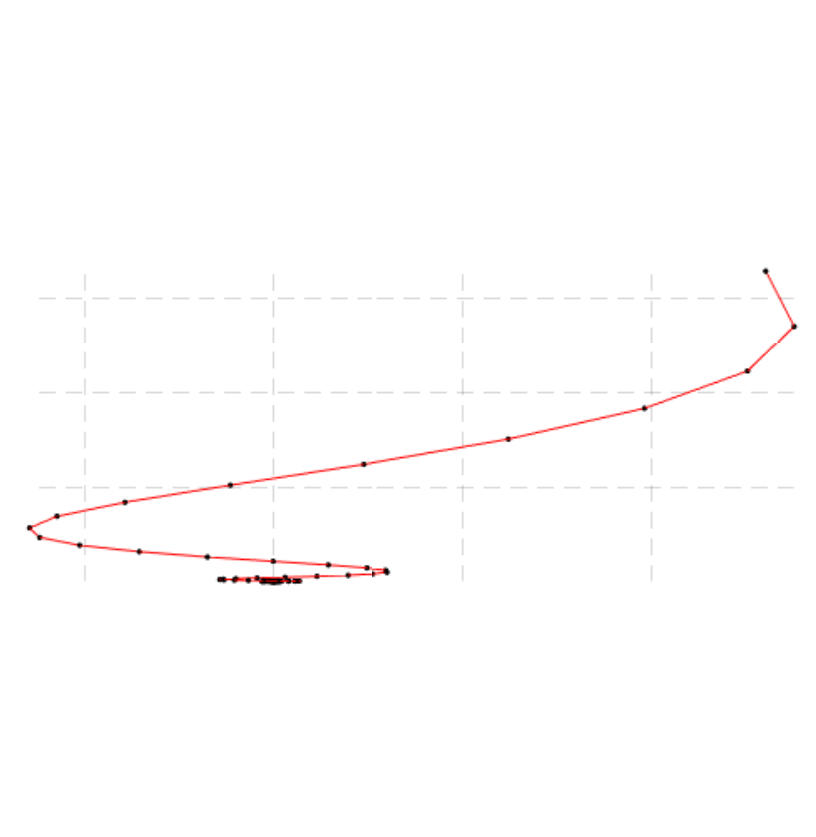} }\\
\caption{\small Oscillation of $\cos(\theta_k)$ and trajectory of $\seq{\xk}$. (a): Oscillation of $\cos(\theta_k)$; (b) Top view of trajectory of $\seq{\xk}$; (c) Side view of trajectory of $\seq{\xk}$. }
\label{fig:type-iii-a}
\end{figure}

\end{example}

\begin{proof}[Proof of Proposition \ref{prop:type-iii-a}] 
Owing to Assumption \ref{assumption-type-iii-a}, denote $\Sigma_A = \diag(S_A), \Sigma_B = \diag(S_B)$ and $\Sigma_C = \diag(S_C)$, then we have for $M$ that
\beq\label{eq:type-iii-a-M}
M 
= \begin{bmatrix} A & - C^T \\ C & B \end{bmatrix}
= \begin{bmatrix} X \Sigma_A X^T  & - \delta X \Sigma_{{C}} Y^T \\ \tau Y \Sigma_{{C}} X^T  & Y \Sigma_B Y^T \end{bmatrix}  
= \begin{bmatrix} X &  \\  & Y  \end{bmatrix}
\begin{bmatrix} \Sigma_A  & - \delta \Sigma_{{C}} \\ \tau \Sigma_{{C}}  & \Sigma_B \end{bmatrix}
\begin{bmatrix} X^T  &  \\  & Y^T \end{bmatrix}  .
\eeq
Define the following matrix
\[
\Sigma
\eqdef
\begin{bmatrix} \Sigma_A  & - \delta \Sigma_{{C}} \\ \tau \Sigma_{{C}}  & \Sigma_B \end{bmatrix}  
\]
which is block diagonal matrix. 
For each $i=1,...,p$, define the $2\times 2$ matrix $D_i = \begin{bmatrix} a_i & - \delta c_i \\ \tau c_i & b_i \end{bmatrix}$. 
Owing to Assumption \ref{assumption-type-iii-a}, when $(a_i-b_i)^2-4\delta\tau c_i^2 < 0$, the eigenvalues of $D_i$, \ie $\frac{ (a_i+b_i) \pm \sqrt{(a_i-b_i)^2-4\delta\tau c_i^2} } {2}$, are complex and their magnitudes is $\delta\tau c_i^2+a_ib_i < 1$. 
Therefore, we have $\lim_{k\to\pinf} D_i^k = 0$ owing to spectral theorem. 
This further implies 
\[
\Minf \eqdef
\lim_{k\to\pinf} M^k 
=
\begin{bmatrix} Y &  \\  & X  \end{bmatrix}
\begin{bmatrix}
\Id_r & \\
& 0_{n-r}  
\end{bmatrix}
\begin{bmatrix} Y^T &  \\  & X^T  \end{bmatrix}  ,
\]
where $r$ is the multiplicities of eigenvalue $1$.

Following the arguments of the proof of Proposition \ref{prop:type-ii}, we have that $\vk \in \ker(\Minf)$ for all $k\in\bbN$. 
Again, define $\tM \eqdef M-\Minf$. 
Based on \eqref{eq:type-iii-a-M} and block-diagonal nature of $\Sigma$, there exists an elementary transformation matrix $Z$ such that, let $d = m-r$
\beq\label{eq:M_decomposition_pd}
\tM =
\begin{bmatrix} X &  \\  & Y  \end{bmatrix} Z 
\begin{bmatrix}
D_{1} & & &\\
& \ddots &  & \\
&  & D_{d} &\\
& & & 0_{2r}
\end{bmatrix}
Z^T \begin{bmatrix} X^T  &  \\  & Y^T \end{bmatrix}
= W \Gamma W^T ,
\eeq
where $W = \begin{bmatrix} X &  \\  & Y  \end{bmatrix} Z $ and $\Gamma$ is the block diagonal matrix. 
The order of $D_i, i=1,...,d$ is such that it complies with Assumption \ref{assumption-type-iii-a}. 
Consider the following decomposition of $\Gamma$
\[
\Gamma_1
= 
\begin{bmatrix}
D_{1} & & &\\
& \ddots &  & \\
&  & D_{q} &\\
& & & 0_{n-2q}
\end{bmatrix}
\qandq
\Gamma_2
= \Gamma - \Gamma_1 .
\]
Let $\sigma = \sqrt{\delta\tau c_1^2 + a_1b_1}$ and $\eta = \frac{\sqrt{\delta\tau c_{q+1}^2 + a_{q+1}b_{q+1}}}{\sigma}$, then $\sfrac{1}{\sigma^k} \Gamma_2^k = O(\eta^k) \to 0$.

Follow the proof of Proposition \ref{prop:type-ii}, $\theta_k$ eventually is determined by the rotation property of $\Gamma_1$. 
Clearly, there exist some $\psi, \phi \in [0, \pi/2]$ and $l,s >0$ such that
\[
\sfrac{1}{{\sigma}} D_1
= \sfrac{1}{{\sigma}} \begin{bmatrix} a_1 & - \delta c_1  \\ \tau c_1  & b_1 \end{bmatrix} 
= \begin{bmatrix} \cos(\psi) & - \sin(\psi) \\ \sin(\psi) & \cos(\psi) \end{bmatrix} 
\begin{bmatrix} \cos(\phi) & \tfrac{s}{l} \sin(\phi) \\ - \tfrac{l}{s}\sin(\phi) & \cos(\phi) \end{bmatrix} . 
\]
Consequently, we get
\beq\label{eq:4-eq-type-iii}
\begin{aligned}
\cos(\psi)\cos(\phi) + \tfrac{l}{s} \sin(\psi) \sin(\phi) &= \sfrac{a_1}{{\sigma}} ,\\
\tfrac{s}{l}\cos(\psi)\sin(\phi) - \sin(\psi)\cos(\phi) &= - \sfrac{\delta c_1}{{\sigma}},\\
\sin(\psi)\cos(\phi)  - \tfrac{l}{s}\cos(\psi)\sin(\phi) &= \sfrac{\tau c_1}{{\sigma}}  , \\
\tfrac{s}{l}\sin(\psi)\sin(\phi) + \cos(\psi)\cos(\phi) &= \sfrac{b_1}{{\sigma}},
\end{aligned}
\eeq
which yields
\[
\psi = \mathrm{arccot}\bPa{ \tfrac{ (\tau-\delta)c_1 }{ b_1 - a_1 } }  ,~~ 
\phi = \arccos\bPa{ \tfrac{(\delta+\tau) c_1 \sin(\psi) + (a_1+b_1)\cos(\psi) }{2{\sigma}} } 
\qandq 
\sfrac{s}{l} = \tfrac{b_1}{\sin(\psi)\sin(\phi) {\sigma}} - \cot(\psi)\cot(\phi).
\]
{Moreover, we have
\[
\begin{aligned}
\Pa{ \pa{ \tfrac{l}{s} + \tfrac{s}{l} } \sin(\psi)\sin(\phi) + 2\cos(\psi)\cos(\phi) }^2 - 4 
= \sfrac{(a_1+b_1)^2}{\sigma^2} - 4 
&= \sfrac{(a_1-b_1)^2 - 4\delta\tau c_1^2}{\sigma^2} < 0 .
\end{aligned}
\]}
This means that $\sfrac{1}{{\sigma}} D_1$ is the composite rotation of Proposition \ref{prop:comp-circular-elliptical}, and so is $\frac{1}{\sigma}\Gamma_1$. 
Therefore, invoking the result of Proposition \ref{prop:elliptical-rotation} and \ref{prop:comp-circular-elliptical}, we have $\theta_k \in \big[ \psi - \chiM , \psi - \chim \big] $ 
with 
$$
\cos(\chiM)
= \frac{ \pa{ \frac{s}{2l} + \frac{l}{2s} } \cos(\phi) - \abs{ \frac{s}{2l} - \frac{l}{2s} } }{ \ssqrt{ \sin^2(\phi)  + \pa{\pa{ \frac{s}{2l} + \frac{l}{2s} } \cos(\phi) - \abs{ \frac{s}{2l} - \frac{l}{2s} } }^2 } }
\qandq
\cos(\chim)
= \frac{ \pa{ \frac{s}{2l} + \frac{l}{2s} } \cos(\phi) + \abs{ \frac{s}{2l} - \frac{l}{2s} } }{ \ssqrt{ \sin^2(\phi)  + \pa{\pa{ \frac{s}{2l} + \frac{l}{2s} } \cos(\phi) + \abs{ \frac{s}{2l} - \frac{l}{2s} } }^2 } }   . \qedhere
$$

\end{proof}

\section{Proofs of Section \ref{sec:trajectory-fom}}\label{proof:trajectory-fom}

\subsection{Riemannian Geometry}
\label{subsec:riemgeom}

Let $\calM$ be a $C^2$-smooth embedded submanifold of $\bbR^n$ around a point $x$. With some abuse of terminology, we shall state $C^2$-manifold instead of $C^2$-smooth embedded submanifold of $\bbR^n$. The natural embedding of a submanifold $\calM$ into $\bbR^n$ permits to define a Riemannian structure and to introduce geodesics on $\calM$, and we simply say $\calM$ is a Riemannian manifold. We denote respectively $\tanSp{\Mm}{x}$ and $\normSp{\Mm}{x}$ the tangent and normal space of $\Mm$ at point near $x$ in $\calM$.
%
%

\vgap

{\noindent}\textbf{Exponential map}~~
Geodesics generalize the concept of straight lines in $\bbR^n$, preserving the zero acceleration characteristic, to manifolds. Roughly speaking, a geodesic is locally the shortest path between two points on $\calM$. We denote by $\mathfrak{g}(t;x, h)$ the value at $t \in \bbR$ of the geodesic starting at $\mathfrak{g}(0;x,h) = x \in \calM$ with velocity $\dot{\mathfrak{g}}(t;x, h) = \qfrac{d\mathfrak{g}}{dt}(t;x,h) = h \in \tanSp{\Mm}{x}$ (which is uniquely defined). 
%
For every $h \in \tanSp{\calM}{x}$, there exists an interval $I$ around $0$ and a unique geodesic $\mathfrak{g}(t;x, h): I \to \calM$ such that $\mathfrak{g}(0; x, h) = x$ and $\dot{\mathfrak{g}}(0;x, h) = h$.
The mapping
 \[
\Exp_x 
: \tanSp{\calM}{x} \to  \calM ,~~   h\mapsto \Exp_{x}(h) =  \mathfrak{g}(1;x, h) ,
\]
is called \emph{Exponential map}.
Given $x, x' \in \calM$, the direction $h \in \tanSp{\calM}{x}$ we are interested in is the one such that $\Exp_x(h) = x' = \mathfrak{g}(1;x, h)$. 

\vgap

{\noindent}\textbf{Parallel translation}~~
Given two points $x, x' \in \calM$, let $\tanSp{\calM}{x}, \tanSp{\calM}{x'}$ be their corresponding tangent spaces. Define 
\[
\tau : \tanSp{\calM}{x} \to \tanSp{\calM}{x'} ,
\]
the parallel translation along the unique geodesic joining $x$ to $x'$, which is isomorphism and isometry with respect to the Riemannian metric.

\vgap

{\noindent}\textbf{Riemannian gradient and Hessian}~~
For a vector $v \in \normSp{\calM}{x}$, the Weingarten map of $\calM$ at $x$ is the operator $\Wgtmap{x}\pa{\cdot, v}: \tanSp{\calM}{x} \to \tanSp{\calM}{x}$ defined by
\[
\Wgtmap{x}\pa{\cdot, v} = - \PT{\tanSp{\calM}{x}} \mathrm{d} V[h]  ,
\]
where $V$ is any local extension of $v$ to a normal vector field on $\calM$. The definition is independent of the choice of the extension $V$, and $\Wgtmap{x}(\cdot, v)$ is a symmetric linear operator which is closely tied to the second fundamental form of $\calM$, see \cite[Proposition II.2.1]{chavel2006riemannian}.

Let $G$ be a real-valued function which is $C^2$ along the $\calM$ around $x$. The covariant gradient of $G$ at $x' \in \calM$ is the vector $\nabla_{\calM} G(x') \in \tanSp{\calM}{x'}$ defined by
\[
\iprod{\nabla_{\calM} G(x')}{h} = \qfrac{d}{dt} G\Pa{\PT{\calM}(x'+th)}\big|_{t=0} ,~~ \forall h \in \tanSp{\calM}{x'},
\]
where $\PT{\calM}$ is the projection operator onto $\calM$. 
The covariant Hessian of $G$ at $x'$ is the symmetric linear mapping $\nabla^2_{\calM} G(x')$ from $\tanSp{\calM}{x'}$ to itself which is defined as
\beq\label{eq:rh}
\iprod{\nabla^2_{\calM} G(x') h}{h} = \qfrac{d^2}{dt^2} G\Pa{\PT{\calM}(x'+th)}\big|_{t=0} ,~~ \forall h \in \tanSp{\calM}{x'} .
\eeq
This definition agrees with the usual definition using geodesics or connections \cite{miller2005newton}. 
Now assume that $\calM$ is a Riemannian embedded submanifold of $\bbR^{n}$, and that a function $G$ has a $C^2$-smooth restriction on $\calM$. This can be characterized by the existence of a $C^2$-smooth extension (representative) of $G$, \ie a $C^2$-smooth function $\widetilde{G}$ on $\bbR^{n}$ such that $\widetilde{G}$ agrees with $G$ on $\calM$. Thus, the Riemannian gradient $\nabla_{\calM}G(x')$ is also given by
\beq\label{eq:RieGradient}
\nabla_{\calM} G(x') = \PT{\tanSp{\calM}{x'}} \nabla \widetilde{G}(x')  ,
\eeq
and $\forall h \in \tanSp{\calM}{x'}$, the Riemannian Hessian reads
\beq\label{eq:RieHessian}
\begin{aligned}
\nabla^2_{\calM} G(x') h
&= \PT{\tanSp{\calM}{x'}} \mathrm{d} (\nabla_{\calM} G)(x')[h]
= \PT{\tanSp{\calM}{x'}} \mathrm{d} \Pa{ x' \mapsto \PT{\tanSp{\calM}{x'}} \nabla_{\calM} \widetilde{G} }[h]   \\ 
&= \PT{\tanSp{\calM}{x'}} \nabla^2 \widetilde{G}(x') h + \Wgtmap{x'}\Pa{h, \PT{\normSp{\calM}{x'}}\nabla \widetilde{G}(x')}  ,
\end{aligned}
\eeq
where the last equality comes from  \cite[Theorem~1]{absil2013extrinsic}. 
When $\calM$ is an affine or linear subspace of $\bbR^{n}$, then obviously $\calM = x + \tanSp{\calM}{x}$, and $\Wgtmap{x'}\pa{h, \PT{\normSp{\calM}{x'}} \nabla \widetilde{G}(x')} = 0$, hence \eqref{eq:RieHessian} reduces to
\[
\nabla^2_{\calM} G(x')
= \PT{\tanSp{\calM}{x'}} \nabla^2 \widetilde{G}(x')  \PT{\tanSp{\calM}{x'}}  .
\]
See \cite{lee2003smooth,chavel2006riemannian} for more materials on differential and Riemannian manifolds.

\vgap

The following lemma presents the expressions of the Riemannian gradient and Hessian for the case of partly smooth functions relative to a $C^2$-smooth manifold. 
 The result follows by combining \eqref{eq:RieGradient}, \eqref{eq:RieHessian}, Definition \ref{dfn:psf} and \cite[Proposition~17]{Daniilidis06} (or \cite[Lemma~2.4]{miller2005newton}).

\begin{lemma}[Riemannian gradient and Hessian]\label{def:riemannian-gradhess}
If $R \in \PSF{x}{\calM_x}$, then for any point $x' \in \calM_x$ near $x$
\[
\nabla_{\calM_x} R(x') = \proj_{T_{x'}}\pa{\partial R(x')} ,
\]
and this does not depend on the smooth representation of $R$ on $\calM_x$. In turn, for all $h \in T_{x'}$, let $\widetilde{R}$ be a smooth representative of $R$ on $\calM_x$,
\[
\nabla^2_{\calM_x} R(x')h = \proj_{T_{x'}} \nabla^2 \widetilde{R}(x')h + \Wgtmap{x'}\Pa{h, \proj_{T_{x'}^\perp}\nabla \widetilde{R}(x')}  ,
\]
where $\Wgtmap{x}\pa{\cdot, \cdot}: T_x \times T_x^\perp \to T_x$ is the Weingarten map of $\calM_x$ at $x$. 
\end{lemma}

The result of Lemma \ref{def:riemannian-gradhess} implies that we can linearize the proximity operators along the $C^2$-smooth manifold, which is discussed in Lemma \ref{lem:lin-generalised-ppa}.

\begin{lemma}[{\cite[Lemma 5.1]{liang2014local}}]
\label{lem:proj-M}
Let $\calM$ be a $C^2$-smooth manifold around $x$. Then for any $x' \in \calM \cap \calN$, where $\calN$ is a neighborhood of $x$, the projection operator $\proj_{\calM}(x')$ is uniquely valued and $C^1$ around $x$, and thus 
\[
x' - x = \proj_{\tanSp{\calM}{x}}(x'-x) + o\pa{\norm{x'-x}} .
\]
If moreover $\calM = x + \tanSp{\calM}{x}$ is an affine subspace, then $x' - x = \proj_{\tanSp{\calM}{x}}(x'-x)$.
\end{lemma}
%

\begin{lemma}[{\cite[Lemma B.1]{liang2017activity}}]\label{lem:parallel-translation}
Let $x \in \calM$, and $\xk$ a sequence converging to $x$ in $\calM$. Denote $\tau_k : \tanSp{\calM}{x} \to \tanSp{\calM}{\xk}$ be the parallel translation along the unique geodesic joining $x$ to $\xk$. Then, for any bounded vector $u \in \bbR^n$, we have
\[
\pa{\tau_k^{-1}\proj_{\tanSp{\calM}{\xk}} - \proj_{\tanSp{\calM}{x}}}u = o(\norm{u})  .
\]
\end{lemma}
%

\begin{lemma}[{\cite[Lemma B.2]{liang2017activity}}]\label{lem:taylor-expn}
Let $x, x'$ be two close points in $\calM$, denote $\tau : \tanSp{\calM}{x} \to \tanSp{\calM}{x'}$ the parallel translation along the unique geodesic joining $x$ to $x'$. The Riemannian Taylor expansion of $\Phi \in C^2(\calM)$ around $x$ reads,
\[\label{eq:taylor-expn}
\tau^{-1} \nabla_{\calM} \Phi(x') = \nabla_{\calM} \Phi(x) + \nabla^2_{\calM} \Phi(x)\proj_{\tanSp{\calM}{x}}(x'-x) + o(\norm{x'-x})  .
\]
\end{lemma}
%

\subsection{Linearization of proximal mapping}

When dealing with non-smooth optimization, one fundamental result, provided by partial smoothness, is the linearization of proximal mapping. 
We first discuss the property of the Riemannian Hessian of a partly smooth function. Let $R \in \lsc(\bbR^n)$ be partly smooth at $\xbar$ relative to $\calM_{\xbar}$ and $\ubar \in \partial R(\xbar)$, define the following smooth perturbation of $R$
\beq\label{eq:smooth-pert}
\barR(x) \eqdef R(x) - \iprod{x}{\ubar} ,
\eeq
whose Riemannian Hessian at $\xbar$ reads $H_{\barR} \eqdef \proj_{T_{\xbar}} \nabla^2_{\calM_{\xbar}} \barR \pa{\xbar} \proj_{T_{\xbar}}  $.

\begin{lemma}[{\cite[Lemma 4.2]{liang2017activity}}]\label{lem:riemhesspsd} 
Let $R \in \lsc(\bbR^n)$ be partly smooth at $\xbar$ relative to $\calM_{\xbar}$, then $H_{\barR}$ is symmetric positive semi-definite if either of the following is true:
\begin{itemize}[leftmargin=2em]
\item $\ubar \in \ri(\partial R(\xbar))$ is non-degenerate. 
\item $\calM_{\xbar}$ is an affine subspace.
\end{itemize}
In turn, $\Id + H_{\barR}$ is invertible and $\pa{\Id + H_{\barR}}^{-1}$ is symmetric positive definite with all eigenvalues in~$]0,1]$.
\end{lemma}

Together with the previous results, Lemma \ref{lem:riemhesspsd} allows us to linearize the generalized proximal mapping defined below. 

\begin{definition}[Generalised proximal mapping]\label{def:gprox}
Let $R \in \lsc(\bbR^n)$ and $\gamma > 0$, the generalized proximal mapping of $R$ is defined by
\beq\label{eq:gprox}
\prox_{\gamma R}^{A}(\wbar) \eqdef \argmin_{x\in \bbR^n} \gamma R(x) + \sfrac{1}{2} \norm{Ax - \wbar}^2  ,
\eeq
where $\wbar \in \bbR^p$ and $A \in \bbR^{p \times n}$ has full column rank.
\end{definition}

Since $A$ has full column rank, $\prox_{\gamma R}^{A}$ is a single-valued mapping. When $A=\Id$, \eqref{eq:gprox} reduces to the standard definition of proximity operator. 
Denote $\xbar \eqdef \prox_{\gamma R}^{A}(\wbar)$, owing to the optimality condition, we have $\ubar \eqdef - A^T(A\xbar - \wbar)/\gamma \in \partial R(\xbar)$. Suppose $R$ is partly smooth at $\xbar$ relative to $\calM_{\xbar}$, and let $\barR \eqdef \gamma R(x) - \iprod{x}{\gamma \ubar} $ be a smooth perturbation of $\gamma R$. 
Define $A_{T_{\xbar}} = A \circ \PT{T_{\xbar}}$ which also has full column rank, hence $A_{T_{\xbar}}^TA_{T_{\xbar}}$ is invertible and we define 
\[
M_{\barR} \eqdef A_{T_{\xbar}} \pa{ \Id + (A_{T_{\xbar}}^T A_{T_{\xbar}})^{-1} H_{\barR} }^{-1} (A_{T_{\xbar}}^T A_{T_{\xbar}})^{-1} A_{T_{\xbar}}^T  . 
\]

\begin{lemma}[{\cite[Lemma C.9]{PoonLiang2019b}}]
\label{lem:lin-generalised-ppa}
For the proximal mapping defined in \eqref{eq:gprox}, suppose $R \in \lsc(\bbR^n)$ is partly smooth at $\xbar$ relative to a $C^2$-smooth manifold $\calM_{\xbar}$ and $\ubar \in \ri \pa{\partial R(\xbar)}$. 
Let $\seq{\wk}$ be a sequence such that $\wk \to \wbar$ and $\xk  = \prox_{\gamma R}^{A}(\wk) \to \xbar$, then for all $k$ large enough, there hold $\xk \in \calM_{\xbar}$ and 
\beq\label{eq:lin-generalised-ppa}
A_{T_{\xbar}} (\xk-\xkm)  = M_{\barR} \pa{\wk - \wkm} + o(\norm{\wk-\wkm})   . 
\eeq
\end{lemma}
\begin{remark}
When $A = \Id$, $\prox_{\gamma R}^{A}$ reduces to the standard proximal mapping and \eqref{eq:lin-generalised-ppa} simplifies to
\[
\xk-\xkm  = \PT{T_{\xbar}} \pa{\Id + H_{\barR}}^{-1} \PT{T_{\xbar}} \pa{\wk - \wkm} + o(\norm{\wk-\wkm})   . 
\]
In \cite{liang2016thesis} and references therein, to study the local linear convergence of first-order methods, linearization with respect to the limiting point $\xbar$ is provided, that is
$
\xk-\xbar  = \PT{T_{\xbar}} \pa{\Id + H_{\barR}}^{-1} \PT{T_{\xbar}} \pa{\wk - \wbar} + o(\norm{\wk-\wbar}) 
$. 
\end{remark}
For the sake of completeness, we provide the proof of Lemma \ref{lem:lin-generalised-ppa} below.
\begin{proof}
Since $R$ is proper convex and lower semi-continuous, we have $R(\xk) \to R(\xbar)$ and $\partial R(\xk) \ni \uk = - A^T (A\xk - \wk)/\gamma \to \ubar \in \ri \pa{\partial R(\xbar)}$, hence $\dist\pa{\uk, \partial R(\xbar)} \to 0$. As a result, we have $\xk \in \calM_{\xbar}$ owing to \cite[Theorem~5.3]{hare2004identifying} and $\uk \in \ri\pa{\partial R(\xk)}$ owing to \cite{vaiter2018model} for all $k$ large enough.

Denote $T_{\xk}, T_{\xkm}$ the tangent spaces of $\calM_{\xbar}$ at $\xk$ and $\xkm$. Denote $\tau_{k} : T_{\xk} \to T_{\xkm}$ the parallel translation along the unique geodesic on $\calM_{\xbar}$ joining $\xk$ to $\xkm$. 
From the definition of $\xk$, let $\hk = \gamma \uk$, we get
\[
\hk \eqdef - A^T\pa{ A\xk - \wk } \in \gamma \partial R(\xk)   
\qandq
\hkm \eqdef - A^T\pa{ A\xkm - \wkm } \in \gamma \partial R(\xkm)   .
\]
Projecting onto corresponding tangent spaces, applying Lemma \ref{def:riemannian-gradhess} 
and the parallel translation $\tau_{k}$ leads to
\[
\begin{aligned}
\gamma \tau_{k} \nabla_{\calM_{\xbar}} R(\xk) 
&= \tau_{k} \PT{T_{\xk}} \pa{ \hk } 
= \PT{T_{\xkm}} \pa{ \hk } + \Pa{\tau_{k} \PT{T_{\xk}} - \PT{T_{\xkm}}} \pa{ \hk } , \\
\gamma \nabla_{\calM_{\xbar}} R(\xkm) 
&= \PT{T_{\xkm}} \pa{ \hkm } .
\end{aligned}
\]
The difference of the above two equalities yields
\beq\label{eq:wk-xk}
\begin{aligned}
\gamma \tau_{k} \nabla_{\calM_{\xbar}} R(\xk)  - \gamma \nabla_{\calM_{\xbar}} R(\xkm) - \Pa{\tau_{k} \PT{T_{\xk}} - \PT{T_{\xkm}}} \pa{ \hkm  }  
&= \PT{T_{\xkm}} \pa{ \hk - \hkm } + \Pa{\tau_{k} \PT{T_{\xk}} - \PT{T_{\xkm}}} \pa{ \hk - \hkm }   .
\end{aligned}
\eeq
Owing to the monotonicity of sub-differential, \ie $\iprod{\hk-\hkm}{\xk-\xkm} \geq 0$, we get
\[
\iprod{A^TA(\xk-\xkm)}{\xk-\xkm} 
\leq \iprod{A^T(\wk-\wkm)}{\xk-\xkm}  
\leq \norm{A}\norm{\wk-\wkm}\norm{\xk-\xkm}  .
\]
Since $A$ has full column rank, $A^TA$ is symmetric positive definite, and there exists $\kappa > 0$ such that $\kappa \norm{\xk-\xkm}^2 \leq \iprod{A^TA(\xk-\xkm)}{\xk-\xkm}$. Back to the above inequality, we get $\norm{\xk-\xkm} \leq \frac{\norm{A}}{\kappa} \norm{\wk-\wkm}$.  
Therefore for $\norm{\hk-\hkm}$, we get
\[
\begin{aligned}
\norm{\hk - \hkm} 
=\norm{ A^T\pa{ A\xk - \wk } - A^T\pa{ A\xkm - \wkm } }  
\leq \norm{A}^2 \norm{\xk-\xkm} + \norm{A}\norm{\wk-\wkm}  
&\leq \Pa{ \sfrac{\norm{A}^3}{\kappa} + \norm{A}} \norm{\wk - \wkm}  .
\end{aligned}
\]
As a result, owing to Lemma \ref{lem:parallel-translation}, we have for the term $\pa{\tau_{k} \PT{T_{\xk}} - \PT{T_{\xkm}}} \pa{ \hk - \hkm }$ in \eqref{eq:wk-xk} that
\[
\Pa{\tau_{k} \PT{T_{\xk}} - \PT{T_{\xkm}}} \pa{ \hk - \hkm }	\\
= o(\norm{ \hk-\hkm } )   
= o(\norm{ \wk-\wkm } )   .
\] 
Define $\barR_{k-1}(x) \eqdef \gamma R(x) - \iprod{x}{\hkm} $ and $H_{\barR, k-1} \eqdef \PT{T_{\xkm}} \nabla_{\calM_{\xbar}}^2 \barR_{k-1}(\xkm) \PT{T_{\xkm}}$, then with Lemma \ref{lem:taylor-expn} the Riemannian Taylor expansion, we have for the first line of \eqref{eq:wk-xk}
\beq\label{eq:rie-hess-xk}
\begin{aligned}
 \gamma \tau_{k} \nabla_{\calM_{\xbar}} R(\xk) - \gamma \nabla_{\calM_{\xbar}} R(\xkm) - \Pa{\tau_{k} \PT{T_{\xk}} - \PT{T_{\xkm}}} \pa{\hkm} 
&= \tau_{k} \Pa{\gamma \nabla_{\calM_{\xbar}} R(\xk) - \PT{T_{\xk}} \pa{\hkm}} - \Pa{ \gamma \nabla_{\calM_{\xbar}} R(\xkm) -  \PT{T_{\xkm}} \pa{\hkm} } \\
&= \tau_{k} \nabla_{\calM_{\xbar}} \barR_{k-1} (\xk) - \nabla_{\calM_{\xbar}} \barR_{k-1} (\xkm)  \\
&= H_{\barR, k-1} (\xk-\xkm) + o(\norm{\xk-\xkm})  \\
&= H_{\barR, k-1} (\xk-\xkm) + o(\norm{\wk-\wkm})   .
\end{aligned}
\eeq
Back to \eqref{eq:wk-xk}, we get
\beq\label{eq:wk-xk-2}
\begin{aligned}
H_{\barR, k-1} (\xk-\xkm)
&= \PT{T_{\xkm}} \pa{ \hk - \hkm } + o(\norm{\wk-\wkm})    .
\end{aligned}
\eeq
Define $ \barR(x) \eqdef  \gamma R(x) - \iprod{x}{ \bar{h} } $ and $H_{\barR} = \PT{T_{\xbar}} \nabla^2_{\calM_{\xbar}} \barR\pa{\xbar} \PT{T_{\xbar}}$, then from \eqref{eq:wk-xk-2} that
\beq\label{eq:xk-xkm-HbarR}
\begin{aligned}
H_{\barR} (\xk-\xkm) + \Pa{ H_{\barR, k-1} - H_{\barR} } (\xk-\xkm)  
&=  \PT{T_{\xbar}} \pa{ \hk - \hkm } +  \Pa{ \PT{T_{\xkm}} - \PT{T_{\xbar}} } \pa{ \hk - \hkm }  + o(\norm{\wk-\wkm})  .
\end{aligned}
\eeq
Owing to continuity, we have $H_{\barR, k-1} \to H_{\barR}$ and $\PT{T_{\xkm}} \to \PT{T_{\xbar}}$, 
\[
\begin{gathered}
\lim_{k\to\pinf} \tfrac{ \norm{ \pa{ H_{\barR, k-1} - H_{\barR} } (\xk-\xkm) } }{ \norm{\xk-\xkm} }
\leq \lim_{k\to\pinf} \tfrac{ \norm{ H_{\barR, k-1} - H_{\barR} } \norm{ \xk-\xkm } }{ \norm{\xk-\xkm} }
= \lim_{k\to\pinf} \norm{ H_{\barR, k-1} - H_{\barR} } = 0  ,  \\
\lim_{k\to\pinf} \tfrac{ \norm{ \pa{ \PT{T_{\xkm}} - \PT{T_{\xbar}} } \pa{\wk - \wkm}  } }{ \norm{\wk-\wkm} }
\leq \lim_{k\to\pinf} \tfrac{ \norm{ \PT{T_{\xkm}} - \PT{T_{\xbar}} } \norm{ \wk-\wkm } }{ \norm{\wk-\wkm} }
= \lim_{k\to\pinf} \norm{ \PT{T_{\xkm}} - \PT{T_{\xbar}} } = 0  , 
\end{gathered}
\]
and $\lim_{k\to\pinf} \tfrac{ \norm{ \pa{ \PT{T_{\xkm}} - \PT{T_{\xbar}}} \pa{\xk - \xkm}  } }{ \norm{\xk-\xkm} } = 0$. 
Combining this with the definition of $\uk$, the fact that  $\xk-\xkm = \PT{T_{\xbar}}(\xk-\xkm) + o(\norm{\xk-\xkm})$ from Lemma \ref{lem:proj-M}, and denoting  $A_{T_{\xbar}} = A \circ \PT{T_{\xbar}}$, equation \eqref{eq:xk-xkm-HbarR} can be written as
\beq\label{eq:aaa}
\begin{aligned}
H_{\barR} (\xk-\xkm)  
=  \PT{T_{\xbar}} \pa{ \uk - \ukm } + o(\norm{\wk-\wkm})  
&=  - \PT{T_{\xbar}} \pa{ A^T\pa{ A\xk - \wk }  -  A^T\pa{ A\xkm - \wkm } } + o(\norm{\wk-\wkm})  \\
&=  - \PT{T_{\xbar}} A^T A \pa{ \xk - \xkm } + \PT{T_{\xbar}} A^T \pa{ \wk - \wkm } + o(\norm{\wk-\wkm})  \\
&=  - A_{T_{\xbar}}^T A_{T_{\xbar}} \pa{ \xk - \xkm } + A_{T_{\xbar}}^T \pa{ \wk - \wkm } + o(\norm{\wk-\wkm}) .
\end{aligned}
\eeq
Since $A$ has full rank, so is $A_{T_{\xbar}}$. Hence $A_{T_{\xbar}}^TA_{T_{\xbar}}$ is invertible and from above we have
\[
\begin{aligned}
\Pa{ \Id + (A_{T_{\xbar}}^T A_{T_{\xbar}})^{-1} H_{\barR} } (\xk-\xkm)  
&= (A_{T_{\xbar}}^T A_{T_{\xbar}})^{-1} A_{T_{\xbar}}^T \pa{\wk-\wkm} + o(\norm{\wk-\wkm})   .
\end{aligned}
\]
Denote $M_{\barR} = A_{T_{\xbar}} \pa{ \Id + (A_{T_{\xbar}}^T A_{T_{\xbar}})^{-1} H_{\barR} }^{-1} (A_{T_{\xbar}}^T A_{T_{\xbar}})^{-1} A_{T_{\xbar}}^T$, then
\beq\label{eq:xkxkm-MbarR-wkwkm}
\begin{aligned}
A_{T_{\xbar}} (\xk-\xkm)    
&= M_{\barR} \pa{\wk-\wkm} + o(\norm{\wk-\wkm})   ,
\end{aligned}
\eeq 
which concludes the proof. 
\qedhere
\end{proof}

\subsection{Proof of main results}\label{annex:linearisation}

\subsubsection{Forward--Backward splitting method}\label{proof:trajectory-fb}

\begin{proof}[Proof of Theorem \ref{thm:trajectory-fb}]
The proof of the result is split into several steps. 

\vgap

{\noindent}\textbf{Linearization of Forward--Backward splitting}~~
The convergence of $\seq{\xk}$ to a global minimizer $\xsol \in \Argmin(\Phi)$ can be guaranteed under proper choices of $\ak$ and $\gamma$, we refer to \cite{liang2017activity} and the reference therein for detailed discussion. 
When $R \in \PSF{\xsol}{\Msol}$ and the non-degeneracy condition \eqref{eq:ndc-fb} holds, then there exists $K > 0$ such that for all $k \geq K$,  there holds $\xk \in \Msol$; see \cite[Theorem 3.4]{liang2017activity}.

Denote $\usol = -\nabla F(\xsol)$, from the non-degeneracy condition \eqref{eq:ndc-fb} we have $\usol \in \ri \pa{\partial R(\xsol)}$. Define $\barR(x) = \gamma R(x) - \iprod{x}{-\nabla F(\xsol)}$, and the following matrices
\beqn
H_{\barR} \eqdef \PT{T_{\xsol}}  \nabla^2_{\Msol} {\barR}(\xsol) \PT{T_{\xsol}} 
\qandq
M_{\barR} \eqdef \PT{T_{\xsol}} (\Id + H_{\barR})^{-1}  \PT{T_{\xsol}}  .
\eeqn
Let $\wk \eqdef \xk - \gamma \nabla F(\xk)$, then the update of $\xk$ entails that $\xkp = \argmin_{x\in\bbR^n} \gamma R(x) + \frac{1}{2} \norm{x - \wk}^2$. 
Owing to Lemma \ref{lem:lin-generalised-ppa}, we get that
\beq\label{eq:lin-fb-1}
\xkp-\xk    
= M_{\barR} \pa{\wk-\wkm} + o(\norm{ \wk-\wkm } )  .
\eeq
Next we deal with the term $\wk-\wkm$, from the local $C^2$-smoothness of $F$ we get
\[
\begin{aligned}
\wk-\wkm
&= \xk - \xkm - \gamma \Pa{ \nabla F(\xk) - \gamma \nabla F(\xkm) }  \\
&= \xk - \xkm - \gamma \nabla^2 F(\xkm) (\xk - \xkm)  +  o(\norm{\xk-\xkm})  \\
&= \xk - \xkm - \gamma \nabla^2 F(\xsol) (\xk - \xkm) - \gamma \Pa{ \nabla^2 F(\xkm) - \nabla^2 F(\xsol) } (\xk - \xkm)  +  o(\norm{\xk-\xkm})   .
\end{aligned}
\]
Since $\xk \to \xsol$, we have $\nabla^2 F(\xkm) \to \nabla^2 F(\xsol)$, then from above we get
\[
\wk-\wkm
= \Pa{\Id - \gamma \nabla^2 F(\xsol) } \pa{ \xk - \xkm } +  o(\norm{\xk-\xkm})   .
\]
As $\Id-\gamma \nabla F$ is non-expansive, then $\norm{ \wk-\wkm } \leq \norm{ \xk-\xkm }$. Therefore, back to \eqref{eq:lin-fb-1}, 
\beq\label{eq:lin-fb-2}
\xkp-\xk    
= M_{\barR} \Pa{\Id - \gamma \nabla^2 F(\xsol) } \pa{\xk-\xkm} + o(\norm{\xk-\xkm})  ,
\eeq
which is the desired linearization with $\mFB = M_{\barR} \Pa{\Id - \gamma \nabla^2 F(\xsol)} $.

\vgap

{\noindent}\textbf{Spectral properties of $\mFB$}~~ 
In this part, we briefly discuss the spectrum of $\mFB$ where more detailed accountant can be found in \cite{liang2017activity} and \cite[Chapter 6]{liang2016thesis}. 
Owing to Lemma \ref{lem:riemhesspsd}, we have $H_{\barR}$ is symmetric positive semi-definite, hence $M_{\barR}$ is symmetric positive definite with all its eigenvalues in $]0, 1]$. As a result, we have
\beq\label{eq:MG-sim}
M_{\barR} \Pa{\Id - \gamma \nabla^2 F(\xsol)}
= M_{\barR}^{1/2} M_{\barR}^{1/2} \Pa{\Id - \gamma \nabla^2 F(\xsol)} M_{\barR}^{1/2} M_{\barR}^{-1/2} 
\sim M_{\barR}^{1/2} \Pa{\Id - \gamma \nabla^2 F(\xsol)} M_{\barR}^{1/2} ,
\eeq
where $M_{\barR}^{1/2} \Pa{\Id - \gamma \nabla^2 F(\xsol)} M_{\barR}^{1/2}$ is symmetric, with all its eigenvalues in $]-1, 1]$ since $\gamma \in ]0, 2/L[$. 

\vgap

{\noindent}\textbf{Trajectory of Forward--Backward splitting}~~
%
%
Let $\vk = \xk - \xkm$, from the above discussion, we have
\[
\begin{aligned}
\vkp 
&= \mFB \vk + o(\norm{\vk})  \\
&= M_{\barR}^{1/2} M_{\barR}^{1/2} \Pa{\Id - \gamma \nabla^2 F(\xsol)} M_{\barR}^{1/2} M_{\barR}^{-1/2} \vk + o(\norm{\vk}) .
\end{aligned}
\]
Denote $\uk = M_{\barR}^{-1/2} \vk$ and $\psi_k =  M_{\barR}^{-1/2} o(\norm{\vk})$, then we reach
\[
\begin{aligned}
\ukp 
&= M_{\barR}^{1/2} \Pa{\Id - \gamma \nabla^2 F(\xsol)} M_{\barR}^{1/2} \uk + \psi_k .
\end{aligned}
\]
Since $M \eqdef M_{\barR}^{1/2} \Pa{\Id - \gamma \nabla^2 F(\xsol)} M_{\barR}^{1/2}$ is symmetric positive semidefinite, there exist an orthogonal matrix $V$ such that $M = V \Sigma V^\top$ where $\Sigma$ is a diagonal matrix with eigenvalues of $M$ on its diagonal. Without loss of generality, let $\Sigma = \Sigma_1 + \Sigma_2$, where
$$
\Sigma_1 = \sigma_1 \begin{bmatrix}
\Id_{p\times p} &\\
& 0
\end{bmatrix}
\qandq 
\Sigma_2 =  \begin{bmatrix}
0_{p\times p} \\
& \sigma_2 &\\
& & \ddots &\\
&&& \sigma_N
\end{bmatrix}
$$
where $\sigma_1>\sigma_2\geq \sigma_3 \geq \cdots\geq \sigma_N$. Let $P_1 = \frac{1}{\sigma_1}\Sigma_1$ and $P_2 = \Id - P_1$.  
Then, letting $w_k = V^\top u_k$ and $\phi_k = V^\top \psi_k$, we have
$$
w_{k+1} = \Sigma_1 w_k + \Sigma_2 w_k + \phi_k.
$$
In particular, 
$$
\sigma_1 w_k - w_{k+1} = \sigma_1 P_2 w_k - \Sigma_2 w_k + \sigma \psi_{k-1}-\psi_k
$$
To bound $\norm{P_2 u_k}$, we have
\[
P_2 w_k = \Sigma_2 w_{k-1} + P_2 \phi_{k-1} = \Sigma_2^k w_{0} + \msum_{j=0}^{k-1} \Sigma_2^j P_2 \phi_{k-j}
\]
which implies
\[\norm{P_2 w_k} \leq \sigma_2^J \norm{P_2 w_{k-J}} + \msum_{j=0}^{J-1} \sigma_2^j o(\norm{w_{k-j}}) = \sigma_2^J \sigma_1^{k-J} +   \msum_{j=0}^{J-1} (\sigma_2/\sigma_1)^j   \alpha_{k-j} \sigma_1^{k}  = o(\sigma_1^k),
\]
where $\norm{\phi_k} = \alpha_k \sigma_1^k$ with $\alpha_k \to 0$, we choose $J= k/2$ and
since $\norm{w_k} = O(\sigma_1^k)$. Therefore,
$$
\min_c \norm{c w_k - w_{k+1}} \leq\norm{\sigma_1 w_k - w_{k+1}} = o(\sigma_1^{k}). 
$$
On the other hand, note that $c_{k}^\star = \argmin_c \norm{c w_k - w_{k+1}}$ satisfies $c_{k}^\star = \cos(\beta_{k+1}) \norm{w_{k+1}}/\norm{w_k}<1$ where $\beta_k$ is the angle between $\wkp,\wk$, therefore we have
\[
s_k \eqdef \norm{c_{k}^\star w_k - w_{k+1}} = \sin(\beta_{k+1}) \norm{\wk} = o(\norm{\wk}) ,
\]
which implies $\sin(\beta_{k+1}) \to 0$, hence we get $\beta_{k+1} \to 0$. 
Since $V$ is an orthogonal matrix, $\beta_{k+1}$ also is the angle between $\uk$ and $\ukp$. 
Back to $\vk$, for which we have
\[
\begin{aligned}
\vkp 
= M_{\barR}^{1/2} \ukp 
= M_{\barR}^{1/2} V \wkp
&= c_k^\star \cos(\beta_{k+1}) M_{\barR}^{1/2} V \wk + o(\norm{\wk})  \\
&= c_k^\star \cos(\beta_{k+1}) \vk + o(\norm{\vk})  .
\end{aligned}
\]
As a result, we obtain
\[
\begin{aligned}
\cos(\theta_{k+1})
= \sfrac{\iprod{\vk}{\vkp}}{\norm{\vk} \norm{\vkp}}
= \sfrac{\iprod{\vk}{c_k^\star \cos(\beta_{k+1}) \vk}}{\norm{\vk} \norm{c_k^\star \cos(\beta_{k+1}) vk + o(\norm{\vk})}} + \sfrac{\iprod{\vk}{o(\norm{\vk})}}{\norm{\vk} \norm{\vkp}} \to 1 ,
\end{aligned}
\] 
hence conclude the proof. \qedhere




\end{proof}

\subsubsection{Douglas--Rachford splitting and ADMM}\label{proof:trajectory-dr}

\begin{proof}[Proof of Theorems \ref{thm:linearization-dr} \& \ref{thm:trajectory-dr-1}]
The proof is also divided into parts. 

\vgap

{\noindent}\textbf{Linearization of Douglas--Rachford}~~ 
Owing to \cite[Theorem 5.1]{liang2017localDR}, we have that when $R \in \PSF{\xsol}{\MmR} ,~ J \in \PSF{\xsol}{\MmJ}$ and the non-degeneracy condition \eqref{eq:ndc-dr} holds, there exists $K > 0$ such that for all $k \geq K$, $(\uk,\xk) \in \MmR \times \MmJ$. 
%
%

From \eqref{eq:dr}, the update of $\xk$: define $\barJ(x) \eqdef \gamma J(x) - \iprod{x}{\zsol-\xsol} $, $H_{\barJ} \eqdef \PTJ{\xsol} \nabla_{\MmJ}^2 \barJ(\xsol) \PTJ{\xsol}$ and $\MJ = \PTJ{\xsol} \pa{\Id + H_{\barJ}}^{-1} \PTJ{\xsol}$, then from Lemma \ref{lem:lin-generalised-ppa} we get
\beq\label{eq:xkzk}
\xk-\xkm 
= \MJ \pa{\zk - \zkm} + o(\norm{\zk-\zkm})    .
\eeq
Now for $\ukp$, let $\wk = 2\xk - \zk$, we get from Lemma \ref{lem:lin-generalised-ppa} that
\beqn
\ukp - \uk
= \MR \pa{\wk - \wkm} + o(\norm{\wk-\wkm})  ,
\eeqn
Define $\barR(x) \eqdef \gamma R(x) - \iprod{x}{\xsol - \zsol}$, $H_{\barR} = \PTR{\xsol} \nabla^2_{\MmR} \barR \pa{\xsol} \PTR{\xsol}$ and $\MR = \PTR{\xsol} \pa{\Id + H_{\barR}}^{-1} \PTR{\xsol}$. 
Since $\norm{\xk-\xkm} \leq \norm{\zk-\zkm}$, we get from above that
\beq\label{eq:vkpzk}
\ukp - \uk
= 2\MR \MJ \pa{\zk - \zkm} - \MR \pa{\zk-\zkm} + o(\norm{\zk-\zkm})   .
\eeq
Summing up \eqref{eq:xkzk} and \eqref{eq:vkpzk}, we get
\beq\label{eq:zkp-zk-dr}
\begin{aligned}
\zkp - \zk
&=\pa{\zk + \ukp - \xk} - \pa{\zkm + \uk - \xkm}   \\
&= \pa{\zk - \zkm} + \pa{\ukp - \uk} - \pa{\xk - \xkm} \\
&= \pa{\Id + 2\MR \MJ - \MR - \MJ}\pa{\zk - \zkm} + o(\norm{\zk-\zkm})   \\
&= \mDR \pa{\zk - \zkm} + o(\norm{\zk-\zkm})   .
\end{aligned}
\eeq
Owing to Lemma \ref{lem:riemhesspsd}, we have $H_{\barR}, H_{\barJ}$ are symmetric positive semi-definite, hence maximal monotone, consequently $(\Id+H_{\barR})^{-1}, (\Id+H_{\barJ})^{-1}$ are firmly non-expansive. Since $\PTR{\xsol}, \PTJ{\xsol}$ are projection operators, thens firmly non-expansive. As a result, both $M_{\barR} = \PTR{\xsol} (\Id+H_{\barR})^{-1} \PTR{\xsol} , M_{\barJ} = \PTJ{\xsol} (\Id+H_{\barJ})^{-1} \PTJ{\xsol}$ are firmly non-expansive owing to \cite[Example 4.14]{bauschke2011convex}, and $\mDR$ is firmly non-expansive \cite[Proposition 4.31]{bauschke2011convex}. 

\vgap 

{\noindent}\textbf{Spectral properties of $\mDR$}~~
Here we present a brief summary on the spectral properties of $\mDR$ and refer to \cite{liang2017localDR,bauschke2014rate} and the reference therein for detailed analysis of the spectral properties of $\mDR$. 
When $R, J$ are locally polyhedral around $\xsol$, then $H_{\barR}, H_{\barJ}$ vanish and consequently 
\[
\begin{aligned}
M_{\barR} = \PTR{\xsol}  , \enskip M_{\barJ} = \PTJ{\xsol}  \qandq  
\mDR = \PTR{\xsol} \PTJ{\xsol} + \PSR{\xsol} \PSJ{\xsol} 
\end{aligned}
\]
where $S_{\xsol}^R = (T_{\xsol}^R)^\bot$ and $S_{\xsol}^J = (T_{\xsol}^J)^\bot$. 
Denote the dimension of $T_{\xsol}^R, T_{\xsol}^J$ are $\dim(T_{\xsol}^R) = p, \dim(T_{\xsol}^J) = q$. Without the loss of generality, we assume that $1 \leq p \leq q \leq n, v$ and $\dim(T_{\xsol}^R \cap T_{\xsol}^J) = d$. Consequently, there are $r=p-d$ principal angles $(\zeta_i)_{i=1,...,r}$ between $T_{\xsol}^R$ and $T_{\xsol}^J$ that are strictly greater than $0$ and smaller than $\pi/2$. Suppose that $\zeta_1\leq \dotsm \leq \zeta_r$.  
Define the following two diagonal matrices
\[
C = \diag\Pa{ \cos(\zeta_1), \dotsm, \cos(\zeta_r) }  \qandq
S = \diag\Pa{ \sin(\zeta_1), \dotsm, \sin(\zeta_r) }  .
\]
Owing to \cite{bauschke2014rate,demanet2016eventual}, there exists a real orthogonal matrix $U$ such that
\[
\mDR
= U
\left[
\begin{array}{cc|cc}
C^2 & CS & 0 & 0\\
- CS  &  C^2  & 0 & 0\\
\hline
0 & 0 & 0_{q-p+2d} & 0\\
0 & 0 & 0 & \Id_{n-p-q}
\end{array}
\right]
U^T ,
\]
which indicates $\mDR$ is normal and all its eigenvalues are inside unit disc.

\vgap

{\noindent}\textbf{Trajectory of Douglas--Rachford}~~
The above spectral properties of $\mDR$ indicates that $\mDR$ is a Type II matrix. 
Let $\mDRinf = \lim_{k\to+\infty}\mDR^k$ and $\tmDR = \mDR - \mDRinf$, then we have
\[
\tmDR
= U
\left[
\begin{array}{cc|c}
C^2 & CS & 0 \\
- CS  &  C^2  & 0 \\
\hline
0 & 0 & 0_{n-2r}
\end{array}
\right]
U^T  .
\]
Denote $\theta_{F} = \zeta_1$ the Friedrichs angle between $T_{\xsol}^R$ and $T_{\xsol}^J$, then invoking Proposition \ref{prop:type-ii} we obtain the trajectory of Douglas--Rachford splitting method. \qedhere
\end{proof}

\begin{proof}[{Proof of Theorem \ref{thm:trajectory-dr-2}}]
From the above proof, we have for $\xk$ that $\xk - \xkm = \MJ \pa{\zk - \zkm} + o(\norm{\zk-\zkm})$ Lemma \ref{lem:proj-M}.
Now for $\ukp$, since $R$ is smooth differentiable, we have
\[
2\xk - \zk - \ukp = \gamma \nabla R(\ukp)  
\qandq
2\xkm - \zkm - \uk = \gamma \nabla R(\uk) .
\]
Since we assume that $R$ is locally $C^2$-smooth around $\xsol$, then when $\ukp, \uk$ is close enough, \ie for sufficiently large $k$, 
\[
\begin{aligned}
\pa{2\xk - \zk - \ukp} - \pa{2\xkm - \zkm - \uk}  
&= \gamma \nabla R(\ukp) - \gamma \nabla R(\uk)\\
&= \gamma \nabla^2 R(\uk) (\ukp-\uk) + o(\norm{\ukp-\uk})  \\
&= \gamma \nabla^2 R(\usol) (\ukp-\uk) + \gamma \Pa{ \nabla^2 R(\uk) - \nabla^2 R(\usol)} (\ukp-\uk)  + o(\norm{\ukp-\uk})  \\
&= \gamma \nabla^2 R(\usol) (\ukp-\uk) + o(\norm{\ukp-\uk})  \\
&= \gamma \nabla^2 R(\usol) (\ukp-\uk) + o(\norm{\zk-\zkm})  ,
\end{aligned}
\]
and consequently, let $M_R = \pa{\Id + \gamma \nabla^2 R(\usol)}^{-1}$,
\[
\begin{aligned}
\ukp - \uk
&= M_R \Pa{2(\xk-\xkm) - (\zk-\zkm)} + o(\norm{\zk-\zkm})   \\
&= 2 M_R (\xk-\xkm)  - M_R (\zk-\zkm) + o(\norm{\zk-\zkm})   \\
&= 2 M_R \MJ(\xk-\xkm)  - M_R (\zk-\zkm) + o(\norm{\zk-\zkm})   .
\end{aligned}
\]
Summing up we get
\[
\begin{aligned}
\zkp - \zk
&= \pa{\zk - \zkm} + \pa{\ukp - \uk} - \pa{\xk - \xkm} \\
&= \pa{\Id + 2 M_R \MJ - M_R - \MJ}\pa{\zk - \zkm} + o(\norm{\zk-\zkm})   \\
&= \Pa{\sfrac{1}{2}\Id + \sfrac{1}{2} \pa{2 M_R - \Id} \pa{2\MJ - \Id} }  \pa{\zk - \zkm} + o(\norm{\zk-\zkm})    .
\end{aligned}
\]
Now we discuss the spectral property of $\mDR$. Owing to convexity, we have $\nabla^2 R(\usol)$ is symmetric positive semi-definite, hence maximal monotone and $M_R$ is firmly non-expansive. When $\gamma < \frac{1}{\norm{\nabla^2 R(\xsol)}}$, we have that all the eigenvalues of $M_R$ are in $]1/2, 1]$, consequently $W_R \eqdef 2 M_R - \Id$ is symmetric positive definite. Therefore, we get
\[
\begin{aligned}
\sfrac{1}{2}\Id + \sfrac{1}{2} W_R \Pa{2\MJ - \Id} 
&= W_R^{1/2} \bPa{ \sfrac{1}{2}\Id + \sfrac{1}{2} W_R^{1/2} \Pa{2\MJ - \Id} W_R^{1/2}  } W_R^{-1/2}  \\
&\sim \sfrac{1}{2}\Id + \sfrac{1}{2} W_R^{1/2} \Pa{2\MJ - \Id} W_R^{1/2}  ,
\end{aligned}
\]
and $\frac{1}{2}\Id + \frac{1}{2} W_R^{1/2} \pa{2\MJ - \Id} W_R^{1/2}$ is symmetric positive semi-definite with all it eigenvalues in $[0,1]$. 
Therefore, following the proof of Theorem \ref{thm:trajectory-fb} for Forward--Backward splitting, we obtain that the angle $\theta_k $ is convergent to $0$. 
\end{proof}

\subsubsection{Primal--Dual splitting}\label{proof:trajectory-pd}

\begin{proof}[Proof of Theorems \ref{thm:linearization-pd} \& \ref{thm:trajectory-pd-1}]
Similar to above, the proof is divided into parts. 

\vgap

{\noindent}\textbf{Linearization of Primal--Dual}~~
From the $\xk$ in \eqref{eq:pd}, define $\barR(x) \eqdef \gammaR R(x) - \iprod{x}{- \gammaR L^T\wsol}$ and $H_{\barR} \eqdef \PTR{\xsol} \nabla_{\MxR}^2 \barR(\xsol) \PTR{\xsol}$. 
Applying Lemma \ref{lem:lin-generalised-ppa} we get
\beq\label{eq:xkp-xsol-2}
\begin{aligned}
 {\xkp-\xk} 
&= M_{\barR} \pa{ \xk - \xkm } - \gammaR M_{\barR} \barL^T \pa{ \wk - \wkm } + o\pa{ \norm{\xk - \xkm} + \gammaR \norm{L} \norm{\wk - \wkm} } .
\end{aligned}
\eeq
Now for the update of $\wkp$. Since $\tau \in [0, 1]$, we have
\[
\begin{aligned}
\norm{\xbarkp-\xbark}
\leq (1+\tau) \norm{\xkp-\xk} + \tau\norm{\xk-\xkm} 
&\leq 4 \pa{ \norm{\xk - \xkm} + \gammaR \norm{L} \norm{\wk - \wkm} } + \norm{\xk-\xkm} \\
&= 5 \norm{\xk - \xkm} + 4\gammaR \norm{L} \norm{\wk - \wkm}  .
\end{aligned}
\]
Define $\barJc(w) \eqdef \gammaJ J^*(w) - \iprod{w}{\gammaJ L \xsol}$ and $H_{\barJc} = \PTJc{\wsol} \nabla_{\MwJc}^2 {\overbar{J^{*}}}(\wsol) \PTJc{\wsol}$, applying Lemma \ref{lem:lin-generalised-ppa} then yields
\beq\label{eq:ykp-ysol-2}
\begin{aligned}
\wkp-\wk
&= M_{\barJc} (\wk-\wkm) + \gammaJ M_{\barJc} \barL \pa{ \xbarkp - \xbark } + \textrm{small~$o$-terms}  \\
&= M_{\barJc} (\wk-\wsol) + (1+\tau)\gammaJ M_{\barJc} \barL (\xkp-\xk) - \tau\gammaJ M_{\barJc} \barL (\xk-\xkm) + \textrm{small~$o$-terms}  \\
&= M_{\barJc} (\wk-\wkm) - \tau\gammaJ M_{\barJc} \barL (\xk-\xkm) \\&\qquad + (1+\tau)\gammaJ M_{\barJc} \barL \Pa{ M_{\barR}  \pa{ \xk - \xkm } - \gammaR M_{\barR} \barL^T \pa{ \wk - \wkm } } + \textrm{small~$o$-terms} \\
&= \Pa{ M_{\barJc} - (1+\tau)\gammaJ\gammaR M_{\barJc} \barL M_{\barR} \barL^T } (\wk-\wkm)  + \Pa{ (1+\tau)\gammaJ M_{\barJc} \barL M_{\barR}  - \tau\gammaJ M_{\barJc} \barL } \pa{ \xk - \xkm }  \\&\qquad + o(\norm{\wkp-\wk}) + o(\norm{\xkp-\xk}) \\&\qquad + o\pa{ \norm{\xk - \xkm} + \gammaR \norm{L} \norm{\wk - \wkm} } + o\pa{ \norm{\wk-\wkm} + \gammaJ\norm{L} \norm{\xk - \xkm} }  .
\end{aligned}
\eeq
Combining \eqref{eq:xkp-xsol-2} and \eqref{eq:ykp-ysol-2}, we get
\beq\label{eq:combination}
\begin{aligned}
\begin{pmatrix} \xkp - \xk \\ \wkp - \wk \end{pmatrix}
&= \begin{bmatrix} M_{\barR} & - \gammaR M_{\barR} \barL^T \\ (1+\tau)\gammaJ M_{\barJc} \barL M_{\barR}  - \tau\gammaJ M_{\barJc} \barL & M_{\barJc} - (1+\tau)\gammaJ\gammaR M_{\barJc} \barL M_{\barR} \barL^T \end{bmatrix}
\begin{pmatrix} \xk - \xkm \\ \wk - \wkm \end{pmatrix}  \\&\qquad
+ o(\norm{\wkp-\wk}) + o( \norm{\xkp-\xk} ) + o\pa{ \norm{\xk - \xkm} + \gammaR \norm{L} \norm{\wk - \wkm} }  \\ &\qquad
+ o\pa{ \norm{\wk-\wkm} + \gammaJ\norm{L} \norm{\xk - \xkm} }
\end{aligned}
\eeq
Now we consider the small $o$-terms. 
Let $a_1,a_2$ be two constants, then we have
\[
\abs{a_1} + \abs{a_2}
= \sqrt{\pa{\abs{a_1} + \abs{a_2}}^2}
\leq \sqrt{2(a_1^2 + a_2^2)}
= \sqrt{2}\begin{Vmatrix} \begin{pmatrix} a_1 \\ a_2 \end{pmatrix} \end{Vmatrix} .
\]
Denote $b = \max\ba{1, \gammaJ\norm{L}, \gammaR\norm{L}}$, then
\[
\begin{aligned}
& \pa{\norm{\wk-\wkm} + \gammaJ\norm{L} \norm{\xk - \xkm}} + \pa{ \norm{\xk - \xkm} + \gammaR \norm{L} \norm{\wk - \wkm} } \\
&\leq 2b \pa{ \norm{\xk - \xkm} + \norm{\wk - \wkm} }
\leq 2\sqrt{2}b \begin{Vmatrix} \begin{pmatrix} \xk-\xkm \\ \wk-\wkm \end{pmatrix} \end{Vmatrix} .
\end{aligned}
\]
Then for $\norm{\wkp-\wk} + \norm{\xkp-\xk}$, we have
\[
\norm{\wkp-\wk} + \norm{\xkp-\xk}
\leq \sqrt{2} \begin{Vmatrix} \begin{pmatrix} \xkp-\xk \\ \wkp-\wk \end{pmatrix} \end{Vmatrix}  .
\]
Combining these into the small $o$-terms of \eqref{eq:combination}, we obtain
\[
\zkp - \zk = \mPD (\zk-\zkm) + o(\norm{\zk-\zkm})  .
\]

\vgap

{\noindent}\textbf{Spectral properties of $\mPD$}~~ 
We refer to \cite[Proposition 3.5]{liang2018localPD} about the non-expansiveness of $\mPD$, below we provide the spectral analysis of $\mPD$ for the case when both $R, J^*$ are locally polyhedral around $\xsol$ and $\wsol$ respectively, the analysis can also be found in \cite{liang2018localPD}.

When $R, J$ are locally polyhedral around $\xsol$, then $H_{\barR}, H_{\barJ}$ vanish and consequently 
\[
\begin{aligned}
M_{\barR} = \Id_{n}  , \enskip M_{\barJ} = \Id_{m}  \qandq  
\mPD 
= \begin{bmatrix} \Id_{n} & - \gammaR \barL^T \\ \gammaJ \barL  & \Id_{m} - (1+\tau)\gammaJ\gammaR \barL \barL^T \end{bmatrix}   .
\end{aligned}
\]
%
Let $p \eqdef \dim(\TR{\xsol}), q \eqdef \dim(\TJc{\wsol})$ be the dimensions of $\TR{\xsol}$ and $\TJc{\wsol}$ respectively, define $\SR{\xsol} = (\TR{\xsol})^{\bot}$ and $\SJc{\wsol} = (\TJc{\wsol})^{\bot}$.
Assume that $q \geq p$, where as the other direction can be treated similarly. 
Let $\barL = X \Sigma_{\barL} Y^T$ the singular value decomposition of $\barL$, denote the rank of $\barL$ as $l \eqdef \rank(\barL)$. Clearly, we have $l \leq p$.
With the SVD of $\barL$, for $\mPD$, we have
\beq\label{eq:mtx-hatM-2}
\begin{aligned}
\mPD
= 
\begin{bmatrix}
\Idn & - \gammaR \barL^T  \\
\gammaJ \barL &  \Idm - (1+\tau) \gammaR \gammaJ \barL \barL^T 
\end{bmatrix} 
&= 
\begin{bmatrix}
Y Y^T &  - \gammaR Y \Sigma_{\barL}^* X^T \\
\gammaJ X \Sigma_{\barL} Y^T  &  X X^T - (1+\tau) \gammaR \gammaJ X \Sigma_{\barL}^2 X^T
\end{bmatrix}  \\
&= 
\begin{bmatrix} Y & ~ \\ ~ & X \end{bmatrix}
\underset{W}{\underbrace{
\begin{bmatrix}
\Idn  &  - \gammaR \Sigma_{\barL}^*  \\
\gammaJ \Sigma_{\barL}  &  \Idm  - (1+\tau) \gammaR \gammaJ \Sigma_{\barL}^2 
\end{bmatrix} 
}}
\begin{bmatrix} Y^T & ~ \\ ~ & X^T \end{bmatrix} .
\end{aligned}
\eeq
Since we assume that $\rank(\barL) = l \leq p$, then $\Sigma_{\barL}$ can be represented as $\Sigma_{\barL}
=
\begin{bmatrix}
\Sigma_{l} & 0_{l, n-l} \\
0_{m-l, l} & 0_{m-l, n-l}  
\end{bmatrix} $ 
where $\Sigma_{l} = (\sigma_{j})_{j=1,...,l}$. 
Back to $W$, we have there exists an elementary transformation $E$ such that
\[
W\!=\!
\begin{bmatrix}
\Id_{l} & 0_{l, n-l} & - \gammaR \Sigma_{l} & 0_{l, m-l} \\
0_{n-l,l} & \Id_{n-l} & 0_{n-l,l} & 0_{n-l, m-l} \\
\gammaJ \Sigma_{l} & 0_{l, n-l} & \Id_{l}-(1+\tau) \gammaR \gammaJ \Sigma_{l}^2 & 0_{l, m-l} \\
0_{m-l, l} & 0_{m-l, n-l} & 0_{m-l, l} & \Id_{m-l} 
\end{bmatrix} 
\!=\!
E\!
\begin{bmatrix}
\Id_{l} & - \gammaR \Sigma_{l} & 0_{l, m+n-2l} \\
\gammaJ \Sigma_{l} & \Id_{l}-(1+\tau) \gammaR \gammaJ \Sigma_{l}^2 & 0_{l, m+n-2l} \\
0_{m+n-2l,l} & 0_{m+n-2l,l} & \Id_{m+n-2l} 
\end{bmatrix} 
\!\! E .
\]
Clearly, $1$ is an eigenvalue of $\mPD$ with multiplicity $m+n-2l$. 
Next we deal with the block diagonal matrix
\[
D = 
\begin{bmatrix}
\Id_{l} & - \gammaR \Sigma_{l}  \\
\gammaJ \Sigma_{l} & \Id_{l}-(1+\tau) \gammaR \gammaJ \Sigma_{l}^2 \end{bmatrix}  .
\]
Again, there exists another elementary transformation $E'$ such that
\beq\label{eq:mtxD}
D = 
E'
\begin{bmatrix}
D_{1} & &  \\
& \ddots &  \\
& & D_{l}  ,
\end{bmatrix}  
E'  ,
\eeq
where for each $D_i,~ i=1,...,l$, we have $D_{i} = \begin{bmatrix} 1 & - \gammaR \sigma_{i} \\ \gammaJ \sigma_{i} & 1 - (1+\tau) \gammaR\gammaJ \sigma_{i}^2 \end{bmatrix}$, which is the $2\times 2$ matrix studied in Type III matrix. 
Therefore, for each $i=1,...,l$, the eigenvalue of $D_i$ is
\[
\rho_i = \qfrac{ \Pa{2-(1+\tau) \gammaJ\gammaR \sigma_j^2} \pm  \ssqrt{ (1+\tau)^2 \gammaJ^2\gammaR^2 \sigma_j^4 - 4 \gammaJ\gammaR \sigma_j^2 } } {2} .
\]
Since $ \gammaJ\gammaR \sigma_j^2 \leq \gammaJ\gammaR \norm{L}^2 < 1$, then $\rho_i$ is complex and 
\[
\abs{\rho_i}
= \qfrac12 \sqrt{ \Pa{2-(1+\tau) \gammaJ\gammaR \sigma_j^2}^2 - \Pa{ (1+\tau)^2 \gammaJ^2\gammaR^2 \sigma_j^4 - 4 \gammaJ\gammaR \sigma_j^2 } } 
= \sqrt{ 1- \tau \gammaJ\gammaR \sigma_j^2 } < 1 .
\]
As a result, $\lim_{k\to+\infty}D_i^k = 0$, $\mPD^k$ is convergent and 
\[
\mPDinf
= 
\begin{bmatrix} Y & ~ \\ ~ & X \end{bmatrix}
\begin{bmatrix}
0_{l} &  & &   \\
& \Id_{n-l} & &   \\
& & 0_{l} &   \\
& & &  \Id_{m-l}  \\
\end{bmatrix} 
\begin{bmatrix} Y^T & ~ \\ ~ & X^T \end{bmatrix} 
\]
owing to \eqref{eq:mtx-hatM-2}, which is symmetric and positive semi-definite. 

\vgap

{\noindent}\textbf{Trajectory of Primal--Dual}~~ 
{From the above discussion, we have that
\begin{itemize}
	\item The real eigenvalues of $\mPD$ are $0$ and $1$. 
	\item For each $D_i$ in $D$ of \eqref{eq:mtxD}, (ii) of Assumption \ref{assumption-type-iii-a} is satisfied. 
\end{itemize}}
{\noindent}This means that the Assumption \ref{assumption-type-iii-a} is verified by $\mPD$, hence we obtain immediately the trajectory of Primal--Dual owing to Proposition \ref{prop:type-iii-a}. 
\qedhere
\end{proof}

\section{Proofs of Section \ref{sec:lp}}\label{sec:proof_acceleration}

For the easy of notation, let $M = \mF$, $c = c_k, C = C_k$ and $\epsilon = \epsilon_k$.

\begin{proof}[Proof of theorem \ref{thm:extrap-error}]
Since $k\in \NN$ is fixed throughout, we let $E_\ell \eqdef E_{k,\ell}$. 
We first show that for $\ell \in \NN$, there holds
\begin{equation}\label{eq:En}
E_\ell = - \sum_{j=0}^{\ell -1} M^{\ell -1-j} F_{k+j} + \sum_{j=1}^\ell  M^j E_0 C^{\ell -j} + \sum_{j=1}^\ell  M^{\ell-j} F_{k-1} C^j.
\end{equation}
We shall prove this by induction.
First note that $E_0 = V_{k-1}C - V_k$
\begin{align*}
V_{k} C 
\overset{\eqref{eq:linear_rel}}{=}\pa{M V_{k-1} + F_{k-1}} C 
&\overset{\eqref{eq:e0}}{=} MV_k + M E_0 +  F_{k-1} C\\
&\overset{\eqref{eq:linear_rel}}{=}  V_{k+1} - F_k   + M E_0 +    F_{k-1} C.
\end{align*}
Hence, $E_1 =- F_k   + M E_0 +    F_{k-1} C$ and \eqref{eq:En} is true for $\ell = 1$. Assume that \eqref{eq:En} is true up to $\ell = m$, then,
\begin{align*}
V_k C^{m+1} 
\overset{\eqref{eq:linear_rel}}{=}\pa{M V_{k-1} + F_{k-1}} C^{m+1} 
&= M V_{k-1} C^{m+1} + F_{k-1} C^{m+1}\\
&\overset{\eqref{eq:e0}}{=} M V_kC^{m}   + M E_0C^{m} + F_{k-1} C^{m+1}\\
&= M \pa{V_{k+m} + E_{m}}  + M E_0C^{m} + F_{k-1} C^{m+1}\\
&\overset{\eqref{eq:linear_rel}}{=} V_{k+m+1} - F_{k+m} +  M E_{m}  + M E_0C^{m} + F_{k-1} C^{m+1}  .
\end{align*}
Therefore, plugging in our assumption on $E_m$ yields
\begin{align*}
E_{m+1} &= - F_{k+m} +  M E_{m}  + M E_0C^{m} + F_{k-1} C^{m+1}\\
&=  - F_{k+m} + \Pa{- \msum_{j=0}^{m -1} M^{m-j} F_{k+j} + \msum_{j=1}^m  M^{j+1} E_0 C^{m -j} + \msum_{j=1}^m  M^{m+1-j} F_{k-1} C^j}  + M E_0C^{m} + F_{k-1} C^{m+1} \\
&= - \sum_{j=0}^{m} M^{m-j} F_{k+j} +  \sum_{j=1}^{m+1} M^{j} E_0 C^{m+1-j} + \sum_{j=1}^{m+1} M^{m+1-j} F_{k-1} C^j.
\end{align*}
To bound the extrapolation error, observe that
\begin{equation}\label{eq:extrap_error}
\begin{split}
 \sum_{m=1}^s  E_m 
 &=  \sum_{m=1}^s \Pa{- \msum_{j=0}^{m-1} M^{m-1-j} F_{k+j} + \msum_{j=1}^m M^j E_0 C^{m-j} + \msum_{j=1}^m M^{m-j} F_{k-1} C^j}\\
  &=-\sum_{\ell=0}^{s-1} M^\ell \sum_{j=0}^{s-1-\ell} F_{k+j}
 + \sum_{\ell=1}^{s} M^\ell E_0 \Pa{\msum_{i=0}^{s-\ell} C^i} + \sum_{\ell=0}^{s-1} M^\ell F_{k-1} \Pa{\msum_{i=1}^{s-\ell} C^i} .
\end{split}
\end{equation}
Note that if $F_{k+j} = 0$ for all $j$, then
\begin{align*}
 \sum_{m=1}^s  E_m &=  \sum_{\ell=1}^{s} M^\ell E_0 \Pa{\msum_{i=0}^{s-\ell} C^i} 
\end{align*}
and
$$
\norm{\bar z_{k,s} - \zsol} \leq \norm{z_{k+s} - \zsol} +  \sum_{\ell=1}^{s} \norm{M^\ell } \norm{ \pa{\msum_{i=0}^{s-\ell} C^i}_{(1,1)}} \epsilon .
$$
%
In the general setting where $F_{k+j}\neq 0$, to bound the first term of \eqref{eq:extrap_error},
define
$$
Z_k \eqdef \begin{bmatrix}
z_{k}|\cdots| z_{k-q+1}
\end{bmatrix}
$$
and note that $V_k= Z_k - Z_{k-1}$. So,
$$
\sum_{j=0}^m F_{k+j} = \sum_{j=0}^m V_{k+j+1}-M V_{k+j} = Z_{k+m+1}-Z_{k} - M Z_{k+m} + M Z_{k-1},
$$
and
\begin{align*}
\Pa{\msum_{\ell=0}^{s-1} M^\ell \msum_{j=0}^{s-1-\ell} F_{k+j}}_{(:,1)}
&=\Pa{\msum_{\ell=0}^{s-1} M^\ell Z_{k+ s-\ell}- M^\ell Z_{k} - M^{\ell+1} Z_{k+s-1-\ell} + M^{\ell+1} Z_{k-1}}_{(:,1)}\\
&= \Pa{Z_{s+k} - M^s Z_k +\msum_{\ell=0}^{s-1}  M^{\ell} (M Z_
{k-1}- Z_k)}_{(:,1)}\\
&= z_{k+s} - M^s z_k  +\sum_{\ell=0}^{s-1} M^\ell (M z_{k-1} - z_k) \\
&= z_{k+s} - M^s z_k  +\sum_{\ell=0}^{s-1} M^\ell \Pa{ M (z_{k-1} - \zsol) - (z_k-\zsol) } + \sum_{\ell=0}^{s-1} M^\ell(M \zsol - \zsol)\\
&= z_{k+s} - M^s z_k  +\sum_{\ell=0}^{s-1} M^\ell \Pa{ M (z_{k-1} - \zsol) - (z_k-\zsol) } + M^s \zsol - \zsol\\
&= z_{k+s}-\zsol - M^s (z_k - \zsol) + \sum_{\ell=0}^{s-1} M^\ell \Pa{ M (z_{k-1} - \zsol) - (z_k-\zsol) }  .
\end{align*}
Therefore we arrive at, 
\begin{align*}
\norm{\bar z_{k,s} - \zsol} 
&\leq  \norm{M^s \pa{z_k - \zsol} } + \norm{\msum_{\ell=0}^{s-1} M^\ell} \norm{ (M (z_{k-1}-\zsol) - (z_k-\zsol)} \\
&\qquad 
+ \sum_{\ell=1}^s \norm{M^\ell} \norm{E_0 \msum_{i=0}^{s-\ell} C^i_{(:,1)}} + \sum_{\ell=0}^{s-1} \norm{M^\ell} \norm{F_{k-1} \msum_{i=1}^{s-\ell} C^i_{(:,1)}}.
\end{align*}
In the case of $s=+\infty$, we have
\begin{align*}
\norm{\bar z_{k,\infty} - \zsol} 
&\leq  \norm{(\Id - M)^{-1}} \norm{ (M (z_{k-1}-\zsol) - (z_k-\zsol)} \\
&\qquad 
+ \sum_{\ell=1}^\infty \norm{M^\ell} \norm{E_0 (\Id - C)^{-1}_{(:,1)}} + \sum_{\ell=0}^{\infty} \norm{M^\ell} \norm{F_{k-1} ((\Id - C)^{-1} - \Id)_{(:,1)}}.
\end{align*}
We have $ \norm{E_0 (\Id - C)^{-1}_{(:,1)}} = \frac{\epsilon}{1-\sum_i c_i}$ where
$$
\epsilon \eqdef \min_{c\in \RR^{q}} \norm{\msum_{j=1}^{q} c_j v_{k-j} - v_{k}}.
$$
Letting $b_i\eqdef \sum_{\ell=i}^{q}c_\ell$ for $i\geq 2$ and $b_1 = 1$,
\[
\begin{aligned}
F_{k-1} ((\Id - C)^{-1} - \Id)_{(:,1)} 
= \sfrac{1}{1-\sum_\ell c_\ell} \sum_{i=1}^{q} b_i f_{k-i} - f_{k-1} 
= \sfrac{1}{1-\sum_\ell c_\ell}\sum_{i=1}^{q}  \Pa{\msum_{\ell=i}^{q}c_\ell}  f_{k-i}  .
\end{aligned}\qedhere
\]
\end{proof}

\begin{proof}[Proof of Theorem \ref{thm:acc-guarantee}]
The first result of Theorem \ref{thm:acc-guarantee} is simply a consequence of Theorem \ref{thm:extrap-error}. 
%
%
%
%
%
%
To control the 
coefficients fitting error $\epsilon_k$, 
we follow closely the arguments of \cite[Section 6.7]{sidi2017vector}, since this amounts to understanding the behavior of the coefficients $c_k$, which are precisely the MPE coefficients. 
Recall our assumption that $M$ is diagonalizable, so $M = U^\top \Sigma U$ where $U$ is an orthogonal matrix and $\Sigma$ is a diagonal matrix with the eigenvalues of $M$ as its diagonal. Then, letting $u_k \eqdef U v_k$,
\begin{align*}
\epsilon_k &= \min_{c\in \RR^{q}} \norm{\msum_{i=1}^{q} c_i v_{k-i} - v_{k}}
= \min_{c\in \RR^{q}} \norm{\msum_{i=1}^{q} c_i \Sigma^{k-i} u_0 - \Sigma^{k} u_0}
= \min_{g\in\calP_q}\norm{ \Sigma^{k-q} g(\Sigma) u_0} 
\leq \norm{u_0} \min_{g\in \Pp_q}\max_{z\in \lambda(M)} \abs{z}^{k-q} \abs{g(z)} 
\end{align*}
where $\calP_q$ is the set of monic polynomials of degree $q$ and $\lambda(M)$ is the spectrum of $M$. Choosing $g = \prod_{j=1}^{q} (z-\lambda_j)$, we have $g(\lambda_j) = 0$ for $j=1,\ldots,q$, so
\begin{equation}
\epsilon_k \leq \norm{u_0} \abs{\lambda_{q+1}}^{k-q} \max_{\ell>q}\prod_{j=1}^{q} \abs{\lambda_j-\lambda_\ell}.
\end{equation}
The claim that $\rho(C_k)<1$ holds since the eigenvalues of $C$ are precisely the roots of the polynomial $Q(z) = z^{k-1} - \sum_{i=1}^{k-1} c_j z^{k-1-i}$, and from \cite{sidi2017vector}, if $\abs{\lambda_{q}}>\abs{\lambda_{q+1}}$, then $Q$ has precisely $q$ roots $r_1,\ldots, r_q$ satisfying $r_j = \lambda_j + \Oo(\abs{\lambda_{q+1}/\lambda_j}^k)$. So, $\abs{r_j}<1$ for all $k$ sufficiently large.
To prove the non-asymptotic bounds on $\epsilon_k$, first
observe that $z_{k+1} - z_{k} = M(z_k - z_{k-1})$ implies $z_{k+1} - \zsol = M(z_k - z_{*})$ and $z_{k+1} - z_k = (M-\Id)(z_k - \zsol)$. So, letting $\gamma_i = -c_{k,i}/(1-\sum_i c_{k,i})$ for $i=1,\ldots, q$ and $\gamma_0 = 1/(1-\sum_i c_{k,i})$, we have
\begin{equation}\label{eq:coefficient_vec}
\mfrac{1}{1-\ssum_i c_{k,i}} \Pa{ v_{k} -\msum_{i=1}^{q} c_{k,i} v_{k-i} } 
= \msum_{i=0}^{q} \gamma_i v_{k-i} 
= (M-\Id) \msum_{i=0}^{q} \gamma_i( z_{k-i-1} - \zsol).
\end{equation}
Now, $y\eqdef \sum_{i=0}^{q} \gamma_i z_{k-i-1} $ is precisely the MPE update and norm bounds on this are presented in \cite{sidi2017vector}. For completeness, we reproduce their arguments here: Let $A\eqdef \Id - M$, by our assumption of $\lambda(M) \subset (-1,1)$, we have that $A$ is positive definite. Then,
\begin{align*}
\norm{A^{1/2}(y - \zsol)}^2 
&= \dotp{A (y-\zsol)}{(y-\zsol)} 
= - \dotp{\ssum_{i=0}^{q} \gamma_i v_{k-i} }{(y-\zsol) + w}
\end{align*}
where $w = \sum_{j=1}^{q} a_j v_{k-j}$ with $a\in \RR^{q}$ being arbitrary, since by definition of $\gamma$, $\dotp{\sum_{i=0}^{q} \gamma_i v_{k-i}}{v_\ell}=0$ for all $\ell=k-q,\ldots, k-1$. We can write
$$
w = \sum_{j=1}^{q} a_j (M-\Id)(z_{k-j-1}-\zsol) = \sum_{j=1}^{q} a_j (M-\Id)M^{k-j-1}(z_{0}-\zsol) = f(M) (z_{k-q-1} - \zsol)
$$
where $f(z) =  (z-1) \sum_{j=1}^{q} a_j z^{q-j}$, and we can write
$$
y - \zsol = \sum_{i=0}^{q} \gamma_i M^{k-i-1} (z_0- \zsol) = g(M)(z_{k-q-1}-\zsol)
$$
where $g(z) = \sum_{i=0}^{q} \gamma_i z^{q-i}$. Therefore, $f(z) + g(z) =  h(z)$, where $h$ is a polynomial of degree $q$ such that $h(1) = 1$. Moreover, since the coefficients $a_j$ are arbitrary, $h$ can be considered as an arbitrary element of $\tilde \Pp_q$, the set of all polynomials of degree $q$ such that $h(1) = 1$. Therefore
\begin{align*}
\norm{A^{1/2}(y - \zsol)}^2 &\leq \norm{A^{1/2}(y - \zsol)} \min_{h\in \tilde \Pp_q}\norm{ h(M) (z_{k-q-1} - \zsol)} \\
&\leq \norm{A^{-1/2}(y - \zsol)} \min_{h\in \tilde \Pp_q}\max_{t\in \lambda(M)}\abs{ h(t)} \norm{z_{k-q-1} - \zsol}  .
\end{align*}
In particular, combining this with \eqref{eq:coefficient_vec}, we have
$$
\mfrac{\epsilon_k}{\abs{1-\ssum_i c_{k,i}}}\leq \norm{z_{k-q-1}-\zsol} \norm{(\Id-M)^{1/2}}  \min_{h\in \tilde \Pp_q}\max_{t\in \lambda(M)}\abs{ h(t)} 
$$
Finally, in our case where $\lambda(M) =[\alpha,\beta]$ with $1>\beta>\alpha >-1$, it is well known that $\min_{h\in \tilde \Pp_q}\max_{t\in \lambda(M)}\abs{ h(t)}$ has an explicit expression (see, for example, \cite{borwein2003monic} or \cite[Section 7.3.1]{sidi2017vector}):
$$
\min_{h\in \tilde \Pp_q}\max_{z\in \lambda(M)} \abs{h(z)} \leq \max_{z\in \lambda(M)} \abs{h_*(z)},
$$
where $h_*(z) \eqdef \frac{T_q\pa{\frac{2z-\alpha-\beta}{\beta-\alpha}}}{T_q\pa{\frac{2-\alpha-\beta}{\beta-\alpha}}}$ where $T_q(x)$ is the $q^{th}$ Chebyshev polynomial and it is well known that
\begin{equation}
\min_{h\in \tilde \Pp_q}\max_{z\in [\alpha,\beta]} \abs{h(z)} 
\leq 2 \BPa{\sfrac{\sqrt{\eta}-1}{\sqrt{\eta}+1}}^q
\end{equation}
where $\eta = \frac{1-\alpha}{1-\beta}$. 
\qedhere
\end{proof}

\end{small}

  

\end{document}